\documentclass[letter,11pt]{amsart}
\usepackage[latin1]{inputenc}   
\usepackage{amssymb,amsmath,amsthm,mathrsfs,mathdots}
\usepackage{graphicx,mathpazo}
\usepackage[all]{xy}
\usepackage[usenames,dvipsnames]{color}
\usepackage{enumitem}

\usepackage{soul}

\usepackage{youngtab,ytableau}
\ytableausetup{centertableaux,boxsize=0.25em}

\usepackage[colorlinks=true,linkcolor=Cerulean,pagebackref,citecolor=Mahogany]{hyperref}

 \usepackage{tikz}
\usetikzlibrary{arrows,decorations.pathmorphing,decorations.pathreplacing,positioning,shapes.geometric,shapes.misc,decorations.markings,decorations.fractals,calc,patterns}

\tikzset{>=stealth',
     cvertex/.style={circle,draw=black,inner sep=1pt,outer sep=3pt},
     vertex/.style={circle,fill=black,inner sep=1pt,outer sep=3pt},
     star/.style={circle,fill=yellow,inner sep=0.75pt,outer sep=0.75pt},
     tvertex/.style={inner sep=1pt,font=\criptsize},
     gap/.style={inner sep=0.5pt,fill=white}}

\renewcommand{\AA}{\ensuremath{\mathbb{A}}}

\newcommand{\PP}{\ensuremath{\mathbb{P}}}

\newcommand{\ZZ}{\ensuremath{\mathbb{Z}}}
\newcommand{\CC}{\ensuremath{\mathbb{C}}}
\newcommand{\KK}{\ensuremath{K}} 

\newcommand{\QQ}{\ensuremath{\mathbb{Q}}}
\newcommand{\TT}{\ensuremath{\mathbb{T}}}

\DeclareMathOperator{\Sing}{Sing}
\DeclareMathOperator{\Spec}{Spec}

\newcommand{\gen}{{\rm gen}}
\newcommand{\prin}{{\rm prin}}
\newcommand{\univ}{{\rm univ}}

\newcommand{\cA}{\mathcal{A}}
\newcommand{\cF}{\mathcal{F}}
\newcommand{\cL}{\mathcal{L}}
\newcommand{\cQ}{\mathcal{Q}}

\newcommand{\cX}{\mathcal{X}}

\renewcommand{\c}{{\boldsymbol{c}}} 
\newcommand{\p}{\boldsymbol{p}} 
\newcommand{\s}{\boldsymbol{s}} 
\renewcommand{\t}{\boldsymbol{t}} 
\newcommand{\x}{\boldsymbol{x}} 
\newcommand{\y}{\boldsymbol{y}}
\newcommand{\z}{\boldsymbol{z}}

\newcommand{\Jac}{\operatorname{Jac}}
\newcommand{\car}{\operatorname{char}}

\newcommand{\bs}[1]{{#1}}

\newcommand{\bsOUT}[1]{} 

\newcounter{CountAlpha}

\theoremstyle{theorem}

\newtheorem{MainThm}[CountAlpha]{Theorem}  
\newtheorem{MainProp}[CountAlpha]{Proposition}

\newtheorem{Thm}{Theorem}[subsection]         
\newtheorem{lemma}[Thm]{Lemma}
\newtheorem{cor}[Thm]{Corollary}

\newtheorem{prop}[Thm]{Proposition}

\theoremstyle{definition}
\newtheorem{defi}[Thm]{Definition} 

\newtheorem{example}[Thm]{Example}
\newtheorem{Bem}[Thm]{Remark}
\newtheorem{Obs}[Thm]{Observation}
\newtheorem{Not}[Thm]{Notation}

\newcommand{\pmo}{{\ensuremath{\pm 1}}}

\newcommand{\rf}{\kappa(\eta)}

\numberwithin{equation}{subsection}

\title[Classification of singularities of cluster algebras of finite type II]{Classification of singularities of cluster algebras of finite type II: coefficients}

\author{Ang\'elica Benito}
\address{Departamento de Did\'acticas Espec\'ificas,
	Facultad de Formaci\'on de Profesorado, 
	Universidad Aut\'onoma de Madrid,
	28049 Madrid, Spain}
\email{angelica.benito@uam.es} 

\author{Eleonore Faber}
\address{
Institut f\"ur Mathematik und Wissenschaftliches Rechnen, Universit{\"a}t Graz, Heinrichstr. 36, A-8010 Graz, Austria and School of Mathematics, University of Leeds, LS2 9JT Leeds, UK
}
\email{eleonore.faber@uni-graz.at}

\author{Hussein Mourtada}
\address{Universit\'e Paris Cit\'e and Sorbonne Universit\'e, CNRS, IMJ-PRG, F-75013 Paris, France}
\email{hussein.mourtada@imj-prg.fr}

\author{Bernd Schober}
\address{
Carl von Ossietzky Universit\"at Oldenburg,
Institut f\"ur Mathematik,
Ammerl\"ander Heerstra{\ss}e 114 - 118,
26129 Oldenburg (Oldb), Germany
}
	\curraddr{None. (Hamburg, Germany).}
\email{schober.math@gmail.com}

\date{\today}

\makeatletter
\@namedef{subjclassname@2020}{\textup{2020} Mathematics Subject Classification}
\makeatother



\thanks{A.B.~is partially supported by Grant PID2022-138916NB-I00 funded by MCIN/AEI/ 10.13039/501100011033 and by ERDF A way of making Europe PID2022-138916NB-I00. E.F.~was supported by the Engineering and Physical Sciences Research Council [grant number EP/W007509/1]. H.M is partially supported by the ANR SINTROP (ANR-22-CE40-0014).}

\begin{document}

\maketitle

\begin{abstract}
	We provide a complete classification of the singularities of cluster algebras of finite cluster type. 
	This extends our previous work about the case of trivial coefficients. 
	Additionally, we classify the singularities of cluster algebras for rank two.
\end{abstract}

\section{Introduction}
 
\subsection{Motivation}
Cluster algebras were introduced by Fomin and Zelevinsky \cite{FZ2002} in the context of total positivity and Lie theory, and quickly developed connections to various different disciplines including combinatorics, representation theory, algebraic geometry, group theory, dynamical systems, mathematical physics and symplectic geometry. For a survey and extensive bibliography we refer to \cite{WilliamsSurvey, KellerDerivedSurvey}.
\\
Our motivation is to study how the singularities of a variety defined by a cluster algebra are reflected in the combinatorial data of the algebra and vice versa. It is natural to begin with the algebras of finite cluster type\footnote{Sometimes these algebras are called \emph{cluster algebras of finite type}, see \cite{BFMS}. In order to avoid confusion with infinitely generated algebras we have adopted the new notation.},
which
can be classified in terms of Dynkin diagrams, see \cite{FZ2003II}.
\\
In \cite{BFMS}, we established a classification of the singularities of cluster algebras of finite cluster type with \emph{trivial coefficients} (see Section~\ref{Subsec:IntroPresent} for an explanation of this notion).
In the current paper we study how coefficients affect the singularities and extend our classification to this general case.
\\
Other results on the singularities of cluster algebras with a slightly different flavor are \cite{BMRS2015,MRZ2018},
see also \cite[Introduction]{BFMS}, where the \bs{content of these papers is} summarized. 

\subsection{Presentations of acyclic cluster algebras}
\label{Subsec:IntroPresent} 
Let us dive a bit deeper into the construction of a cluster algebra. 
\bsOUT{There are different perspectives 
(e.g., on how to encode the so called coefficients).
We recall them in Section~\ref{Subsec:GeomType_Perpectives}.}
For the moment, we choose simplicity over precision in the presentation. 

Fix a ground field $ \KK $.
The construction of a cluster algebra over $ \KK $ begins with a labeled seed $ \Sigma = (\widetilde \x, \widetilde B) $,
where $ \widetilde \x = (x_1, \ldots, x_m) $
are called the {\em cluster variables}
and $ \widetilde B = (b_{ij})_{i\in \{1, \ldots, m\}, j \in \{ 1, \ldots, n \}} $ is an $ m \times n $ integer matrix fulfilling certain technical conditions, which we will specify later in Section~\ref{Subsec:GeomType_Perpectives},
for $ n, m \in \ZZ_+ $ with $ n \leq m $. 
The matrix $ \widetilde B $ is called the {\em extended exchange matrix}.
We denote by $ \cA(\Sigma ) $ the corresponding cluster algebra.
Here, $ \x = (x_1, \ldots, x_n) $ is the given set of mutable variables of $ \cA(\Sigma) $,
while $ (x_{n+1}, \ldots, x_m) $ are connected to the coefficients (sometimes we call them frozen variables).
We follow the {\em convention} to assume that {\em coefficients are invertible}.
\\
For any $ k \in \{ 1, \ldots, n \} $,
we may mutate in direction $ k $.
In this process, the cluster variable $ x_k $ is substituted by a new cluster variable $ x_k' $ which fulfills the {\em exchange relation}
\begin{equation}
	\label{eq:intro_exchange}
		x_k x_k' = \prod_{\substack{i=1\\[3pt]b_{ik} > 0}}^m x_i^{b_{ik}} + \prod_{\substack{i=1\\[3pt]b_{ik} < 0}}^m x_i^{-b_{ik}}
		\ . 
\end{equation} 
There is also a modification of the matrix $ \widetilde B $, 
for details see Section~\ref{Subsec:GeomType_Perpectives},
but there exist simple presentations for the cluster algebras that we consider not requiring this extra technical step, see below. 
\\
By iterated mutation in every direction $ k \in \{ 1, \ldots, n \} $, we collect more and more cluster variables together with their corresponding exchange relations. 
In general, this is not a finite process, i.e., there may be infinitely many pairwise different cluster variables. 
Eventually, $ \cA(\Sigma) $ is generated by all cluster variables. 

At first sight, it seems hard to do explicit computations with cluster algebras as there might be infinitely many cluster variables and exchange relations 
(and, in fact, even if there are finitely many, their number can be rather big).
Nonetheless, we can make use of a result of Berenstein, Fomin, and Zelevinsky \cite{BFZ2005} 
(recalled in Theorem~\ref{Thm:Presentation})
which provides a simple presentation of the cluster algebra $ \cA(\Sigma) $ if $ \Sigma $ is acyclic (Definition~\ref{Def:acyclic}).
More precisely, with this additional hypothesis, 
$ \cA(\Sigma) $ is isomorphic to the algebra that we obtain after mutating $ \Sigma $ in each direction once,
i.e., on the geometric side, we have 
(if $ \Sigma $ is acyclic)
\begin{equation}
	\label{eq:intro_presentation}
	\Spec(\cA (\Sigma))
	\cong 
	V (x_k x_k' - \prod_{\substack{i=1\\[3pt]b_{ik} > 0}}^m x_i^{b_{ik}} - \prod_{\substack{i=1\\[3pt]b_{ik} < 0}}^m x_i^{-b_{ik}} \mid k \in \{ 1, \ldots, n \} ) \ , 
\end{equation}
where $ V (f_1, \ldots, f_n) $ denotes the variety associated to the ideal generated by $ f_1, \ldots, f_n $ in $\KK[x_{n+1}^{\pm 1},\ldots,x_{m}^{\pm 1}][x_1,\ldots,x_n, x_1', \ldots, x_n'] $.

In the present article, all labeled seeds that we consider in the context of singularity theory will be acyclic. 
Hence, it is sufficient for us to work with this special presentation. 
As a benefit, we reduce some of the complexity and additionally, we may even approach cluster algebras that are not of finite cluster type. 

We may rewrite the products appearing in the exchange relation \eqref{eq:intro_exchange} by \bsOUT{distinguishing
	the part coming from the chosen $ \x $ and the part determined by the coefficient cluster variables}
\bs{separating each product into a factor in $ \x = (x_1, \ldots, x_n) $ and one in} $ (x_{n+1}, \ldots , x_m) $. 
By introducing the abbreviations
\begin{equation}
	\label{eq:intro_coeff_abbrev}
	s_k := \prod_{\substack{i=n+1\\[3pt]b_{ik} > 0}}^m x_i^{b_{ik}}
	\ \ \
	\mbox{ and } 
	\ \ \ 
	t_k := \prod_{\substack{i=n+1\\[3pt]b_{ik} < 0}}^m x_i^{-b_{ik}}
\end{equation}
for the coefficients, the exchange relations \eqref{eq:intro_exchange} become 
\begin{equation}
	\label{eq:intro_exchange_abbrev}
	x_k x_k' - s_k \prod_{\substack{i=1\\[3pt]b_{ik} > 0}}^n x_i^{b_{ik}} - t_k \prod_{\substack{i=1\\[3pt]b_{ik} < 0}}^n x_i^{-b_{ik}}
	= 0 
	\ , 
	\ \ \
	\mbox{for } 
	k \in \{ 1, \ldots, n \}
	\ . 
\end{equation}
We denote by $  B $  the upper $ n \times n $ sub-matrix of $ \widetilde B $. 
Notice that the exponents appearing in the products in \eqref{eq:intro_exchange_abbrev} are the absolute values of the entries of the $ k $-th column in $ B $. 
A cluster algebra has {\em trivial coefficients} if $ s_k = t_k = 1 $ 
for all $ k \in \{ 1, \ldots , n \} $.
Forgetting that $ \s := ( s_1, \ldots , s_n),  \t := (t_1, \ldots , t_n) $ are abbreviations, 
we observe that \eqref{eq:intro_exchange_abbrev} are precisely the exchange relations of the labeled seed
\[
	\Sigma^\gen := (\widetilde \x^\gen, \widetilde B^\gen ) 
	\ , 
\]
where $ \widetilde x^\gen := (\x, \s, \t) $ and $ \widetilde B^\gen $ is the matrix that we obtain by extending $ B $ by the $ n \times n $
identity matrix $ I_n $ as well as its additive inverse $  -I_n $. 
Due to the connection \eqref{eq:intro_coeff_abbrev} between
$ \cA(\Sigma) $ and $ \cA(\Sigma^\gen) $,
we call $ \cA(\Sigma^\gen) $ the {\em cluster algebra with generic coefficients associated to $ B $}.
\\
In Section~\ref{Subsec:gen_coeff}, we elaborate more on the reason for choosing the name ``generic coefficients".
Further, we discuss the differences to the cluster algebra with universal coefficients in details. 
A notable fact is that for the universal coefficients one has to determine the so called $ g $-vectors of {\em all} cluster variables and thus this quickly becomes hard to handle as well as \bs{requiring} to deal with cluster algebras that have only finitely many cluster
variables. 
In contrast to this, cluster algebras with the generic coefficients are more feasible.

Since we assume coefficients to be invertible, we may work with $ s_k^{-1} $ or $ t_k^{-1} $. 
This allows us to draw a connection between $ \cA( \Sigma^\gen) $ and the cluster algebra with principal coefficients associated to $ (\x, B) $. 
A labeled seed $ \Sigma^\prin $ for the latter is determined by extending $ B $ by $ I_n $. 
Thus, we have
\[
	\Spec(\cA (\Sigma^\prin))
	\cong 
	V (x_k y_k - c_k \prod_{\substack{i=1\\[3pt]b_{ik} > 0}}^n x_i^{b_{ik}} - \prod_{\substack{i=1\\[3pt]b_{ik} < 0}}^n x_i^{-b_{ik}} \mid k \in \{ 1, \ldots, n \} ) \ , 
\]
where $ \c = (c_1,\ldots , c_n) $ denotes the coefficients
and the ambient space for the variety on the right hand side is the obvious one. 
Using that $ s_k, t_k $ are invertible, it is not hard to show that $ \Spec(\cA(\Sigma^\gen)) $ is isomorphic to a trivial family over $ (\KK^\times)^n $, 
where each fiber is isomorphic to $\Spec( \cA(\Sigma^\prin)) $ (Lemma~\ref{Lem:gen=prin}). 
Hence, any investigation on the singularity theory of  $ \Spec(\cA(\Sigma^\gen)) $ can be reduced to studying $\Spec( \cA(\Sigma^\prin)) $.

The goal of the present article is to extend the classification of the singularities of cluster algebras of finite cluster type to the general case with possibly non-trivial coefficients. 
Beyond that, we also take a look into the non-finite cluster type case by classifying the singularities of cluster algebras of rank two. 
As explained above, we reduce this problem to the investigation
of the cluster algebra with generic, resp.~principal, coefficients. Given $ \Sigma = ( \x, B) $ (with $ B $ of size $ n \times n $), we often write $ \cA^\gen_{\s,\t} (\Sigma) $ (resp.~$ \cA^\prin_{\c} (\Sigma) $) for the cluster algebra with
generic (resp.~principal) coefficients associated to $ \Sigma $ in order to emphasize the origins in $ \Sigma $ as well as the role of $ (\s, \t) $ (resp.~$ \c $) as coefficients.

\subsection{Notions from singularity theory}
In order to make the formulation of our main theorems meaningful,
we recall some background on simple singularities.
For a complete introduction, we refer to \cite{GreuelKroening}.

Let $ \KK $ be a field.
Let $ X \subseteq \AA_\KK^N $ be a $ 2 $-dimensional variety
with an isolated singularity at a closed point $ v \in X $, for some $ N \geq 3 $.
We say that $ X $ is locally (at $ v $) isomorphic to an isolated {\em hypersurface singularity of type $ A_k $}
if there is an isomorphism between the completion of the local ring of $ X $ at $ v $ and $ \KK[[x,y,z]]/ \langle xy + z^k \rangle $. 
If this is the case, then the singularity of $ X $ at $ v $
is resolved by $ k $ consecutive point blowups, starting with $ v $ as the first center. 
\\
Fix $ n \geq 3 $
and let $ V \subset \AA_\KK^N $ be an $ n $-dimensional variety with an isolated singularity at a closed point $ v \in V $ (where $ N > n $).
We say that $ V $ is locally (at $ v $) isomorphic to an isolated {\em hypersurface singularity of type $ A_1 $}
if the completion of the local ring of $ V $ at $ v $ is isomorphic to  $ \KK [[z_0, \ldots, z_n]]/ \langle g \rangle $,
where the element $ g $ is of the form
\[
	g = \begin{cases}
		z_0^2 + z_1 z_2 + \cdots + z_{n-1} z_n
		& \mbox{if $ n $ is even,}
		\\[3pt]
		z_0 z_1 + z_2 z_3 + \cdots + z_{n-1} z_n
		& 
		\mbox{if $ n $ is odd.}
	\end{cases}
\]
Analogous to before, such a singularity is resolved by blowing up the isolated singular point $ v $. 

Sometimes we will say that a given $ n $-dimensional singularity $ X \subseteq \AA_\KK^N $ is locally isomorphic to a {\em cylinder over a hypersurface singularity of type $ A_1 $}.
By the latter we mean that there exist an open subset $ U $ as well as an isolated hypersurface singularity $ H = V ( g(z_0, \ldots, z_m )) \subset  \AA_\KK^{m+1} $ of type $ A_1 $, for some $ m < n $,
such that
$ X $ is 
locally in $ U $
isomorphic to 
$ \Spec( \KK[z_0, \ldots, z_m,\ldots, z_n]/ \langle g(z_0, \ldots, z_m) \rangle ) $.

Since we admit ground fields that are not necessarily perfect,
we have to be a bit more careful when determining the singular locus. 
In general, one has to take the notion of $ p $-bases into account and apply Zariski's regularity criterion \cite[Theorem 11, p.~39]{Zar47}.
We overcome this in the given cluster algebras setting by first considering only the vanishing of certain minors which do not require to choose a $ p $-basis of the ground field
and afterwards we discuss the impact of the remaining minors. 
\bsOUT{In other words, we determine a relative singular locus in the first step.}

\subsection{Viewpoint on cluster algebras as families}
\label{Subsec:Intro_family}
In~\cite{BFMN_degen}, the authors advertise the perspective to make a clear distinction between so called frozen variables and coefficients in order to study cluster algebras as families with varying \bsOUT{parameters as the} coefficients.
Using this, they construct toric degenerations of cluster varieties establishing new interesting connections.
We adapt this viewpoint for the varieties determined by cluster algebras with generic $ \Spec(\cA^\gen_{\s,\t} (\Sigma)) $ 
(resp.~principal $ \Spec (\cA^\prin_{\c} (\Sigma)) $) coefficients by considering them as families with \bsOUT{parameter} \bs{coefficient} space $ \Spec ( \KK_{\s,\t} ) := \Spec ( \KK[\s^{\pm 1},\t^{\pm 1}] ) $ (resp.~$ \Spec(\KK_\c) := \Spec(\KK [\c^{\pm 1}] ) $). 
Therefore, our precise goal is to understand the singularity type of the fibers of the family  
	\[
	\phi \colon \Spec ( \cA^\prin_{\c} (\Sigma) ) \longrightarrow \Spec (\KK [\c^{\pm 1}]) =: S 
	\]
and to study how the singularity type of the fibers of this family varies.
If one does not necessarily assume that coefficients are invertible, the relative situation would be above the affine space 
$ \Spec (\KK [\c ]) $,
for example, see \cite{BFMN_degen,BMN_Families,INT_complexes}.  \bsOUT{In Example~\ref{Ex:non-invertible}, we briefly address this within the singularities context.}

\begin{Not}
	For a closed point $ \eta \in S $, 
	we denote the fiber above $ \eta $ by  
	\[
		 \Spec ( \cA^\prin_{\c} (\Sigma) )_\eta := \phi^{-1} (\eta) \ . 
	\]
	\bs{Notice that this is a variety over the residue field $ \kappa(\eta) = \mathcal{O}_{S,\eta}/ \mathfrak{m}_\eta $ at $ \eta $,
		where $\mathcal{O}_{S,\eta} $ is the local ring of $ S $ at $ \eta $
		and $ \mathfrak{m}_\eta $ is its unique maximal ideal .}
\end{Not}
 
Let us illustrate the perspective as a family \bs{by} an example:

\begin{example}[Finite cluster type $ A_3 $, cf.~Section~\ref{Subsec:An}]
	\label{Ex:Intro_A3}
	The cluster algebra with principal coefficients of finite cluster type $ A_3 $ is given by (cf.~\eqref{eq:An_exchange})
	\[
		\cA^\prin_\c (A_3) := 
		\KK_\c [\x,\y] 
		/
		\langle 
		\,
		x_1 y_1 - c_1 - x_2, 
		\, 
		x_2 y_2 - c_2 x_1 - x_3 ,
		\,
		x_3 y_3 - c_3 x_2 - 1 
		\, 
		\rangle 
		\ ,
	\] 
	where we abbreviate $ (\x,\y) := (x_1, x_2, x_3, y_1, y_2, y_3) $. 
	Using that $ c_1, c_2, c_3 $ are invertible,
	we can deduce the following isomorphism 
	$ \cA^\prin_\c (A_3)
	\cong 
	\KK_\c [z_1, \ldots, z_4] 
	/
	\langle 
	f
	\rangle $
	(cf.~Proposition~\ref{Prop:An_hypersurface}),
	where 
	\[
		f := z_1 z_2 z_3 z_4 - z_3 z_4 - z_1 z_4  -  z_1 z_2  + 1 - c_1^{-1} c_3^{-1} 
		\ . 
	\]
	The vanishing of the partial derivatives with respect to $ \z = (z_1, \ldots, z_4 ) $ 
	leads to the conclusion that 
	the fiber
	$ \Spec ( \cA^\prin_\c (A_3)_\eta ) $
	has an isolated singularity at the origin if and only if 
	\bsOUT{$ \eta_1^{-1} \eta_3^{-1} = 1 $}
	\bs{the image of $ 1 - c_1^{-1} c_3^{-1} $ in residue field $ \kappa(\eta) $ is zero}, for \bsOUT{$ \eta = (\eta_1, \eta_2, \eta_3) $} $ \bs{\eta} \in S  := \Spec (\KK_\c) $.
	Since the coefficients are integers, 
	no $ p $-basis has to be taken into account (if $ \KK $ is non-perfect of characteristic $ p > 0 $).
	\\
	In the singular case, it can be seen that we have a hypersurface singularity of type $ A_1 $ (after a suitable local coordinate change -- for details, see the proof of Theorem~\ref{Thm:A_prin}).
	Hence, considering the family defined by the fibers of	
	\[
		\phi \colon \Spec (\cA^\prin_\c (A_3)) \longrightarrow 
		S
		\ ,
	\]  
	we have the stratification 
	$ S = S_1 \sqcup S_2 $
	of the base $ S $,
	where $ S_1 := V(c_1 c_3 - 1) \subseteq S $ and $ S_2 := S \setminus S_1 $,
	such that the fibers over all points in $ S_1 $ are isomorphic 
	to a hypersurface singularity of type $ A_1 $,
	while for every point in $ S_2 $, the corresponding fiber is regular.  
	In particular, the singularity type of the fibers of the family given by $\phi$ is
fixed along each stratum $ S_i $. 
\end{example}

\subsection{Classification theorems for finite cluster type} 
In order to formulate our main results on cluster algebras of finite cluster type,
we provide an ad hoc definition of continuant polynomials and introduce a useful abbreviation.

\begin{defi}
	\label{Def:ContPoly_lamda_n}
	Let $ n \in \ZZ_{\geq 0} $. 
	\begin{enumerate}
		\item 
		We define the {\em continuant polynomial} 
		$ P_n (y_1, \ldots, y_n) \in \ZZ[y_1, \ldots, y_n] $ by 
		$ P_0 := 1 $, $ P_1 (y_1) := y_1 $, 
		and, for $ n \geq 2 $, through the recursion 
		\[  
			P_n (y_1, \ldots, y_n ) := y_1 P_{n-1} (y_2, \ldots, y_n) - P_{n-2} (y_3, \ldots, y_n) \ .  
		\]
	
		\item 
		For $ \c = (c_1, \ldots, c_n )$ invertible, 
		we define the term $ \lambda_n(\c) $ via
		\[	
		\lambda_{2s}(\c)
		:=			
		\displaystyle 
		\prod_{\alpha = 1}^{s} c_{2 \alpha}^{-1} 
		\ \ \
		\mbox{ and  }  
		\ \ \
		\lambda_{2s+1}(\c)
		:=	
		\displaystyle 
		\prod_{\alpha = 1}^{s+1} c_{2 \alpha -1 }^{-1}
		, 
		\ \ \ 
		\mbox{\bs{for $ s \in \ZZ_{\geq 0} $}}
		\ . 
		\] 
		For $ k < n $, we sometimes use the notation $ \lambda_k (\c) = \lambda_k (c_1, \ldots, c_n ) := \lambda_k (c_1, \ldots, c_k) $.    
	\end{enumerate}
\end{defi} 

In fact, the coefficients of $ P_n (y_1, \ldots, y_n ) $ are contained in $ \{-1,0,1\} $. 
In Section~\ref{Subsec:Cont}, we provide more background on continuant polynomials.

Since cluster algebras of finite cluster type can be classified by the Dynkin diagrams $ A_{n_1} , B_{n_2} , C_{n_3} , D_{n_4} , E_6, E_7, E_8, F_4, G_2 $, 
where $ n_\ell \geq \ell $ for $ \ell \in \{ 1, 2, 3, 4 \} $,
we formulate the theorems for each case separately. 
Furthermore, we use the notation $ \cA_\c^\prin ( \mathcal X) $ for the cluster algebra with principal coefficients of finite cluster type $ \mathcal X $.
\\
The key step to simplify the computation of the singular locus is to deduce a new presentation of the cluster algebra such that we obtain only a small number of relations. 

When reading the statements, we remind the reader 
of the notation $ \KK_\c = \KK [\c^{\pm 1}] $.
Furthermore, 
	the base of the respective families is $ S := \Spec (\KK_\c) $ in each case (of course, with varying $ \c $).
We point out that for every stratification appearing in the following theorems, the singularity type of the fibers of the given family is
fixed along each stratum, as in Example~\ref{Ex:Intro_A3}.
\\
\bs{For a closed point $ \eta \in S $ and $ \chi \in \mathcal{O}_S = \KK_\c $ 
	we denote by $ \chi (\eta) $ the image of $ \chi $ in the residue field $ \kappa(\eta) $ at $ \eta $.
For example, we will use the expression $ \lambda_n (\eta) $ for the image of $ \lambda_n \in \mathcal{O}_S $ (Definition~\ref{Def:ContPoly_lamda_n}(2)) in $ \kappa(\eta) $,
or for the function $ \chi \in \mathcal{O}_S $ given by $ \chi(\c)= c_1 $ we denote by $ \eta_1 $ the image in $ \kappa(\eta) $.}

\begin{MainThm}[{Proposition~\ref{Prop:An_hypersurface} and Theorem~\ref{Thm:A_prin}}]
	\label{Thm:A}
	\bs{Let $ n \geq 2 $.} 
	There is an isomorphism
	\[ 
		\cA^\prin_\c(A_n) \cong \KK_\c[z_1, \ldots, z_{n+1}]/ \langle P_{n+1} (z_1, \ldots, z_{n+1}) - \lambda_n(\c) \rangle \ . 
	\] 
	Consider the family defined by 
	$ \phi \colon \Spec ( \cA^\prin_{\c}(A_n)) \to S $.
	For a closed point $ \eta \in S $, the fiber 
	$ \Spec ( \cA^\prin_{\c}(A_n))_\eta $ is singular if and only if
	\[
	n = 2m - 1 \ \ \mbox{ and } \ \  
	\lambda_{2m-1} (\eta) = (-1)^{m} \ .
	\]  
	In the singular case, 
	$ \Spec(\cA^\prin_{\c}(A_n))_\eta  $ is isomorphic to an isolated hypersurface singularity of type $ A_1 $.
	\\
	\underline{In other words}, we have:
	\begin{enumerate}
		\item[(i)]
		If
		 $ n = 2m-1 $ is odd,
		there is a stratification
		$ S = S_1 \sqcup S_2 $ of the base $ S $, 
		where 
		\[ 
			S_1 := V(\lambda_{2m-1} (\c) - (-1)^{m}) \subset S 
		\ \ \mbox{ and } \ \  
			S_2 := S \setminus S_1  \ , 
		\] 
		such that 
		the fibers above every point of $ S_2 $ are regular and 
		those above the points of $ S_1 $ are isomorphic to a \bs{hypersurface} singularity of type $ A_1 $.

		\item[(ii)]
		If $ n $ is even, all fibers of $ \phi $ are regular. 
\end{enumerate}
\end{MainThm}

\bs{Notice that we did not include the case of finite cluster type $ A_1 $ in the statement.
	This is a particular situation where the corresponding variety is either regular or a union of two regular lines intersecting transversally in a closed point.
	See 
Observation~\ref{Obs:Special_Type_A1}.}

\begin{MainThm}[{Lemma~\ref{Lem:Bn_2_eq} and Theorem~\ref{Thm:B_prin}}]
	\label{Thm:B}
	There exists an isomorphism 
	$
	\cA^\prin_\c(B_n) \cong 
	\KK_\c[z_1, \ldots, z_{n-1}, u_1, u_2, u_3 ]/ \langle g_n, h_n \rangle $, 
	where the generators of the ideal on the right hand side are
	\[
	\begin{array}{l}
		g_n
		:= 
		\Big( u_1  u_2  - \lambda_{n}(\c) \Big)  u_3
		- \lambda_{n-1}(\c)^{-1} u_1^2 - P_{n-2}(z_1,\ldots, z_{n-2}) 
		\ ,
		
		\\[5pt]
		
		h_n := u_1 u_2  - \lambda_{n}(\c) - P_{n-1}(z_1, \ldots, z_{n-1}) 
		\ . 
	\end{array}
	\]
	The following characterization holds for the fibers of $ \phi \colon \Spec (\cA_\c^\prin (B_n)) \to S $ with $ \eta \in S $ being a closed point:
		\begin{enumerate}
			\item 
			If $ n = 2m+1 $ is odd, 
			then $ \Spec (\cA_\c^\prin (B_n))_\eta $ is singular if and only if 
			$ \lambda_{n}(\eta) = (-1)^{m+1} $.
			\item 
			If $ n = 2m $ is even,
			then $ \Spec (\cA_\c^\prin (B_n))_\eta $ is singular if and only if 
			$ \car(\KK) = 2 $ and $ \lambda_{n-1}(\eta) \in \bs{\kappa(\eta)^2} $ \bsOUT{$ \KK^2  $} is a square. 
		\end{enumerate}
		In the singular cases, 
		the singular locus is a closed point and
		$ \Spec (\cA_\c^\prin (B_n))_\eta $ is locally isomorphic to a hypersurface with an isolated singularity of type $ A_1 $.
	\\
	\underline{In other words}, the following holds:
	\begin{enumerate}
		\item[(i)]
		For
		$ n = 2m+1 $ odd,
		we have a stratification
		$ S = S_1 \sqcup S_2 $, 
		where the strata are 
		\[  
			S_1 := V(\lambda_{n} (\c) - (-1)^{m+1}) \subset S 
			\ \ \mbox{ and } \ \ 
			S_2 := S \setminus S_1 \ ,
		\] 
		such that 
		the fibers above every point of $ S_2 $ are regular, while 
		those above the points of $ S_1 $ are isomorphic to a hypersurface singularity of type $ A_1 $. 
		
		\item[(ii)]
		If $ n $ is even and $ \car(\KK) = 2 $, there is a stratification $ S = S_1 \sqcup S_2 $
		with 
		\[  
			S_1 := \Big\{ \eta \in S  \, \Big| \,  \sqrt{\lambda_{n-1}(\eta)} \in \bs{\kappa(\eta)} \bsOUT{\KK} \Big\}
			\ \ \mbox{ and } \ \  
			S_2 := S \setminus S_1 \ ,
		\]
		such that all fibers above $ S_2 $ are regular and 
		every fiber above a point of $ S_1 $ is isomorphic to a hypersurface singularity of type $ A_1 $. 
		
		\item[(iii)]
		 If $ n $ is even and $ \car(\KK) \neq 2 $, all fibers of $ \phi $ are regular. 
	\end{enumerate}	
\end{MainThm}

Notice that $S_1$ in Theorem~\ref{Thm:B}{\em (ii)} is a constructible subset by Chevalley's theorem. 
Moreover, if $\KK$ is \bsOUT{a perfect} \bs{an algebraically closed} field of characteristic $2$ then every element in \bsOUT{$\KK$} \bs{$ \kappa(\eta) $} is a square \bs{for every closed point $ \eta \in S $} and thus $ S_1 = S $.

\begin{MainThm}[{Lemma~\ref{Lem:Cn_hypersurface} and Theorem~\ref{Thm:C_prin}}]
	\label{Thm:C}
	The cluster algebra $ \cA^\prin_\c(C_n) $
	is isomorphic to 
	\[
	\KK_\c[z_1, \ldots, z_{n+1}]/
	\langle
	\,
	P_n(z_1, \ldots, z_n) z_{n+1}  -  P_{n-1}(z_1, \ldots, z_{n-1})^2  -
	\lambda_{n-2}(\c)  \lambda_n(\c)  
	\, 
	\rangle  \ .
	\] 
	The singularities of the fibers of $ \phi \colon \Spec (\cA_\c^\prin (C_n) ) \to S $ are characterized as follows, where $ \eta \bsOUT{= (\eta_1, \ldots, \eta_n)} \in S $ is a closed point:	
	\begin{enumerate}
		\item 
		Let $ \car(\KK) \neq 2 $.
		The fiber $ \Spec (\cA_\c^\prin(C_n))_\eta $ is singular if and only if 
		$ n $ is odd and 
		$ -\eta_n \in \bs{\kappa(\eta)^2} \bsOUT{\KK^2} $ is a square in \bs{$ \kappa(\eta) $}\bsOUT{ $ \KK $}.
		In the singular case, $ \Sing(\Spec (\cA_\c^\prin(C_n))_\eta) $ is a closed point and locally at the latter,
		$ \Spec (\cA_\c^\prin(C_n))_\eta $ is isomorphic to a hypersurface singularity of type $ A_1 $.
		
		\item 
		If $ \car(\KK) = 2 $, 
		then the fiber 
		$ \Spec (\cA_\c^\prin(C_n))_\eta $ is singular if and only if 
		$ -\eta_n \in \bs{\kappa(\eta)^2}$\bsOUT{ $\KK^2 $  is a square in $ \KK $}.
		In the singular case, let  $ \delta_n \in \bs{\rf} \bsOUT{\KK} $ be such that $ \delta_n^2 = -\eta_n^{-1} $ and 
		set $ \rho_n ( \eta) := \delta_n \lambda_{n-2} (\eta) $.  
		Then, we have:
		\begin{enumerate}
			\item 
			The singular locus is itself singular if and only if $ n-1 = 2m $ and $ \rho_n (\eta) = 1 $.
			
			\item 
			If $ n-1 = 2m $ and $ \rho_n(\eta) = 1 $,
			then $ \Sing ( \Spec (\cA_\c^\prin(C_n))_\eta ) $ has an isolated singularity at a closed point and locally at the latter, $ \Spec (\cA_\c^\prin(C_n))_\eta $ is isomorphic to the hypersurface singularity 
			
			\[
			\Spec (\bs{\rf}[x_1, \ldots, x_{2m} , y, z] / \langle yz + \Big( \sum_{i=1}^{m} x_{2i-1} x_{2i} \Big)^2 \rangle 
			\] 
			
			and $ \Sing ( \Spec (\cA_\c^\prin(C_n))_\eta )  $ identifies along this isomorphism with 
			
			\[
			V \Big( y,z, \sum_{i=1}^{m} x_{2i-1} x_{2i} \Big)
			\]

			which is isomorphic to an hypersurface of type $ A_1 $ if $ m > 1 $ and a union of two lines if $ m = 1 $.
			
			\item 
			At any point $ q $ at which $ \Sing ( \Spec (\cA_\c^\prin(C_n))_\eta )  $  is regular, $ \Spec (\cA_\c^\prin(C_n)_\eta) $ is isomorphic to a $ (n-2) $-dimensional cylinder over the hypersurface singularity 
			\[ 
			\Spec(\bs{\rf}[x,y,z]/\langle xy - z^2 \rangle ) 
			\] 
			of type $ A_1 $.
		\end{enumerate}
	\end{enumerate}
\underline{In other words}, we have the following cases of stratification:
\begin{enumerate}
	\item[(i)] 
	If $ \car(\KK) \neq 2 $ and $ n $ is odd, 
	there exists a stratification $ S = S_1 \sqcup S_2 $
	with 
	\[  
	S_1 := \Big\{ \eta \in S  \, \Big| \,  \sqrt{-\eta_n} \in \bs{\rf} \Big\}
	\ \ \mbox{ and } \ \  
	S_2 := S \setminus S_1 
	\]
	such that every fiber above a point of $ S_2 $ is regular and 
	all fibers above $ S_1 $ are isomorphic to a hypersurface singularity of type $ A_1 $. 
	
	\item[(ii)]
	For $ \car(\KK) \neq 2 $ and $ n $ even, all fibers of $ \phi $ are regular. 
	
	\item[(iii)]
	If $ \car(\KK) = 2 $ and $ n $ is \bsOUT{odd} \bs{even}, 
	there is a stratification $ S = S_1 \sqcup S_2 $,
	where 
	\[  
	S_1 := \Big\{ \eta \in S  \, \Big| \,  \sqrt{-\eta_n} \in \bs{\rf} \Big\}
	\ \ \mbox{ and } \ \  
	S_2 := S \setminus S_1 \ , 
	\]
	such that the fiber of $ \phi $ is regular above every point of $ S_2 $,
	while 
	all fibers above $ S_1 $ are isomorphic and singular.
	In the latter case, the singular locus of such a fiber is regular  
	and at every point of the \bsOUT{singular} singular locus, $ \Spec (\cA_\c^\prin(C_n))_\eta $ is isomorphic to a $ ( n- 2 ) $-dimensional cylinder over a hypersurface surface singularity of type $ A_1 $.

	\item[(iv)]
	For $ \car(\KK) = 2 $ and $ n $ \bsOUT{even} \bs{odd $( n = 2m+1)$}, 
	we have the stratification
	$ S = S_1 \sqcup S_2 \sqcup S_3 $,
	where (using the notation $ \rho_n (\eta) $ from (2))
	\[  
	S_1 := \Big\{ \eta \in S  \, \Big| \,  \sqrt{-\eta_n} \in \bs{\rf} \mbox{ and } \rho_n (\eta) = 1 \Big\} \ ,
	\]
	\[
	S_2 :=\Big\{ \eta \in S  \, \Big| \,  \sqrt{-\eta_n} \in \bs{\rf} \Big\} \setminus S_1 \ , 
	\ \ \mbox{ and } \ \  
	S_3 := S \setminus ( S_1 \cup S_2 ) \ , 
	\] 
	with the property that fibers above two points lying in the same stratum are isomorphic and  
	\begin{itemize}
		\item 
		fibers above points of $ S_3 $ are regular,
		
		\item 
		fibers above $ S_2 $ are singular with the same description of the singularities as for $ S_1 $ in (iii),
		and
		
		\item 
		if $ \eta \in S_1 $, then $ \Spec (\cA_\c^\prin(C_n))_\eta $ is singular and its singular locus is isomorphic to a hypersurface singularity of type $ A_1 $.
		Locally at the isolated singularity of the singular locus, 
		the fiber $ \Spec (\cA_\c^\prin(C_n))_\eta $ is isomorphic to
		\[
		\Spec (\bs{\rf}[x_1, \ldots, x_{2m} , y, z] / \langle yz + \Big( \sum_{i=1}^{m} x_{2i-1} x_{2i} \Big)^2 \rangle 
		\ , 
		\]  
		while at any other point of the singular locus, the description of the singularity is analogous to the one at points of $ S_1 $ in (iii).
	\end{itemize}
\end{enumerate}
\end{MainThm}

In Case~(2)(b), a desingularization can be constructed via three blowups analogous to the setting with trivial coefficients \cite[Proposition~5.5(2)]{BFMS}:
First, the singularities of the singular locus are resolved by blowing up the origin.
Then, the strict transform of the singular locus is regular and thus can be blown up.
Finally, we created a component in the singular locus, which is contained in the exceptional divisor of the first blowup. 
By choosing the strict transform of this component as the center for the third blowup, we resolve the singularities of $ \Spec (\cA_\c^\prin (C_n)) $. 

\begin{MainThm}[{Lemma~\ref{Lem:Dn_2_eq} and Theorem~\ref{Thm:D_prin}}]
	\label{Thm:D}
	 The cluster algebra $ \cA^\prin_\c(D_n) $ is isomorphic to  
	$ \KK_\c[z_1, \ldots, z_{n-2}, u_1, u_2, u_3, u_4 ]/ \langle h_1, h_2 \rangle $,
	where 
	\[
	\begin{array}{l}
		h_1 := 
		u_1 u_2 - u_3 u_4  -
		\lambda_{n-1}(\c)^{-1} u_2 u_4 
		\big(
		u_1 u_3 + P_{n-3} (z_1,\ldots, z_{n-3})  \big)
		\ ,
		
		\\[5pt]
		
		h_2 := 
		u_3 u_4  - P_{n-2} (z_1,\ldots, z_{n-2}) - \lambda_{n-1}(\c)
		\ . 
	\end{array} 
	\]
	We have the following cases for the fibers $ \Spec (\cA_\c^\prin (D_n))_\eta  $ of $ \phi \colon \Spec (\cA_\c^\prin (D_n)) \to S $:
	\begin{enumerate}
		\item[(a)] 
		If $ n = 4 $ and $ \lambda_3 (\eta) = 1 $,
		then $ \Sing (\Spec( \cA_\c^\prin (D_4))_\eta) $ identifies with the six coordinate axes in $ \AA_{\bs{\rf}}^6 $ along the isomorphism of the first part of the theorem.
		
		\item[(b)] 
		If $ n-2 = 2m $ is even, $ n > 4 $, and  $ \lambda_{n-1} (\eta) = (-1)^{m+1} $,
		then $ \Sing (\Spec( \cA_\c^\prin (D_n))_\eta) $ consists of 
		four regular, irreducible components $ Y_1, \ldots, Y_4 $ of dimension one 
		and one singular, irreducible component $ Y_0 $ of dimension $ n - 3 $
		whose singular locus is a closed point 
		coinciding with the intersection $ \bigcap_{i=1}^4 Y_i $.

		\item[(c)] 
		Otherwise
		(i.e., if $ n $ is odd, or if $ n $ is even and $ \lambda_{n-1} (\eta) \neq 1 $), 
		then the singular locus of the fiber is irreducible, regular, and of dimension $ n-3 $.
	\end{enumerate}
	The singularities classify as follows:
	\begin{itemize}
		\item 
		\underline{Case (c):} 
		Along its regular irreducible singular locus,
		$ \Spec(\cA_\c^\prin (D_n))_\eta $ is isomorphic to a cylinder 
		over a $ 3 $-dimensional hypersurface singularity of type $ A_1 $.
		
		\item 
		\underline{Cases (b):} 
		Let $ 0 \in Y_0 $ be the singular point of $ Y_0 $.
		Along $ Y_0 \setminus \{ 0 \} $, the situation is the same as in Case (c).
		Further, along $ Y_i \setminus \{ 0 \} $, for $ i \in \{ 1, \ldots, 4 \} $, 
		$ \Spec(\cA_\c^\prin (D_n))_\eta $ is isomorphic to an $ n $-dimensional hypersurface singularity of type $ A_1 $. 
		Finally, locally at $ 0 $, 
		$ \Spec(\cA_\c^\prin (D_n))_\eta $ is isomorphic to the intersection of two hypersurface singularities of type $ A_1 $, while  
		$ Y_0 $ is isomorphic to a $ ( n- 3) $-dimensional hypersurface singularity of type $ A_1 $.

		\item 
		\underline{Cases (a):}
		The situation is the same as in Case (b) with the exception that $ Y_0 $ is the union of two lines here 
		and locally at $ 0 $, 
		$ \Spec(\cA_\c^\prin (D_n))_\eta $ is isomorphic to the intersection of a hypersurface singularity of type $ A_1 $ with a divisor of the form $ V (xy) $. 
	\end{itemize}

	\underline{In other words}, we have:
	\begin{enumerate}
		\item[(i)] 
		If $ n = 4 $, 
		there is a stratification $ S = S_1 \sqcup S_2 $ with
		\[
			S_1 := V (\lambda_3 (\c) - 1 ) \subseteq S
			\ \ \mbox{ and } \ \  
			S_2 := S \setminus S_1  
		\]
		such that the singularities of all fibers above points of $ S_1 $ (resp.~$ S_2 $) 
		are as described in Case~(a) (resp.~Case~(c)) above. 
		
		\item[(ii)]
		For $ n $ even \bs{$( n = 2m+2)$} and $ n > 4 $,
		we have a stratification $ S = S_1 \sqcup S_2 $, where
		\[
			S_1 := V (\lambda_{n-1} (\c) - (-1)^{m+1} ) \subseteq S
			\ \ \mbox{ and } \ \  
			S_2 := S \setminus S_1  \ , 
		\]
		such that for every $ \eta \in S_1 $ (resp.~$ \eta \in S_2 $), the singularities of the fiber $ \Spec(\cA_\c^\prin (D_n))_\eta $ are as described in Case~(b) (resp.~Case~(c)) above. 
		
		\item[(iii)]
		If $ n $ is odd, then the singularities of every fiber of $ \phi $ are classified as in Case~(c) above.
		
	\end{enumerate}
\end{MainThm}

In Cases~(a) and~(b), a desingularization of $ \Spec (\cA_\c^\prin (D_n)) $ is obtained by two blowups that are constructed analogous to the situation with trivial coefficients \cite[Proposition~4.6(2)(c)]{BFMS}:
First, blow up the origin and then take as the second center the strict transform of $ Y_0 \cup \cdots \cup Y_4 $. 

\begin{MainThm}[{cf.~Lemma~\ref{Lem:En_2_eq} and Theorem~\ref{Thm:E_prin}}]
	\label{Thm:E}
	For $ k \in \{ 6,8\} $, 
		all fibers of the family defined by 
		$ \phi \colon \Spec ( \cA_\c^\prin (E_k) )\to S  $ are regular. 
		\\
		The cluster algebra 
		$ \cA^\prin_\c(E_7) $ is
		isomorphic to 
		$ \KK_\c[z_1, \ldots, z_{5}, u_1, \ldots, u_5 ]/ \langle h_1, h_2, h_3 \rangle $,
		where
		\[
		\begin{array}{l}
			h_1 :=
			P_{5}(z_1, \ldots, z_{5}) - u_3 P_2 (u_1,u_2)
			\ ,
			
			\\[3pt]
			
			h_2 := 
			u_3 u_4  - P_{4} (z_1,\ldots, z_{4}) - \lambda_{5}(\c)
			\ ,
			
			\\[3pt]
			
			h_3 := 
			P_{3} (u_1, u_2, u_5)  - P_{4} (z_1,\ldots, z_{4})
			\ . 
		\end{array} 
		\]
		The fibers $ Spec ( \cA_\c^\prin (E_7) )_\eta $ of the corresponding family are singular if and only if we have $ \lambda_{5} (\eta)  = -1 $.
		In the singular case, the singular locus is a regular irreducible surface $ Y $ 
		and locally at $ Y $, the variety $ \Spec ( \cA_\c^\prin (E_7) )_\eta $ is isomorphic to a cylinder over a $ 5 $-dimensional hypersurface singularity of type $ A_1 $.
		\\
		\underline{In other words}, there is a stratification
			$ S = S_1 \sqcup S_2 $, 
			where 
			\[ 
			S_1 := V(\lambda_{5} (\c) + 1) \subset S 
			\ \ \mbox{ and } \ \  
			S_2 := S \setminus S_1  \ , 
			\] 
			such that 
			the fibers above every point of $ S_2 $ are regular and 
			those above the points of $ S_1 $ are singular and locally at their regular singular loci, $ \Spec ( \cA_\c^\prin (E_7) )_\eta $ is isomorphic to a cylinder over a $ 5 $-dimensional hypersurface singularity of type $ A_1 $.  	
\end{MainThm}

In fact, Lemma~\ref{Lem:En_2_eq} and Theorem~\ref{Thm:E_prin} are more general as we put $ E_6, E_7, E_8 $ in a unified setting in order to avoid \bs{treating} the three cases separately as in \cite[Section~4.3]{BFMS}.
Hence, Section~\ref{Subsec:En} also includes results on singularities of certain cluster algebras that are not of finite cluster type. 

\begin{MainProp}[{Lemma~\ref{Lem:F_4}}]
	\label{Prop:F}
	The cluster algebra 
	$ \cA^\prin_\c(F_4) $ is isomorphic to a trivial family over $ \KK_\c $,
	where each fiber is isomorphic to the corresponding cluster algebra
	$ \cA(F_4) $
	with trivial coefficients.
	Since
	$ \Spec (\cA (F_4)) $ is isomorphic to a regular hypersurface in $ \AA_\KK^5 $, the respective result holds for \bs{the fibers of} $ \Spec(\cA^\prin_\c(F_4)) $. 
\end{MainProp}

\begin{MainProp}[{Proposition~\ref{Prop:G_prin}}]
	\label{Prop:G}
	There is an isomorphism 
	\[	
		\cA^\prin_\c(G_2)
	\cong
	\KK_{(c_1,c_2)}[x, y, z]/
	\langle 
	\, 
	xyz -
	y -  c_1 - x^3
	\,
	\rangle  
	\ .
	\] 
	A fiber above $ \eta  \in S $ of the family defined by 
		$ \phi \colon \Spec(\cA_\c^\prin (G_2))\to S $ is singular if and only if $ \car(\KK) = 3 $ and $ \eta_1 \in  \bs{\rf}^3 $ is a cubic element.
	In the singular case, $ \Spec(\cA_\c^\prin (G_2))_\eta $  is isomorphic to a hypersurface with an isolated singularity of type $ A_2 $ at a closed point.
		\\
	\underline{In other words}, the following holds:
	\begin{enumerate}
		\item[(i)]
		If $ \car(\KK) = 3 $, then there is a stratification $ S = S_1 \sqcup S_2 $
		with 
		\[  
		S_1 := \Big\{ \eta \in S  \, \Big| \, \sqrt[3]{\eta_1} \in \bs{\rf} \Big\}
		\ \ \mbox{ and } \ \  
		S_2 := S \setminus S_1 \ ,
		\]
		such that all fibers above $ S_2 $ are regular and 
		every fiber above a point of $ S_1 $ is isomorphic to a hypersurface singularity of type $ A_2 $. 
		
		\item[(ii)]
		If $ \car(\KK) \neq 3 $, all fibers of $ \phi $ are regular. 
	\end{enumerate}
\end{MainProp}

\begin{Obs}
	\bs{
		So far, we have investigated the fibers of the map $ \Spec (\cA^\prin_\c (\Sigma)) \to S $ in Theorems~\ref{Thm:A} -- Proposition~\ref{Prop:G}. 
		In order to connect them to statements on the singularities of cluster algebras with coefficients,
		we observe the following: 
		In each result the conditions to have a singularity depend only on parts of $ c_1, \ldots, c_n $ and not all of them.
		In particular, recall the definition of $ \lambda_n (\c) $, Definition~\ref{Def:ContPoly_lamda_n}(2). 
		Therefore, we could start with a tropical semifield $ \PP $ with free set of generators $ (p_1, \ldots, p_n ) = (c_1, \ldots, c_n ) $ (see Subsection~\ref{Subsec:GeomType_Perpectives})
		and then specialize only those elements $ c_i $ which are relevant for the existence of singularities appropriately to $ +1 $ or $ -1 $ for the condition to be true. 
		As a result, we obtain a cluster algebra with non-trivial coefficients which has the described singularities. 
		 }
\end{Obs}

\begin{Bem}
	In a private conversation, 
	Nathan Ilten pointed out the following to the authors:
	If we restrict to (algebraically closed) base fields of characteristic zero, then any time the exchange matrix $ B $ has full rank, the families $ \Spec ( A^\prin_\c (\Sigma ) ) $ and  $ \Spec (A^\gen_{\s,\t} (\Sigma)) $ 
	admit a torus action that acts transitively on the base,
	by arguments similar to those in \cite[Section~5]{INT_complexes}.
	In particular, all fibers are isomorphic, so the corresponding stratification in such instances is trivial;
	e.g., this is reflected in Theorems~\ref{Thm:A}, \ref{Thm:B}, \ref{Thm:C} if $ n $ is even.
	\\	
	More generally, by their results in \cite{INT_complexes},
	it suffices to just add enough frozen coefficients to the cluster algebra with trivial coefficients so that the resulting cluster algebra $ \cA' $ is positively graded and has full rank. 
	Indeed, in that case, the fibers over the torus of the universal cluster algebra for $ \cA' $ are all torus translates, hence isomorphic. 
	Therefore, it suffices to understand the fibers of $ \cA' $. 
	\\
	Finally, one of their main results (e.g.~\cite[Theorem~1.3.1]{INT_complexes}) is a description of the most singular fiber over the affine space (still in characteristic zero): 
	it is the Stanley-Reisner scheme associated to the cluster complex.
\end{Bem}

\subsection{Classification for cluster algebras of rank two}
Given a labeled seed $ \Sigma = (\x, B) $,
the corresponding cluster algebra is of rank two if the exchange matrix $ B $ is of the form 
\[
	B = \begin{pmatrix}
		0 & a \\
		b & 0 
	\end{pmatrix}
	\ , 
\]
where $ a, b \in \ZZ $ are either both zero or are both non-zero and of opposite sign. 
While the situation for $ a = b = 0 $ is rather simple from the perspective of singularity theory (Lemma~\ref{Lem:rk2_zero}),
we prove the following result in the other \bs{cases}:

\begin{MainThm}[Theorem~\ref{Thm:rk2_nonzero}]
	\label{Thm:rank2}
	\bs{Let $ \Sigma_2 $ be the labeled seed of a cluster algebra of rank two and assume that $ a,b \neq 0 $ using the notation above.}
	\bs{Assume that  $ \KK $ is algebraically closed.}
	 	Let $ X_\eta := \Spec (\cA_\c^\prin (\Sigma_2))_\eta $ be the fiber above $ \eta = (\eta_1, \eta_2) \in S $ of the family determined by $ \phi \colon \Spec (\cA_\c^\prin (\Sigma_2)) \to S $. 
	\bsOUT{Assume that  $ \KK $ is algebraically closed.}
	\begin{enumerate}
		\item 
		The singular locus of $ X_{\eta}  $ consists of two disjoint components (which are not necessarily irreducible and which are possibly empty),
		$ \Sing (X_{\eta}) = Y_a \, \sqcup \,  Y_ b $,
		where	
		\[  
		\begin{array}{l} 
			Y_a :=  V (a, x_1^a + \eta_2 , x_2,  x_1 y_1 - \eta_1, y_2 )
			\ ,
			\\[3pt]
			Y_b :=
			V (b, x_1, x_2^b + \eta_1, y_1,  x_2 y_2 - \eta_2 ) 
			\ .
		\end{array}  
		\] 
		\bs{Note that, up to isomorphism, $Y_a$ and $ Y_b$ are independent of $ \eta $.} 
		Notice that $ Y_a $ or $ Y_b $ may be empty depending on the characteristic $ p := \car(\KK) $ of $ \KK $. 
		For instance, we have
		\[
		a \not\equiv 0 \mod p  \Longrightarrow Y_a = \varnothing
		\ \
		\mbox{ and } 
		\ \
		b \not\equiv 0 \mod p  \Longrightarrow 
		Y_b = \varnothing
		\ .  
		\]
		
	\item[(2)] 
		$ \Spec (\cA_\c^\prin (\Sigma_2))  $ is isomorphic to a trivial family over $ S $,
		where each fiber is isomorphic to the cluster algebra corresponding to $ \Sigma_2 $ with trivial coefficients.
		Hence, fixing $ \eta = (\eta_1, \eta_2) \in S $, we have  
		\[
		\begin{array}{rcl}
			X_\eta
			& \cong 
			& \Spec (\KK [x_1', x_2',y_1', y_2'] / \langle \, x_1' y_1' - 1 - x_2'^b  , \,  x_2' y_2' - 1 - x_1'^a \,  \rangle ) 
			\\[3pt]
			& \cong 
			& \Spec (\KK [w_1, w_2, z_1, z_2] / \langle \, w_1 z_1 - 1 + w_2^b  , \,  w_2 z_2 - 1 + w_1^a \,  \rangle )		
		\end{array} 
		\]
		\item[(3)]
		Let $ \alpha, \beta \in \ZZ_+ $ be the largest positive integers such that $ \alpha \, | \, a $ and $ \alpha \not\equiv 0 \mod p $,
		resp.~$ \beta \, | \, b $ and $ \beta \not\equiv 0 \mod p $.
		We have: 
		\[  
		\begin{array}{l}
			\displaystyle 
			Y_a \cong \bigcup_{\zeta \in \mu_\alpha (\KK)}  
			V (a, w_1 - \zeta , w_2, z_1 - \zeta^{-1}, z_2 ) 
			\ ,
			\\[5pt]
			\displaystyle   
			Y_b \cong
			\bigcup_{\xi \in \mu_\beta (\KK)}  
			V (b, w_1, w_2 - \xi , z_1,  z_2 - \xi^{-1} )
			\ .
		\end{array}
		\] 
		where the disjoint unions range over the $ \alpha $-th (resp.~$ \beta $-th) roots of unity $ \mu_\alpha (\KK) $ (resp.~$ \mu_\beta (\KK)) $ in $ \KK $. 
		
		\item[(4)] 
		If $ \Sing ( X_{\eta} ) $ is non-empty, then the singularities classify as follows:
		Let $ Y_{i,j} $ be a connected component of $ Y_i $, with $ i \in \{ a, b \} $.
		Locally, along $ Y_{i,j} $,
		the fiber $ X_\eta $ is isomorphic to a two-dimensional hypersurface singularity of type $ A_{p^m-1} $,
		where $ m := m(i) $ is the positive integer such that $ |a| = \alpha p^m  $ (if $ i = a $), resp.~$ |b| = \beta p^m $ (if $ i = b $).
	\end{enumerate} 
\end{MainThm}

The additional hypothesis that $ \KK $ is algebraically closed is made in order to avoid a too technical statement. 
In Remark~\ref{Rk:rk2_arbi_field}, we briefly address the aspects that need to be taken into account if $ \KK $ is not necessarily algebraically closed. 

\subsection{Summary of contents}
In Section~\ref{Sec:ClusterBasics}, we recall all notions and results from the theory of cluster algebras which we need.
\bsOUT{A reader who is familiar with cluster algebras may generously skip this part and directly move on with Section~\ref{Sec:GenFamily}.}
In Section~\ref{Sec:GenFamily}, we introduce the notion of cluster algebras with generic coefficients and discuss its connection to the case of principal coefficients.
The following Section~\ref{Sec:FiniteSing} is devoted to the classification results in the finite cluster type case and their proofs. 
In particular,
we recall the required facts on continuant polynomials (Section~\ref{Subsec:Cont}), which are a key tool for our investigations.
Finally we address the singularity theory of cluster algebras of rank two in Section~\ref{Sec:Rk2}.

\subsection*{Acknowledgments}
Major parts of the research on this project have been realized during a two-week-long stay at the International Centre of Mathematical Sciences (ICMS), Edinburgh, UK,
within their ``Research in Groups" program.
We are appreciative for the opportunity to work at the Maxwell House with its stimulating and carefree environment.
We thank the institute for their hospitality. 
In particular, we are grateful to the staff whose efforts made the stay a rewarding experience. 
\\
\bs{The authors thank the anonymous referees whose useful comments helped to improve the article.}

\section{Cluster algebras basics}
\label{Sec:ClusterBasics}

Let us recall some basics on cluster algebras.
In contrast to the section on cluster basics in our previous work \cite[Section~2]{BFMS}, 
we also take the case of non-trivial coefficients into account. 
\bs{For full details on the theory of cluster algebras, we refer to} \bsOUT{We follow mainly \cite[Section~2]{FZ2007} and to some extent} \cite{FZ2002,FZ2003II,BFZ2005,FZ2007}.

Let $ (\PP, \oplus, \cdot ) $ be a semifield, i.e., 
it fulfills the same axioms as a field with the possible exception for the existence of neutral element with respect to $ \oplus $ and for the existence of an inverse element with respect to $ \oplus $.

Fix a positive integer $ n \in \ZZ_+ $ and a field $ \KK $. 
The {\em ambient field} for a cluster algebra $ \cA $ of rank $ n $ is a field $ \cF$ isomorphic to the field of rational functions in $ n $ independent variables with coefficients in $ \KK \PP $. 
\bs{Recall that in the ring structure of $ \KK\PP $, the operation $ \oplus$ plays no role.} 
\\
Let us point out that \cite{FZ2007} treat the more general case over $ \ZZ \PP $ instead of $ \KK \PP $.
Since we are interested in problems in the context of algebraic geometry, we restrict ourselves to the case over a field $ \KK $ and whenever we refer to \cite{FZ2007}, we provide the adapted variant for our setting. 
On the other hand, in contrast to many references in the literature, 
we want to allow fields $ \KK $ different from $ \QQ $ or $ \CC $; 
in particular, we take fields of positive characteristic into account.

\subsection{Cluster algebras of geometric type \bsOUT{and a review on different perspectives}}
\label{Subsec:GeomType_Perpectives}

\bs{A cluster algebra is of {\em geometric type} if the coefficient semifield $ \PP $
is a tropical semifield (\cite[Definition~2.12]{FZ2007}).}  
Let us \bs{recall} \bsOUT{first remind the reader of} the definition of a tropical semifield.
 
\begin{defi}[{\cite[Definition~2.2]{FZ2007}}] 
	Let $ ( \PP, \cdot ) $ be a multiplicative free abelian group with a finite \bs{free} set of generators 
	$ (p_1, \ldots, p_\ell) $.
	We endow $ \PP $ additionally with the auxiliary addition $ \oplus $ given by
	\[
	\prod_{i=1}^\ell p_i^{a_i} \oplus \prod_{i=1}^\ell p_i^{b_i} := \prod_{i=1}^\ell p_i^{\min\{ a_i, b_i\}}.
	\]
	Using this addition, $ (\PP, \oplus, \cdot ) $ is a semifield,
	which is usually denoted by $ \operatorname{Trop}(p_1, \ldots , p_\ell ) $, the {\em tropical semifield}. 
\end{defi}

Recall that an $ n $-tuple $ \x = (x_1, \ldots, x_n) $ of elements in $ \cF $ is called a {\em free generating set} if 
$ x_1, \ldots, x_n $ are algebraically independent over $ \KK\PP $ and $ \cF= \KK\PP(x_1, \ldots, x_n) $. 
\\
Moreover, recall that an $ n \times n $ integer matrix $ B = ( b_{ij} ) $ is called  {\em skew-symmetrizable}
if there exist positive integers $ d_1, \ldots, d_n $ such that, for all $ i, j \in \{ 1, \ldots, n \} $, we have
\[
	d_i b_{ij} = - d_j b_{ji} \ . 
\] 
In particular, a skew-symmetrizable is sign-skew-symmetric.

\bs{ 
	In this article, we will work with the definition of cluster algebras via extended exchange matrices, cf.~{\cite[after Definition~2.12]{FZ2007}}.
	For this let
	\begin{itemize}
		\item 
		$ \y = (y_1, \ldots, y_n) $ be an $ n $-tuple of elements in $ \PP $, called the {\em coefficient tuple},
		
		\item 
		$ B = (b_{ij}) $ be a skew-symmetrizable $ n \times n $ integer matrix,
		called the {\em exchange matrix},
		
		\item 
		$ \x = (x_1, \ldots, x_n) $ be an $ n $-tuple of elements in $ \cF $ forming a free generating set, called the {\em cluster}. 
	\end{itemize}
	
	Rename the generators $ \p = (p_1, \ldots, p_\ell) $ of the tropical semifield of $ \PP $ by calling them $ x_{n+1}, \ldots, x_{m} $, where $ m := n + \ell $,
	and set 
	\[
	\widetilde{\x} := (x_1, \ldots, x_n, \ldots, x_m). 
	\]
	Since the coefficients $ y_1, \ldots, y_n $ are Laurent monomials in $ x_{n+1}, \ldots, x_m $, we may define the integers 
	\[
	b_{ij} \in \ZZ,
	\ \ \
	\mbox{for } j \in \{ 1, \ldots, n \} 
	\mbox{ and } i \in \{ n+1, \ldots, m \}
	\]
	via 
	\begin{equation}
		\label{eq:y_j}
		y_j = \prod_{i=n+1}^m  x_i^{b_{ij}} \ . 
	\end{equation}
	
	In other words,
	we are extending the $ n \times n $ exchange matrix $ B $ to an {\em extended exchange matrix $ \widetilde B $} of size $ m \times n $,
	where 
	\[
	\widetilde B := ( b_{ij})_{ i \in \{ 1, \ldots, m \}, j \in \{1, \ldots, n\}} \ .
	\]
	Thus, the new entries below $ B $ (i.e., those $ b_{ij} $ with $ i \geq n + 1 $) are determined by the exponents of $ x_i $  in \eqref{eq:y_j}. 
	The pair $ \Sigma = (\widetilde{\x},\widetilde{B}) $ is called a {\em labeled seed}.
	
	A key tool for the construction of a cluster algebra is the notion of mutation. 
	
	\begin{defi}
		\label{Def_missing}
		Let $ \Sigma = (\widetilde \x, \widetilde B) $ be a labeled seed and $ k \in \{ 1, \ldots, n \} $.  
		The {\em seed mutation} $ \mu_k $ in direction $ k $ transforms $ (\widetilde \x, \widetilde B) $ into the labeled seed $ \mu_k (\widetilde \x, \widetilde B)  := (\widetilde \x', \widetilde B') $,
		where:
		\begin{enumerate}
			\item 
			a new cluster $ \x' = (x_1', \ldots, x_n') $ is introduced via $ x_j' := x_j $ if $ j \neq k $ 
			and $ x_k' \in \cF $ is defined by the {\em exchange relation}:
			\begin{equation}
				\label{eq:exchange_TildeB_1}
				x_k x_k' 
				= 
				\prod\limits_{\substack{i=1 \\ b_{ik} > 0 }}^m x_i^{b_{ik}} + \prod\limits_{\substack{i=1 \\ b_{ik} < 0 }}^m x_i^{-b_{ik}} 
				\ .
			\end{equation}
			Notice that the appearing exponents are the entries of the $ k $-th column of $ \widetilde{B} $.
			
			\item 
			the new exchange matrix 
			$ \widetilde B' = (b_{ij}') $ is of the same size as $ \widetilde B $ and its entries are determined by
			\[
			b_{ij}' := 
			\begin{cases}
				- b_{ij} & \mbox{if } i = k \mbox{ or } j = k \ ,
				\\
				b_{ij} + \operatorname{sgn}(b_{ik}) [ b_{ik} b_{kj} ]_+
				& \mbox{otherwise} \ .
			\end{cases}
			\]	
		\end{enumerate}
		Since we are only allowed to mutate in directions $ \{ 1, \ldots, n \} $,
		the variables $ (x_1, \ldots, x_n) $ are sometimes called the {\em mutable variables} of $ (\widetilde{\x}, \widetilde{B}) $,
		while $ (x_{n+1}, \ldots, x_m) $ are called the {\em frozen variables} or {\em coefficient variables} of $ (\widetilde{\x}, \widetilde{B}) $, cf.~\cite[Definition~3.1.1]{FWZ2016}, for example.
	\end{defi}

Note that the seed mutation is well-defined, i.e., $ \mu_k (\Sigma) $ is again a labeled seed in $ \cF $.
Furthermore, $ \mu_k $ is an involution, i.e., we have $ \mu_k ( \mu_k (\Sigma) ) =  \Sigma $.

Let us point out that a key idea in \cite{BFMN_degen} is to introduce a distinction between the notion of coefficient variables 
and that of frozen variables.
This allows them to consider cluster algebras in a relative situation with respect to the coefficient ring, 
which is then leading to families of cluster algebras in the universal setting. 
\\
While we are also working with families of cluster algebras arising from a relative situation, 
we do not have to make a distinction between frozen and coefficient variables.
For more on this, we refer to Section~\ref{Sec:GenFamily}.
}

\bsOUT{ 
\begin{defi}[{\cite[Definition~2.3]{FZ2007}}]
	\label{Def:seed}
	A {\em labeled $ Y $-seed} $ (\y, B) $ in $ \PP $ consists of
	\begin{itemize}
		\item an $ n $-tuple $ \y = (y_1, \ldots, y_n) $ of elements in $ \PP $
		and
		
		\item 
		a skew-symmetrizable $ n \times n $ integer matrix $ B = (b_{ij}) $.
	\end{itemize}

	A {\em labeled seed} $ (\x,\y,B) $ in $ \cF $ consists of 
	\begin{itemize}
		\item a labeled $ Y $-seed $ (\y, B) $ in $ \PP $ and
		
		\item 
		an $ n $-tuple $ \x = (x_1, \ldots, x_n) $ of elements in $ \cF $ forming a free generating set. 
	\end{itemize}
	
	The tuple $ \x $ is called the {\em cluster},
	the tuple $ \y $ is called the {\em coefficient tuple},
	and the matrix $ B $ is called the {\em exchange matrix} of $ (\x,\y, B) $.
\end{defi}

A key tool for the construction of a cluster algebra is seed mutation, which creates a new labeled seed from a given one.
More precisely, mutation is given as follows using the notation 
\[ 
	[a]_+ := \max \{ a, 0 \} \ , 
	\ \ \ 
	\mbox{ for } a \in \ZZ.
\]

\begin{defi}[{\cite[Definition~2.4]{FZ2007}}]
	\label{Def:seed_mutation}
	Let $ (\x,\y,B) $ be a labeled seed in $ \cF $
	and choose a $ k \in \{ 1, \ldots, n \} $.
	The {\em seed mutation} $ \mu_k $ in direction $ k $ transforms $ (\x,\y, B) $ into the labeled seed $ \mu_k (\x,\y,B)  := (\x',\y',B') $,
	where: 
	\begin{enumerate}
		\item 
		The $ n \times n $ matrix $ B' = (b_{ij}') $ is determined by
		\[
			b_{ij}' := 
			\begin{cases}
				 - b_{ij} & \mbox{if } i = k \mbox{ or } j = k \ ,
				 \\
				 b_{ij} + \operatorname{sgn}(b_{ik}) [ b_{ik} b_{kj} ]_+
				 & \mbox{otherwise} \ .
			\end{cases}
		\]
		
		\item
		The coefficient tuple $ \y' = (y_1', \ldots, y_n') $ is defined via
		\[
		y_j' := 
		\begin{cases}
		y_k^{-1} & \mbox{if } j = k \ ,
		\\
		y_{j} y_k^{[b_{kj}]_+} (y_k \oplus 1)^{-b_{kj}}
		& \mbox{if } j \neq k \ .
		\end{cases}
		\]
		 
		\item 
		The cluster $ \x' = (x_1', \ldots, x_n') $ is given by $ x_j' := x_j $ if $ j \neq k $ 
		and $ x_k' \in \cF $ is defined by the {\em exchange relation}
		\begin{equation}
		\label{eq:exchange_y}
			x_k' := \frac{\displaystyle  y_k \prod\limits_{i=1}^n x_i^{[b_{ik}]_+} + \prod\limits_{i=1}^n x_i^{[-b_{ik}]_+} }{\displaystyle  (y_k \oplus 1) x_k }
			\ .
		\end{equation} 
	\end{enumerate} 
\end{defi}

Note that the seed mutation is well-defined, i.e., $ \mu_k (\x,\y,B) = (\x', \y', B') $ fulfills the properties of a labeled seed in $ \cF $, Definition~\ref{Def:seed}.
Furthermore, $ \mu_k $ is an involution, i.e., we have $ \mu_k ( \mu_k (\x,\y,B) ) =  (\x,\y,B) $.

As pointed out in \cite[Remark~2.7]{FZ2007}, the coefficient tuple $ \y $ may be encoded as a $ 2n $-tuple $ \p = ( p_1^\pm, \ldots, , p_n^\pm ) $ of elements in $ \PP $ which are normalized in the sense that $ p_j^+ \oplus p_j^- = 1 $ for every $ j \in \{ 1, \ldots, n \} $. 
This is the original definition of a coefficient tuple, 
as in \cite[Definition~5.3]{FZ2002} or \cite[\S~1.2]{FZ2003II},
and the equivalence is explained in \cite[(5.2), (5.3)]{FZ2002}. 
More precisely, for $ j \in \{ 1, \ldots, n \} $, we have 
\[
	p_j^+ = \frac{y_j}{y_j \oplus 1} \ ,
	\ \ \
	p_j^- = \frac{1}{y_j \oplus 1} \ ,
	\ \ \mbox{ and, on the other hand, }
	\ \ y_j = \frac{p_j^+}{p_j^-} \ .
\]
The exchange relation \eqref{eq:exchange_y} can be rephrased in terms of $ \p $ as
\begin{equation}
	\label{eq:exchange_p}
	x_k x_k' = p_k^+ \prod\limits_{i=1}^n x_i^{[b_{ik}]_+} + p_k^- \prod\limits_{i=1}^n x_i^{[-b_{ik}]_+} 
	\ .
\end{equation}
We sometimes use the abbreviation 
\begin{equation}
	\label{eq:exchange_rhs}
	P_k(\x) :=  p_k^+ \prod\limits_{i=1}^n x_i^{[b_{ik}]_+} + p_k^- \prod\limits_{i=1}^n x_i^{[-b_{ik}]_+} 
	\ .
\end{equation}

\bigskip 
}

Recall that the {\em $ n $-regular tree} $ \TT_n $ is the infinite simply laced tree such that every vertex $ v $ has $ n $ edges and the latter are labeled with the numbers $ 1, \ldots, n $.

\begin{defi}[{\cite[Definition~2.9]{FZ2007}}]
	A {\em cluster pattern} is an assignment of a labeled seed $ \Sigma_t = \bs{( \widetilde \x_t, \widetilde B_t)} \bsOUT{( \x_t, \y_t, B_t)} $ to every vertex $ t \in \TT_n $ such that the following property holds:
	\\
	For every pair of vertices $ t, t' \in \TT_n $ which are joined by an edge, say with label $ k $, we have 
	$ \mu_k (\Sigma_t) = \Sigma_{t'} $. 
\end{defi}

Note that we also have $ \mu_k (\Sigma_{t'}) = \Sigma_{t} $ in the above situation since $ \mu_k $ is an involution.

\begin{defi}[{\cite[Definition~2.11]{FZ2007}}] 
	\label{Def:ClusterAlgebra}
	For a cluster pattern determined by the seeds $ \Sigma_t = \bs{( \widetilde \x_t, \widetilde B_t)} \bsOUT{( \x_t, \y_t, B_t)}  $, $ t \in \TT_n $, 
	we set
	\[
		\cX := \bigcup_{t \in \TT_n } \x_t = \{ x_{i,t} \mid t \in \TT_n, i \in \{ 1 , \ldots, n \} \},
	\] 
	where $ \x_t = (x_{1,t}, \ldots, x_{n,t} ) $.
	We call $ x_{i,t} \in \cX $ the {\em cluster variables}.
	\\
	The {\em cluster algebra} $ \cA $ associated to the given cluster pattern is the $ \KK\PP $-subalgebra of the ambient field $ \cF $ generated by all cluster variables,
	\[
		\cA = \KK\PP [\cX] \ . 	
	\]
	For a labeled seed $ \Sigma = \bs{ (\widetilde \x , \widetilde B) = ( \widetilde \x_t, \widetilde B_t)} \bsOUT{( \x, \y, B) = ( \x_t, \y_t, B_t)}  $ with $ t \in \TT_n $ a fixed value,
	we also use the notation
	\[
		\cA = \cA (\Sigma) = \bs{\cA(\widetilde \x, \widetilde B)} \bsOUT{\cA (\x, \y, B)} \ . 
	\]
\end{defi}

\medskip 

By the last part, we may define a cluster algebra also by giving a labeled seed $ \Sigma \bsOUT{ =  (\x, \y, B)} $ in $ \cF $ and to create then a cluster pattern by repeatedly mutating in all possible directions.

\bsOUT{The setting in which we are interested in are cluster algebras of geometric type.

\begin{defi}[{\cite[Definition~2.12]{FZ2007}}]
	\label{Def:geomType}
	A cluster algebra (resp.~a cluster pattern) is of {\em geometric type} if the coefficient semifield $ \PP $
	is a tropical semifield.  
\end{defi}

From the perspective on coefficients of the labeled seeds $ (\x_t,\y_t, B_t) $ of a cluster pattern
as $ 2n $-tuple  $ \p_t = ( p_{1,t}^\pm, \ldots, , p_{n,t}^\pm ) $ of elements in $ \PP = \operatorname{Trop}(p_1, \ldots, p_\ell) $, 
the corresponding cluster algebra is of geometric type if $ \p_t $ is normalized, that is, $p_{j,t}^+ \oplus p_{j,t}^-=1$, 
and if for every $ i \in \{ 1, \ldots, n \} $ and $ t \in \TT_n $, 
$ p_{i,t}^+ $ and $ p_{i,t}^- $ are monomials in $ p_1, \ldots, p_\ell $ with {\em non-negative} exponents. 
For more on this, we refer to the discussion following \cite[Definition~2.12]{FZ2007}.

\begin{Bem}[Extended exchange matrices, cf.~{\cite[after Definition~2.12]{FZ2007}}]
	\label{Rk:coeff_in_matrix}
	If we have a cluster pattern of geometric type, 
	it is possible to rephrase the coefficients 
	by extending the exchange matrices $ B_t $ suitably.
	Let us explain this for a fixed seed $ (\x, \y, B ) $ in order to avoid carrying the index $ t \in \TT_n $ along the way.
	\\
	Recall that $ \x = (x_1, \ldots, x_n) $.
	Rename the generators $ \p = (p_1, \ldots, p_\ell) $ of the tropical semifield of $ \PP $ by calling them $ x_{n+1}, \ldots, x_{m} $, where $ m := n + \ell $. 
	Since the coefficients $ y_1, \ldots, y_n $ are Laurent monomials in $ x_{n+1}, \ldots, x_m $, we may define the integers 
	\[
	b_{ij} \in \ZZ,
	\ \ \
	\mbox{for } j \in \{ 1, \ldots, n \} 
	\mbox{ and } i \in \{ n+1, \ldots, m \}
	\]
	via 
	\begin{equation}
	\label{eq:y_j_old}
	y_j = \prod_{i=n+1}^m  x_i^{b_{ij}} \ . 
	\end{equation}
	
	In other words,
	we are extending the $ n \times n $ exchange matrix $ B $ to an {\em extended exchange matrix $ \widetilde B $} of size $ m \times n $,
	where 
	\[
	\widetilde B := ( b_{ij})_{ i \in \{ 1, \ldots, m \}, j \in \{1, \ldots, n\}} \ .
	\]
	Thus, the new entries below $ B $ (i.e., those $ b_{ij} $ with $ i \geq n + 1 $) are determined by the exponents of $ x_i $  in \eqref{eq:y_j}. 
	\\
	Since $ y_j = p_j^+/ p_j^- $ and the exponents of $ p_j^+ $ and $ p_j^- $ are non-negative (as we are in the geometric type),
	\eqref{eq:y_j} and \eqref{eq:exchange_p} provide that the exchange relation can be written as
	\begin{equation}
	\label{eq:exchange_TildeB_1_old}
	x_k x_k' = \prod\limits_{i=1}^m x_i^{[b_{ik}]_+} + \prod\limits_{i=1}^m x_i^{[-b_{ik}]_+} 
	= 
	\prod\limits_{i:b_{ik} > 0 } x_i^{b_{ik}} + \prod\limits_{i: b_{ik}<0} x_i^{-b_{ik}}, \ \  \text{ for } k \in \{ 1, \ldots, n \}
	\ .
	\end{equation}
	Notice that the index $ i $ of each product ranges in $ \{ 1, \ldots, n , \ldots , m \} $.
	Hence, the appearing exponents are the entries of the $ k $-th column of $ \widetilde{B} $.
	\\
	By \cite[Proposition~5.8]{FZ2002}, the mutation rule for the coefficients in Definition~\ref{Def:seed_mutation} (2)
	takes the same form as mutation rule for the entries of the matrix $ B $ in Definition~\ref{Def:seed_mutation} (1) by applying it in the case $ i \geq n+1 $.
	In particular, if we set $ \widetilde x := (x_1, \ldots, x_n, \ldots, x_m) $,
	then $ (\widetilde \x, \widetilde B) $ together with all its mutations in directions $ \{ 1, \ldots, n \} $ 
	determines the same data as the seed $ (\x,\y,B) $. 

	Since we are only allowed to mutate in directions $ \{ 1, \ldots, n \} $,
	the variables $ (x_1, \ldots, x_n) $ are sometimes called the {\em mutable variables} of $ (\widetilde{\x}, \widetilde{B}) $,
	while $ (x_{n+1}, \ldots, x_m) $ are called the {\em frozen variables} or {\em coefficient variables} of $ (\widetilde{\x}, \widetilde{B}) $, cf.~\cite[Definition~3.1.1]{FWZ2016}, for example.
\end{Bem} 
 	
	Let us point out that a key idea in \cite{BFMN_degen} is to introduce a distinction between the notion of coefficient variables 
	and that of frozen variables.
	This allows them to consider cluster algebras in a relative situation with respect to the coefficient ring, 
	which is then leading to families of cluster algebras in the universal setting. 
	\\
	While we are also working with families of cluster algebras arising from a relative situation, 
	we do not have to make a distinction between frozen and coefficient variables.
	For more on this, we refer to Section~\ref{Sec:GenFamily}.
}

\begin{Bem}[Quiver perspective in the case of skew-symmetric exchange matrices]
	\label{Rk:Quiver_view}
	Let $ \Sigma  \bsOUT{= (\x,\y,B )} $ be a labeled seed. 
	If the exchange matrix $ B $ is skew-symmetric, 
	there is an interpretation of $ \Sigma $ in terms of quivers. 
	\bsOUT{To explain this, we use the variant of Remark~\ref{Rk:coeff_in_matrix} $ (\widetilde{\x}, \widetilde{B}) $ to encode the seed solely with mutable and frozen variables $ \widetilde{\x} $ as well as an $ m \times n $ matrix $ \widetilde{B} $.} 
	\\
	The quiver $ \cQ(\widetilde{B}) $ associated to $ \widetilde{B} = (b_{ij} ) $ is the finite directed graph 
	with $ m $ vertices, labeled by $ 1, \ldots, m $
	and such that,
	for $ i \geq j $ there are $ |b_{ij}| \in \ZZ_{\geq 0} $ many pairwise different arrows going from the vertex $ i $ to the vertex $ j $, if $ b_{ij} \geq 0 $, 
	respectively, going from $ j $ to $ i $, if $ b_{ij} < 0 $.
	In order to distinguish mutable and frozen variables in $\cQ(\widetilde{B})$, the first are marked with circles, while the latter are marked with squares.
	\\
	For example, if the extended exchange matrix is 
	\[
	\widetilde{B} = 
	\begin{pmatrix}
	0 & 1 & 0 & 0 \\
	-1 & 0 & 1 & 0 \\
	0 & -1 & 0 & 1 \\
	0 & 0 & -1 & 0 \\
	\end{pmatrix}
	\]
	then 
	the resulting quiver  $ \cQ(\widetilde{B}) $ is the $ A_4 $-quiver,
	
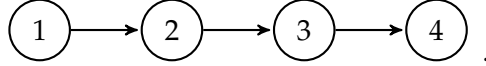
\begin{center} 
		\begin{tikzpicture}[->,>=stealth',shorten >=1pt,auto,node distance=1.75cm, thick,main node/.style={circle,draw, minimum size=8mm},frozen node/.style={rectangle,draw=black, minimum size=8mm}
		]

		\node[main node] (1) {$\, 1 \,$};

		\node[main node] (2) [right of=1] {$\, 2 \,$};

		\node[main node] (k) [right of=2] {$\, 3 \,$};

		\node[main node] (n) [right of=k] {$\, 4 \,$};

		\path 
		(1) edge (2)
		(2) edge (k)
		(k) edge (n);
		
		\end{tikzpicture}
		\ .
	\end{center}
	On the other hand, we obtain for the extended exchange matrix 
	\[
	\widetilde{B} = 
	\begin{pmatrix}
	0 & 2  \\
	-2 & 0 \\
	1 & 1 \\
	0 & - 1 
	\end{pmatrix}
	\] 
	the following quiver $ \cQ(\widetilde{B}) $: 
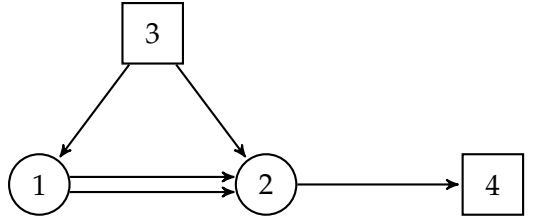
	\begin{center} 
		\begin{tikzpicture}[->,>=stealth',shorten >=1pt,auto,node distance=1.75cm, thick,main node/.style={circle,draw, minimum size=8mm},frozen node/.style={rectangle,draw=black, minimum size=8mm}
		]

		\node[main node] (1) at (0,0) {$ 1 $};
		\node[main node] (2) at (3,0) {$ 2 $};
		\node[frozen node] (3) at (1.5,2) {$ 3 $};
		\node[frozen node] (4) at (6,0) {$ 4 $};
		
		\draw[->] (0.4,0.1) -- (2.62,0.1);
		\draw[->] (0.4,-0.1) -- (2.62,-0.1);
		
		\path 
		(3) edge (1);

		\path 
		(3) edge (2)
		(2) edge (4);

		\end{tikzpicture}
		\ .
	\end{center}
	
	The exchange relation \eqref{eq:exchange_TildeB_1} then translates to	
	\[ 
	x_k x_k' = 	
	\prod\limits_{i \to k } x_i^{b_{ik}} + \prod\limits_{k \to i} x_i^{-b_{ik}} 
	\ .
	\]
	Furthermore, there is also a mutation rule for the quiver $ \cQ(\widetilde{B}) $, which is compatible with the mutation of the matrix described in Definition~\ref{Def_missing}(2). 
	Since this is not relevant for the present article,
	we refer to \cite[Section~2]{BFMS} for more details and examples.
\end{Bem}

\subsection{Presentations of cluster algebras arising from acyclic seeds}
\bsOUT{From now on, we assume that every cluster algebra considered is of geometric type. }

Let us recall a result on finding a presentation of a cluster algebra  $ \cA = \cA(\bs{\Sigma)} \bsOUT{\x,\y,B)} $ without having to determine all cluster variables.
For this, we need to recall the following notions.

\begin{defi}[{\cite[Definitions~1.14]{BFZ2005}}]
	\label{Def:acyclic} 
	Let $ \Sigma = \bs{(\widetilde \x, \widetilde B)} \bsOUT{(\x,\y,B)} $ be a labeled seed $ \cF $.
	
	\begin{enumerate}
		\item 
		The graph $ \Gamma(B) $ is defined as the simply laced directed graph encoding the sign pattern of the matrix $ B = (b_{ij}) $,
		i.e., $ \Gamma(B) $ has $ n $ vertices $ 1 , \ldots, n $ and there is an edge from $ i $ to $ j $ 
		whenever $ b_{ij} > 0  $.	
		Sometimes, we also write $ \Gamma(\Sigma) $ instead of $ \Gamma(B) $.
		
		\item 
		The seed $ \Sigma $ is called {\em acyclic} if there exists no oriented cycle in $ \Gamma(\Sigma) $.
	\end{enumerate}
\end{defi}

The following result is a consequence of \cite[Theorem~1.20/Corollary~1.21]{BFZ2005}.

\begin{Thm}[{\cite{BFZ2005}}]
	\label{Thm:Presentation}
	Let  $ \Sigma = \bs{(\widetilde \x, \widetilde B)} \bsOUT{(\x,\y,B)} $ be an acyclic labeled seed 
	and let $ \cA = \cA(\Sigma) $ be the corresponding cluster algebra. 
	For $ k \in \{ 1, \ldots, n \} $, let $ x_k' $ be the cluster variable which we obtain after mutating $ \Sigma $ in direction $ k $, see \bs{\eqref{eq:exchange_TildeB_1}} \bsOUT{\eqref{eq:exchange_y}},
	and let $ Q_k(\x) $ be the binomial given by the right hand side of the exchange relation \bs{\eqref{eq:exchange_TildeB_1}} \bsOUT{\eqref{eq:exchange_rhs}}.
	\\
	We have:
	\[
		\cA \cong \KK\PP [x_1, \ldots, x_n, x_1', \ldots, x_n']/
		\langle \,
		x_1 x_1' - Q_1(\x), \, \ldots, \, x_n x_n' - Q_n(\x)
		\, \rangle 
		\ .
	\]
	Moreover, $ ( x_1 x_1' - Q_1(\x) , \ldots,  x_n x_n' - Q_n(\x) ) $ is a Gr\"obner basis with respect to any term order in which $ x_1', \ldots, x_n' $ are much more expensive than $ x_1, \ldots, x_n $. 
\end{Thm}

In other words, given an acyclic labeled seed $ \Sigma $, 
we obtain a presentation for the cluster algebra $ \cA(\Sigma) $ by mutating in each direction once. 
This is connected to the notion of the lower bound cluster algebra $ \cL (\Sigma) $ \cite[Definition~1.10]{BFZ2005}, 
which is a lower approximation for $ \cA(\Sigma) $ in general.

The attentive reader may observe that 
there is the extra assumption on the seed to be totally mutable
in \cite[Theorem~1.20/Corllary~1.21]{BFZ2005}.
The reason for this is that \cite[beginning of Subsection~1.1]{BFZ2005} requires the matrix $ B $ in the seed only to be sign-skew-symmetric.
In general, this property is not stable along the mutation rule for the matrix as in Definition~\ref{Def_missing}.
Hence, they define a seed (with sign-skew-symmetric matrix) to be {\em totally mutable} if it admits unlimited mutations in all directions, 
which means that after any number of mutations the resulting matrix is again sign-skew-symmetric matrix. 
\\
By \cite[Proposition~4.5]{FZ2002}, the mutation of a skew-symmetrizable matrix provides again a skew-symmetrizable matrix. 
In particular, the labeled seeds as considered in the present article, i.e., whose exchange matrices are skew-symmetrizable, are always totally mutable.

\medskip 

\subsection{Classification of cluster algebras of finite cluster type}
\label{Subsec:Class_fin}
Two labeled seeds $ \Sigma^{(1)} $ and $ \Sigma^{(2)} $ in $ \cF $ are {\em mutation-equivalent} if there exists a finite sequence of mutations transforming $ \Sigma^{(1)} $ into $ \Sigma^{(2)} $ (up to permutation of the cluster variables).

\begin{defi} 
	\label{Def:fin_type}
	Let $ \Sigma =  \bs{(\widetilde \x, \widetilde B)} \bsOUT{(\x,\y,B)}  $ be a labeled seed in $ \cF $. 
	The corresponding cluster algebra $ \cA(\Sigma) $ is of {\em finite cluster type}
	if the set of seeds is finite.
\end{defi}

In the literature, cluster algebras of finite cluster type are sometimes just called cluster algebras of finite type (e.g.~\cite{FZ2003II,FWZ2017,BFMS}).
In order to make a distinction to the notion of algebras of finite type from commutative algebra, 
we follow the convention to use the expression cluster algebras of finite cluster type. 

By~\cite[Theorem~1.4]{FZ2003II},
cluster algebras of finite cluster type can be identified with Cartan matrices of finite type.
For a brief review on this in the case of trivial coefficients, we refer to \cite[Section~2.1]{BFMS}.
Examples for detailed references are \cite{FZ2003II}, \cite[Chapter~5]{FWZ2017}, or \cite[5.1]{Marsh13}.
\\
On the other hand, Cartan matrices of finite type are classified by the Dynkin diagrams
$ A_{n_1}, B_{n_2}, C_{n_3}, D_{n_4}, E_6, E_7, E_8, F_4, G_2 $, 
where $ n_\ell \geq \ell $ for $ \ell \in \{ 1, 2,3,4\} $,
see \cite[Section~6.4]{Carter2005} or \cite[Theorem~5.2.6]{FWZ2017}.

This allows us to reformulate the classification theorem of cluster algebra of finite cluster type (\cite[Theorem 1.4]{FZ2003II}) in the following way.
For the connection between the exchange matrices and the Dynkin diagrams, 
we refer to \cite[Section~5.2]{FWZ2017}.

\begin{Thm}
	\label{Thm:FinTypClass}
	Let $ \Sigma_t =  \bs{( \widetilde \x_t, \widetilde B_t)} \bsOUT{( \x_t, \y_t, B_t)}   $, $ t \in \TT_n $, be a cluster pattern in $ \cF $
	and let $ \cA $ be the corresponding cluster algebra. 
	Then, $ \cA $ is of finite cluster type 
	if and only if 
	there exists some $ t_0 \in \TT_n $ such that 
	the exchange matrix $ B_{t_0} $
	is one of the matrices in the following list:
	\begin{itemize}
		\item[$(A_n)$] 
		\
		$ 
		\begin{pmatrix} 
		0 & 1 & \\
		-1 & \ddots & \ddots   \\
		& \ddots & 0 & 1  \\
		&  & -1 & 0 &
		\end{pmatrix}  
		$ 
		\ ($ n \times n  $ matrix), 
		\ for $ n \geq 1 $ ;
		
		\bigskip 
		
		\item[$(B_n)$] 
		\
		$ 
		\begin{pmatrix} 
		0 & 1 & \\
		-1 & \ddots & \ddots   \\
		 & \ddots & 0 & 1  \\
		   &  & -1 & 0 & 1  \\
		   &  &  & -2  & 0 \\
		\end{pmatrix}  
		$ 
		\ ($ n \times n  $ matrix), 
		\ for $ n \geq 2 $ ;
		
		\bigskip 
		
		\item[$(C_n)$] 
		\
		$
		\begin{pmatrix} 
		0 & 1 & \\
		-1 & \ddots & \ddots   \\
		& \ddots & 0 & 1  \\
		&  & -1 & 0 & 2  \\
		&  &  & -1  & 0 \\
		\end{pmatrix}  
		$
		\ ($ n \times n  $ matrix),  
		\ for $ n \geq 3 $ ;

		\bigskip 
		
		\item[$(D_n)$] 
		\ 
		$ 
		\begin{pmatrix} 
		0 & 1 & \\
		-1 & \ddots & \ddots   \\
		& \ddots & 0 & 1  \\
		&  & -1 & 0 & 1 & 1\\
		&  &  & -1 & 0 & 0  \\
		&  &  & -1 & 0  & 0 \\
		\end{pmatrix}  
		$ 
		\ ($ n \times n  $ matrix), 
		\ for $ n \geq 4 $ ;

		\bigskip 
		
		\item[$(E_k)$] 
		\ 
		$ 
		\begin{pmatrix} 
		0 	& 1  & \\
		-1 	& \ddots  & \ddots   \\
		& \ddots 	& 0  & 1   \\
		&& -1 &  0 & 1  & 1  & 0 \\
		&&    & -1 & 0  & 0  & 0\\
		&&    & -1 & 0  & 0  & 1  \\
		&&    &  0 & 0  & -1 & 0 \\
		\end{pmatrix}  
		$ 
		\ ($ k \times k $ matrix), 
		\ for $ k \in \{ 6,7,8 \} $ ;

		\bigskip 
		
		\item[$(F_4)$] 
		\ 
		$ 
		\begin{pmatrix} 
		0 & 1 \\ 
		-1 & 0 &1 \\ 
		& -2 & 0 & 1\\  
		& & -1 & 0 
		\end{pmatrix} $ ;
		
		\bigskip
		
		\item[$(G_2)$] 
		\ 
		$ 
		 \begin{pmatrix} 
		 0 & 1 \\ 
		 -3 & 0 
		 \end{pmatrix} $ .
	\end{itemize}
	Here, all non-specified entries are zero.  
\end{Thm}

\begin{Bem}
	\label{Rk:FinTypPresentation}
	Every matrix in the list of Theorem~\ref{Thm:FinTypClass}
	provides an acyclic labeled seed (Definition~\ref{Def:acyclic}).
	In particular, Theorem~\ref{Thm:Presentation} can be applied to find a presentation of the corresponding cluster algebra.
\end{Bem}

\section{Cluster algebra with generic coefficients}
\label{Sec:GenFamily}

In this section, we introduce the main object for our considerations on the singularities of cluster algebras, a generic family from which we may deduce any other cluster algebra of the same type.  
After its definition, we recall the notion of the cluster algebra with principal coefficients associated to a given cluster algebra. 
Along this, we outline some of the differences between those two algebras.

\bsOUT{Starting from this section, we describe all labeled seeds $ \Sigma $ in $ \cF $ in terms of extended exchange matrices
as described in Remark~\ref{Rk:coeff_in_matrix}
since this will simplify our presentation.}

\medskip 

\subsection{Setup}

We fix a labeled seed $ \Sigma = (\widetilde{\x},\widetilde{B}) $ (as always, of geometric type), 
where $ \widetilde{\x}  = (x_1, \ldots, x_m) $, 
$ \widetilde{B} = (b_{ij} ) $ is an $ m \times n $ matrix such that $m \geq n$ are positive integers and the top square matrix $ B = (b_{ij})_{i,j \in \{1, \ldots, n \}} $ is skew-symmetrizable. 
\\
The mutation of $ \Sigma $ in some direction $ k \in \{ 1, \ldots, n \} $ is given by $ \mu_k (\Sigma) = (\widetilde{\x}', \widetilde{B}' ) $,
where
the cluster variables $ \widetilde{\x}' = (x_1', \ldots, x_m') $ 
and the extended exchange matrix $ \widetilde{B}' = (b_{ij}') $ obey the following transformation rules 
(
see \eqref{eq:exchange_TildeB_1} and {\cite[Definition~2.4]{FZ2007}})
\begin{equation}
	\label{eq:mutate_B}
		b_{ij}' := 
		\begin{cases}
		- b_{ij} & \mbox{if } i = k \mbox{ or } j = k \ ,
		\\
		b_{ij} + \operatorname{sgn}(b_{ik}) [ b_{ik} b_{kj} ]_+
		& \mbox{otherwise} \ ,
	\end{cases}
\end{equation}
\begin{equation}
	x_j' := x_j
	\ \ \ \mbox{ if } j \in \{ 1 , \ldots, m \} \setminus \{ k \} 
	\ ,
\end{equation}
\begin{equation}
\label{eq:exchange_TildeB}
x_k x_k' = \prod\limits_{i=1}^m x_i^{[b_{ik}]_+} + \prod\limits_{i=1}^m x_i^{[-b_{ik}]_+} 
= 
\prod\limits_{i:b_{ik} > 0 } x_i^{b_{ik}} + \prod\limits_{i: b_{ik}<0} x_i^{-b_{ik}} 
\ .
\end{equation}
The cluster algebra $ \cA = \cA(\Sigma) $ corresponding to $ \Sigma $ is obtained via repeated mutation in all directions $ k \in \{ 1, \ldots, n \} $ (Definition~\ref{Def:ClusterAlgebra}).
Moreover, if $ \Sigma $ is acyclic (Definition~\ref{Def:acyclic}),
Theorem~\ref{Thm:Presentation} provides the presentation
\begin{equation}
	\label{eq:presentation}
	\cA \cong \KK\PP [x_1, \ldots, x_n, x_1', \ldots, x_n']/
	\langle \,
	x_1 x_1' - Q_1(\widetilde{\x}), \, \ldots, \, x_n x_n' - Q_n(\widetilde{\x})
	\, \rangle 
	\ ,
\end{equation} 
where $ Q_k(\widetilde{\x}) := 
\prod_{b_{ik} > 0 } x_i^{b_{ik}} + \prod_{b_{ik}<0} x_i^{-b_{ik}}  $
and $ \PP = \operatorname{Trop}(x_{n+1}, \ldots , x_m ) $.

\subsection{Cluster algebra with generic coefficients}
\label{Subsec:gen_coeff}
\begin{defi}
	\label{Def:gen_coeff}
	Let $ \Sigma = (\widetilde{\x}, \widetilde{B}) $ be a labeled seed as above.
	Let $ \t = (t_1, \ldots, t_n) $ and $ \s = (s_1, \ldots, s_n) $ be \bs{additional algebraically} independent \bs{variables.} \bsOUT{parameters with values $ t_k, s_k \in \KK $, 
	for $ k \in \{ 1, \ldots, n \} $.} 
	We define the {\em cluster algebra with generic coefficients} $ \cA^\gen := \cA^\gen_{\s,\t} :=  \cA^\gen_{\s,\t} (\Sigma) $ associated to $ \Sigma $ 
	as the cluster algebra given by the labeled seed $ \Sigma^\gen := \Sigma^\gen_{\s,\t} := (\widetilde{\x}^\gen, \widetilde{B}^\gen) $,
	where 
	\[
		\widetilde{\x}^\gen := (\x,\s,\t) = (x_1, \ldots, x_m, s_1, \ldots, s_n, t_1, \ldots, t_n) 
	\] 
	and $ \widetilde{B}^\gen $ is the $ (m+2n) \times n $ matrix
	which we obtain by expanding $ \widetilde{B} $ by the $ n \times n $ unit matrix $ I_n $ as well as its additive inverse $ -I_n $ as additional rows, 
	\[
		\widetilde{B}^\gen 
		:= \begin{pmatrix}
		\widetilde{B} \\ I_n \\ -I_n
		\end{pmatrix}
		\ .
	\] 
\end{defi}

We often drop the reference to the \bsOUT{parameters} \bs{coefficients} $ \s, \t $, 
if there is no confusion possible. 

\begin{Bem}
	\phantomsection \label{Rk:gen_a_lot}
	\begin{enumerate}
		\item 
		By Theorem~\ref{Thm:Presentation}, 
		if $ \Sigma $ is acyclic, then
		we have that $ \cA^\gen_{\s,\t} $ is isomorphic to
		\[
			\KK\PP [\x, x_1', \ldots, x_n']/
			\langle \,
			x_k x_k' - s_k \prod\limits_{b_{ik} > 0 } x_i^{b_{ik}} - t_k \prod\limits_{b_{ik}<0} x_i^{-b_{ik}} 
			\mid k \in \{ 1, \ldots, n \} 
			\, \rangle 
			\ .
		\]
		We observe that if $ s_k = t_k = 1 $ for all $ k \in \{ 1, \ldots, n \} $, 
		we regain the original cluster algebra $ \cA = \cA (\Sigma) $.
		
		\smallskip  
		
		\item 
		Let us stay in the acyclic setting.
		If we denote by $ B $ the skew-symmetrizable $ n \times n $ matrix on top of $ \widetilde{B} $, 
		then we see that 
		it is possible to obtain $ \cA (\widetilde{\x}, \widetilde{B}) $ from 
		$ \cA^\gen_{\s,\t} (\x, B) $ by substituting
		\begin{equation}
		\label{eq:specialization} 
			s_k = \prod\limits_{i > n  : b_{ik} > 0 } x_i^{b_{ik}}
			\ \ \ \mbox{ and } \ \ \ 
			t_k = \prod\limits_{i > n  : b_{ik} < 0 } x_i^{-b_{ik}} 
			\ .
		\end{equation}
		Because of this connection, we choose the name cluster algebra with generic coefficients. Let us point out that this substitution is not a coefficient specialization in general.
		\\
		Furthermore, 
		since all seeds which we consider starting from Section~\ref{Sec:FiniteSing} are acyclic, this allows us to reduce our considerations to the case where we only have generic coefficients, 
		i.e., where we start with a seed $ (\widetilde{\x}, \widetilde{B}) = (\x, B) $ for $ \x = (x_1, \ldots, x_n) $ all mutable and $ B $ a skew-symmetrizable $ n \times n $ matrix.  
		
		\smallskip

		\item 
		The idea for the cluster algebra with generic coefficients can already be found in a special case in \cite[Example~2.4]{FZ2002} with $ (\s, \t) $ taking the role of generators for the tropical semi-field $ \PP $.
		\\
		As we will see later in Lemma~\ref{Lem:gen=prin}
		for acyclic seeds, 
		if we restrict the \bsOUT{parameters} \bs{coefficients} to the torus $ (\KK^\times)^{2n} $, 
		then  
		$ \bs{\Spec ( \cA^\gen_{\s,\t} ) } $ can be identified with a trivial family over $ (\KK^\times)^n $, where each fiber is \bs{the spectrum of} the cluster algebra with principal coefficients $ \bs{\Spec ( \cA^\prin )} $ associated to $ \widetilde{B} $.
		We recall the definition of $ \cA^\prin $ in Definition~\ref{Def:prin} below.
		
		\smallskip  
		
		\item 
		If we have a seed $ (\widetilde{\x}, \widetilde{B}) $, whose exchange matrix $ B $ is skew-symmetric, 
		the seed corresponds to a quiver $ \cQ  $ with mutable and frozen vertices, as explained in Remark~\ref{Rk:Quiver_view}.
		The construction of the cluster algebra with generic coefficients translates in this setting to defining a new quiver $ \cQ^\gen $
		which is obtained as follows from $ \cQ $:
		For every mutable vertex $ k $ of $ \cQ $, 
		we introduce two new frozen vertices, say $ k_s $ and $ k_t $, 
		with the property that there is a single arrow going from $ k_s $ to $ k $ and another one going from $ k $ to $ k_t $. 
		The corresponding matrix coincides with $ \widetilde{B}^\gen $ of Definition~\ref{Def:gen_coeff}. 
		\\
		For example, if $ \cQ $ is the $ A_4 $-quiver (see Remark~\ref{Rk:Quiver_view}), then $ \cQ^\gen $ is pictured as:  
		\begin{center} 
			\begin{tikzpicture}[->,>=stealth',shorten >=1pt,auto,node distance=1.75cm, thick,main node/.style={circle,draw, minimum size=8mm},frozen node/.style={rectangle,draw=black, minimum size=8mm}
			]

			\node[main node] (1) {$\, 1 \,$};
			\node[frozen node] (s1) [above of=1] {$\, 1_s \,$};
			\node[frozen node] (t1) [below of=1] {$\, 1_t \,$};

			\node[main node] (2) [right of=1] {$\, 2 \,$};
			\node[frozen node] (s2) [above of=2] {$\, 2_s \,$};
			\node[frozen node] (t2) [below of=2] {$\, 2_t \,$};

			\node[main node] (k) [right of=2] {$\, 3 \,$};
			\node[frozen node] (sk) [above of=k] {$\, 3_s \,$};
			\node[frozen node] (tk) [below of=k] {$\, 3_t \,$};

			\node[main node] (n) [right of=k] {$\, 4 \,$};
			\node[frozen node] (sn) [above of=n] {$\, 4_s \,$};
			\node[frozen node] (tn) [below of=n] {$\, 4_t \,$};

			\path
			(s1) edge (1)
			(1) edge (t1);
			
			\path 
			(s2) edge (2)
			(2) edge (t2);

			\path 
			(sk) edge (k)
			(k) edge (tk);

			\path 
			(sn) edge (n)
			(n) edge (tn);

			\path 
			(1) edge (2)
			(2) edge (k)
			(k) edge (n);
			
			\end{tikzpicture}
		\end{center} 
	This also explains our choice for the names of the \bsOUT{parameters} \bs{coefficients} since $ k_s $ corresponds to a new arrow with a frozen vertex as \underline{s}ource,
	while $ k_t $ is associated to an arrow whose \underline{t}arget is frozen. 
	\end{enumerate}
\end{Bem}

\subsection{Cluster algebra with principal coefficients and its connection to the cluster algebra with generic coefficients}

We recall the notion of cluster algebras with principal coefficients and 
relate them to cluster algebras with generic coefficients.
As in~\cite[Section~4.1]{INT_complexes},
we consider the slightly more general case,
where we start with an extended exchange matrix $ \widetilde{B} $ 
instead of only an exchange matrix $ B  $ as in~\cite[Definition~3.1 and Remark~3.2]{FZ2007}. 
\bsOUT{The motivation for this is the perspective as family analogous to Remark~\ref{Rk:gen_a_lot}(4).} 

\begin{defi}[]
	\label{Def:prin}
	Let $ \Sigma = ( \widetilde{\x}, \widetilde{B} ) $ be a labeled seed,
	where $ \widetilde{B} $ has size $ m \times n $ with $m \geq n$. 
	Define 
	\[ 
		\Sigma^\prin := ( \widetilde{\x}^\prin ,\widetilde{B}^\prin) 
	\]
	to be the labeled seed
	with cluster variables 
	$ \widetilde{\x}^\prin  = (\widetilde{x}, c_1, \ldots, c_n ) $ for new frozen variables $ (c_1, \ldots, c_n) $
	and 
	with extended exchange matrix
	$ \widetilde{B}^\prin $ the $ (m+n) \times n $ matrix obtained by expanding $ \widetilde{B} $ by the $ n \times n $ unit matrix $ I_n $ as additional rows,
	\[
		\widetilde{B}^\prin := \begin{pmatrix}
		\widetilde{B} \\ I_n
		\end{pmatrix}.
	\]  
	The corresponding cluster algebra $ \cA^\prin := \cA^\prin (\Sigma ) := \cA (\Sigma^\prin)  $ is called the {\em cluster algebra with principal coefficients} associated to $ \Sigma $.	
\end{defi}

Sometimes we also write $ \cA^\prin_\c $ or $ \cA^\prin_\c (\Sigma ) $ if we want to emphasize the role of the new frozen variables $ \c = (c_1, \ldots, c_n) $ as \bsOUT{parameters} \bs{coefficients}.

In the next lemma, we show how the cluster algebra with generic resp.~principal coefficients may be identified for acyclic labeled seeds 
if we restrict their \bsOUT{parameter} \bs{coefficient} spaces to the torus, 
where all \bsOUT{parameters} \bs{coefficients} are invertible.

\begin{lemma}
	\label{Lem:gen=prin}
	Let $ \Sigma $ be an acyclic labeled seed. 
	\bsOUT{If} 
	\bs{Since} we assume that the 
	\bsOUT{parameters} 
	\bs{coefficients}
	$ (\s,\t) $ of $ \cA^\gen_{\s,\t}(\Sigma) $ and
	\bsOUT{the parameters} 
	$ \c $ of 
	$ \cA^\prin_\c (\Sigma) $ are all invertible,
	\bsOUT{then} 
	there exists an isomorphism
	\[ 
		\bs{\Spec} ( \cA^\gen_{\s,\t}(\Sigma) ) \cong \bs{\Spec} (\cA^\prin_\c (\Sigma)) \times (\KK^\times)^n  ,
	\]  
	\bs{where, on the level of rings, the isomorphism is induced by the morphism 
		\[
			\KK[\c^\pmo][\x, y_1, \ldots, y_n] \otimes \KK[\t^\pmo]
			\rightarrow
			\KK[\s^\pmo, \t^\pmo] [\x, x_1', \ldots, x_n']
		\]
		determined by extending the map $ \KK[\x] \otimes \KK[\t^\pmo] \to \KK[\s^\pmo, \t^\pmo] [\x, x_1', \ldots, x_n'] , f\otimes g \mapsto f \cdot g $
		through
		$ y_k \otimes 1 \mapsto t_k^{-1} x_k' $
		and
		$ c_k \otimes 1 \mapsto t_k^{-1} s_k $
		for $ k \in \{ 1, \ldots, n \} $. 
	}
\end{lemma}

\begin{proof}
	Since $ \Sigma $ is acyclic, 
	Theorem~\ref{Thm:Presentation}  
	provides the presentations 
	\[
	\cA^\gen_{\s,\t} (\Sigma)
	\cong 
	\KK_{\s,\t} [\x, x_1', \ldots, x_n']/
	\langle \,
	x_k x_k' - s_k \prod\limits_{b_{ik} > 0 } x_i^{b_{ik}} - t_k \prod\limits_{b_{ik}<0} x_i^{-b_{ik}} 
	\mid k \in \{ 1, \ldots, n \} 
	\, \rangle 
	\]
	and
	\[
	\cA^\prin_\c (\Sigma)
	\cong 
	\KK_\c [\x, y_1, \ldots, y_n]/
	\langle \,
	x_k y_k - c_k \prod\limits_{b_{ik} > 0 } x_i^{b_{ik}} - \prod\limits_{b_{ik}<0} x_i^{-b_{ik}} 
	\mid k \in \{ 1, \ldots, n \} 
	\, \rangle 
	\ ,
	\]
	where we abbreviate $ \KK_{\s,\t} := \KK[\s^\pmo, \t^\pmo] $ and $ \KK_\c :=  \KK[\c^\pmo] $. 
	
	Consider $ \cA^\gen_{\s,\t} (\Sigma) $.
	Since $ s_k $ and $ t_k $ are invertible, for $ k \in \{ 1, \ldots, n \} $,
	we may multiply 
	the exchange relation
	$ x_k x_k' - s_k \prod\limits_{b_{ik} > 0 } x_i^{b_{ik}} - t_k \prod\limits_{b_{ik}<0} x_i^{-b_{ik}}  = 0
	$ by $ t_k^{-1} $.
	Afterwards, we introduce the new \bsOUT{variable} \bs{coefficients} $ y_k' := t_k^{-1} x_k' $ 
	and \bsOUT{the new variable} $ c_k' := s_k t_k^{-1} $ \bs{for $ k \in \{ 1, \ldots, n \} $}. 
	Thus, the relation becomes
	\[
		x_k y_k' - c_k' \prod\limits_{b_{ik} > 0 } x_i^{b_{ik}} - \prod\limits_{b_{ik}<0} x_i^{-b_{ik}} = 0
		\ .
	\]
	As this holds for every $ k \in \{ 1, \ldots, n \} $, 
	we obtain precisely the exchange relations defining $ \cA^{\prin}_\c (\Sigma) $
	\bs{by} identifying $ y_k' = y_k $ and $ c_k' = c_k $.
	The assertion follows.
\end{proof}

\begin{Bem}[Cluster algebras with universal coefficients]
	\label{Rk:Univ} 
	Fixing a labeled seed $ (\widetilde{\x}, \widetilde{B}) $ in $ \cF $, 
	the Laurent phenomenon \cite[Theorem~3.1]{FZ2002} states
	that every cluster variable obtained by iterated mutation 
	can be expressed as a Laurent polynomial in $ \widetilde{\x} $ with coefficients in $ \ZZ\PP $.
	\\
	Using the notion of cluster algebras with principal coefficients, 
	it is possible to introduce the cluster algebra with universal coefficients $ \cA^{\univ} (\Sigma) = \cA ( \Sigma^\univ ) $ \cite[Section~12]{FZ2007}.
	The latter is universal in the sense that any cluster algebra of the same mutation type as $ \Sigma $ can be obtained from $ \cA^{\univ} (\Sigma) $  via a coefficient specialization. 
	For more details, we refer to the literature,
	e.g.~\cite{FZ2007,Read}. 
	\\
	Let us only briefly recall that the extended exchange matrix of $ \Sigma^\univ $ is obtained by expanding $ \widetilde{B} $ with suitable rows.
	The latter are coming from the 
	$ g $-vectors of the cluster algebra with principal coefficients associated to 
	the transpose $ B^T $ of the exchange matrix.
	Here, the $ g $-vectors are determined by the expression of the cluster variables in terms of the initial variables $ \widetilde{\x} $.
	\\
	Notice that the number of $ g $-vectors is equal to the number of cluster variables.
	Hence, we only obtain polynomial exchange relations for $ \cA^{\univ} (\Sigma) $ 
	if $ \cA(\Sigma) $ is of finite cluster type,
	see also Remark~\ref{Rk:non-polynomial}. 
	
	Suppose that $ \cA(\Sigma) $ is of finite cluster type.
	As discussed in Remark~\ref{Rk:gen_a_lot}(2) \bsOUT{and in particular \eqref{eq:specialization}}, 
	it is possible to obtain $ \cA^{\univ} (\Sigma) $ from $ \cA^{\gen}_{\s,\t} (\Sigma) $ \bs{as a commutative algebra via \eqref{eq:specialization}}.
	Hence, a reader interested in understanding the cluster algebra with universal coefficients may take \eqref{eq:specialization} as definition of an abbreviation to simplify computations
	so that one has to perform a substitution in the final results in order to get the counterpart for universal coefficients. Once again, we remind the reader that this substitution is not a coefficient specialization, in general.
	\\
	Nonetheless, let us provide the warning that the substitution requires to determine all $ g $-vectors of a given cluster algebra of finite cluster type, which is not an easy task in general. 
	For example, see \cite[Section~9]{Read}, where the rank-2 case is discussed.
\end{Bem}

\section{Classification of the singularities in the finite cluster type case}
\label{Sec:FiniteSing}

We come to the core of the present work.
Via a case-by-case study, we classify the singularities of cluster algebras of finite cluster type with generic coefficients.
Following Remark~\ref{Rk:gen_a_lot}(2), we start with labeled seeds of the form $ \Sigma = (\x,B) $,
where $ \x = (x_1, \ldots, x_n) $ and $ B $ is a skew-symmetrizable $ n \times n $-matrix.
\\
As explained before, we stick to the situation where the generic coefficients are invertible.  
Since all labeled seeds corresponding to the matrices listed in the classification of finite cluster type (Theorem~\ref{Thm:FinTypClass}) are acyclic, we can apply Theorem~\ref{Thm:Presentation} to write down presentations of the associated cluster algebras. 
Furthermore, since we assume the coefficients to be invertible, 
Lemma~\ref{Lem:gen=prin} provides that we can restrict ourselves to cluster algebras with principal coefficients.
The extension to generic coefficients of the results proven for principal coefficients in this section is straightforward and thus,
	we obtain Theorems~\ref{Thm:A} -- \ref{Thm:E}.
\\
In conclusion, the algebra associated to $  \Sigma $ which we are working with is of the form 
\[ 
	\cA^\prin_\c(\Sigma) \cong
	\KK_\c [x_1, \ldots, x_n,  y_1, \ldots, y_n]/
	\langle \,
	x_k y_k - c_k \prod\limits_{b_{ik} > 0 } x_i^{b_{ik}} - \prod\limits_{b_{ik}<0} x_i^{-b_{ik}} 
	\mid k \in \{ 1, \ldots, n \} 
	\, \rangle 
	\ ,
\]
where we use the abbreviation $ \KK_\c :=  \KK[\c^\pmo] = \KK[c_1, c_1^{-1}, \ldots, c_n, c_n^{-1}] $.

\subsection{Facts about continuants polynomials}
\label{Subsec:Cont}
As in \cite{BFMS}, continuant polynomials are the key tool for us to deduce a new presentation of $ \cA^\prin_\c(\Sigma) $ which is more suitable for the study of the singularities. 
Continuant polynomials have recently appeared in other work related to cluster algebras, e.g.~in \cite{MG-Ovsienko, LeclereMorierGenoud}.
\\
Let us recall some facts on continuant polynomials from \cite[Section~3]{BFMS}. 
This also includes results whose proofs can be found in \cite{Muir}.

\begin{defi}
	Let $ n \in \ZZ_+ $.
	The determinant of a tri-diagonal matrix
	\[
		\begin{pmatrix}
		y_1 & b_1 & \\
		c_1 & y_2 & b_2 \\
			& \ddots & \ddots & \ddots \\
			&	& c_{n-2} & y_{n-1} & b_{n-1} \\
			&	&	& c_{n-1} & y_n
		\end{pmatrix}
	\]
	(where all non-specified entries are zero)
	is called a {\em continuant} of order $ n $.
	
	We use the notation $ P_n (y_1, \ldots, y_n) $
	for the continuant which we obtain if $ b_i = c_i = 1 $ for all $ i \in \{ 1, \ldots, n- 1 \} $.
	We make the convention to put $ P_0 := 1 $. 
\end{defi}

The particular continuants $ P_n (y_1, \ldots, y_n) $ are those which are relevant for our work here.
Hence, we restrict our attention to them even though several of the mentioned results may be proven in a more general setting.

One verifies that $ P_1 (y_1) = y_1, 
P_2 (y_1, y_2) = y_1 y_2 - 1, P_3 (y_1,y_2, y_3) = y_1 y_2 y_3 - y_1 - y_3 $.

The following statements follow from the definition of the continuant respectively can be found in 
\cite[Numbers~545, 547(3), 561(4)]{Muir}.

\begin{lemma}
	\phantomsection
	\label{Lem:ContFacts}
	\begin{enumerate}
		\item 
		All terms of $ P_n (y_1, \ldots, y_n) $
		can be obtained from the monomial  
		$ y_1 \cdots y_n $ 
		by replacing every pair of consecutive $ y_i $ by $ - 1 $.
		\\ 
		This implies that $ P_n (y_1, \ldots, y_n)_{\leq 2 } $, i.e., the terms of order at most $ 2 $ of $ P_n (y_1, \ldots, y_n) $
		are of the following form, for $ k \in \ZZ_+ $:

		\begin{enumerate}
			\item 
			$\displaystyle P_{4k}(y_1, \ldots, y_{4k})_{\leq 2} = 
			1 - \sum_{\ell=1}^{2k} \sum_{m=\ell}^{2k}
			y_{2\ell - 1} y_{2m} =
			$
			
			\medskip 
			
			$ = 1 - y_1 y_2 - y_1 y_4 - \cdots - y_1 y_{4k}
			- y_3 y_4 - \cdots  y_3 y_4 - \cdots - y_{4k-1} y_{4k} $ \ ,
			
			\medskip 
			\smallskip 
			
			\item 
			$ \displaystyle P_{4k+1}(y_1, \ldots, y_{4k+1})_{\leq 2 } 
			=  \sum_{\ell=1}^{2k+1} y_{2\ell - 1} 
			= y_1 + y_3 + \cdots + y_{4k+1} $ \ ,

			\medskip 
			\smallskip 
				
			\item 
			$ \displaystyle P_{4k+2}(y_1, \ldots, y_{4k+2})_{\leq 2 } 
			= -1 + \sum_{\ell=1}^{2k+1} \sum_{m=\ell}^{2k+1} y_{2\ell-1} y_{2m} = $
			
			\medskip

			$ = -1 + y_1 y_2 + y_1 y_4 + \cdots + y_1 y_{4k+2} + y_3 y_4 + \cdots + y_{4k+1} y_{4k+2} $ \ ,

			\medskip 
			\smallskip 
			
			\item 
			$ \displaystyle P_{4k+3}(y_1, \ldots, y_{4k+3})_{\leq 2 } 
			=  - \sum_{\ell=1}^{2k+2} y_{2\ell - 1} 
			= - y_1 - y_3 - \cdots - y_{4k+3} $ \ .

		\end{enumerate} 
		 
		\medskip 
		 
		\item 
		The continuant polynomial $ P_n(y_1, \ldots, y_n) $ is symmetric,
		\[ 
			P_n (y_1, \ldots, y_n) = P_n (y_n, \ldots, y_1 ) 
			\ . 
		\] 
		
		\medskip 
		
		\item 
		For $ k \in \{ 1, \ldots, n -1 \} $, 
		the recursion 
		\[
			P_n(y_1, \ldots, y_n) 
			= P_k (y_1, \ldots, y_k) P_{n-k} (y_{k+1}, \ldots, y_n) 
			- 
			P_{k-1} (y_1, \ldots, y_{k-1}) P_{n-k-1} (y_{k+2}, \ldots, y_n) 
		\]
		holds.
		In particular, we get for $ k = 1 $,
		\[ 
			P_n(y_1, \ldots, y_n) = y_1 P_{n-1} (y_2, \ldots, y_n)  -  P_{n-2}(y_3, \ldots, y_n) 
		\]
		
		\medskip 
		
		\item 
		The derivatives are,
		for $ k \in \{ 1, \ldots, n \} $,
		\[
			\frac{\partial P_n(y_1, \ldots, y_n)}{\partial y_k}
			=
			P_{k-1} (y_1, \ldots, y_{k-1}) P_{n-k} (y_{k+1}, \ldots, y_n) 
		\] 
		In particular, we have for $ k = 1 $,
		\[
			\frac{\partial P_n(y_1, \ldots, y_n)}{\partial y_1}
			=
			P_{n-1} (y_{2}, \ldots, y_n) 
		\] 
	\end{enumerate}
\end{lemma}

From the perspective of singularity theory,
the following result is proven in \cite{BFMS}.

\begin{prop}[{\cite[cf.~Lemma~3.7 and Proposition~3.8]{BFMS}}] 
	\label{Prop:ContSing}
	\bs{Let $ \kappa $ be any field.}
	For $ n \in \ZZ_+ $ and a parameter $ \lambda $ taking values in $ \bs{\kappa} $,
	we define the following family $ X_{n,\lambda} $ of varieties over $ \bs{\kappa} $
	\[
		X_{n, \lambda } := \Spec ( \KK [y_1, \ldots, y_n]/ \langle P_n(y_1, \ldots, y_n) + \lambda \rangle ). 
	\]
	We have that  $ X_{n,\lambda} $ is singular if and only if $ n = 2m $ is even and $ \lambda = (-1)^{m+1} $.  
	\\
	Moreover, if $ X_{n,\lambda} $ is singular, then it is has only an isolated singularity of type $ A_1 $ at the origin. 
	In particular, the singularities are resolved \bs{by} the blowup with center the origin. 
\end{prop}

While the above statement is characteristic-free,
we warn the reader that the condition $ \lambda = (-1)^{m+1} $ varies with the characteristic.

\begin{Obs} 
	We consider the family 
	$ 
		\phi \colon  X_{n,\lambda} \to S
	$ 
	over the base $ S := \Spec(\bs{\kappa}[\lambda]) = \AA_{\bs{\kappa}}^1 $.
	Proposition~\ref{Prop:ContSing} implies that there is a stratification $ S = S_0 \sqcup S_1 $,
	where
	\[
		S_0 := \AA_{\bs{\kappa}}^1 \setminus V(\lambda - (-1)^{m+1} ) 
		\ \ \ 
		\mbox{ and } 
		\ \ \ 
		S_1 := V(\lambda - (-1)^{m+1} )
		\ ,
	\]
	such that the singularity type of the fibers of the family given by $\phi$ is
fixed along $ S_1 $ and $ S_0 $, respectively. 
\end{Obs}

\subsection{Type $ \boldsymbol{A_n} $}
\label{Subsec:An}
	We begin our study with the singularities of cluster algebras of finite cluster type $ A_n $.
	
	\begin{Obs}
		\label{Obs:Special_Type_A1}
		\bs{For $ n = 1 $, we have $ \cA^\prin_\c(A_1) \cong \KK_{c_1} [x_1, y_1] / \langle x_1 y_1 - c_1 - 1 \rangle $ by \eqref{eq:An_exchange}. 
			We observe that 
			$ x_1 y_1 - c_1 - 1 = P_2 (x_1, y_1) - c_1 $.
		For the family 
	$ \phi \colon \Spec ( \cA^\prin_{\c}(A_1)) \to S $
	a closed point $ \eta \in S $ is singular if and only if the image of $ c_1+1 $ is zero in $ \kappa(\eta) $. 
	In the singular case, the fiber is $ V (x_1 y_1) \subseteq \AA_{\kappa(\eta)}^2 $ which is a union of two lines meeting transversally in the unique intersection point.
	Blowing up the intersection point resolves the singularities.
} 
	\end{Obs}

	Fix $ n \geq \bs{2} $. 
	Following Theorem~\ref{Thm:FinTypClass}, 
	we choose in the initial labeled seed $ \Sigma = (\x, B ) $
	the $ n \times n $ matrix
	\[
	B =  
	\begin{pmatrix} 
	0 & 1 & \\
	-1 & \ddots & \ddots   \\
	& \ddots & 0 & 1  \\
	&  & -1 & 0 &
	\end{pmatrix} 
	\ . 
	\]
	The corresponding cluster algebra with principal coefficients $ \c = (c_1, \ldots, c_n ) $ (with values in $ \KK^\times $) 
	is 
	\begin{equation}
	\label{eq:An_exchange} 
		\cA^\prin_\c(\Sigma) \cong
		\KK_\c [\x, \y]
		\hspace{-0.15cm}\raisebox{-0.2cm}{$\bigg/$}
		\hspace{-0.2cm}\raisebox{-0.5cm}%
		{$\Bigg\langle \,
		\begin{minipage}{3.4cm}
		$ x_1 y_1 - c_1 -  x_2 
		\\[2pt]
		x_k y_k - c_k x_{k-1} -  x_{k+1} 
		\\[2pt]
		x_n y_n - c_n x_{n-1} -  1 $
		
		\end{minipage}
		\, \Bigg| \,\, k \in \{ 2, \ldots, n-1 \} 
		\, \Bigg\rangle
		\ ,
		$} 
	\end{equation} 	
	where $ \x = (x_1, \ldots, x_n ) $ and $ \y = (y_1, \ldots, y_n ) $.
	
	We often also write $ \cA^\prin_\c(A_n)  $ instead of $ \cA^\prin_\c(\Sigma) $ in order to emphasize that we are considering the finite cluster type $ A_n $.
	
	Following the strategy applied in \cite{BFMS},
	we first determine a new presentation
	which has \bs{fewer} relations than the one just given.
	This will reveal a close connection to the case of trivial coefficients and in particular to the singularity theory of continuant polynomials as discussed in Proposition~\ref{Prop:ContSing}. 
	
	Recall that we introduced in Definition~\ref{Def:ContPoly_lamda_n}(2) the expressions
	\[ 
		\lambda_{2s}(\c)
		=			
		\displaystyle 
		\prod_{\alpha = 1}^{s} c_{2 \alpha}^{-1} 
		\ \ \
		\mbox{ and  }  
		\ \ \
		\lambda_{2s+1}(\c)
		=	
		\displaystyle 
		\prod_{\alpha = 1}^{s+1} c_{2 \alpha -1 }^{-1}
		\ ,
		\ \ \
		\mbox{ for } \c = (c_1, \ldots, c_n) \mbox{ invertible} \ .
	\] 
	
	\begin{prop}
		\label{Prop:An_hypersurface} 
		\bs{Let $ n \geq 2 $.} 
		The cluster algebra $ \cA^\prin_\c(A_n) $
		is isomorphic to $ \KK_\c[z_1, \ldots, z_{n+1}]/ \langle f_n \rangle $
		\bs{(indeed, it is an isomorphism over $ \KK_\c$)}
		with 
		\[
			f_n(z_1, \ldots, z_{n+1})
			:= 
			P_{n+1} (z_1, \ldots, z_{n+1}) - \lambda_n(\c)
			\ ,
		\]
		where $P_{n+1}(z_1, \ldots, z_{n-1}) $ is the continuant polynomial defined in Section~\ref{Subsec:Cont}.
	\end{prop}
	
	\begin{proof}
		\bsOUT{For $ n = 1 $, we have $ \cA^\prin_\c(A_1) \cong \KK_{c_1} [x_1, y_1] / \langle x_1 y_1 - c_1 - 1 \rangle $ by \eqref{eq:An_exchange}. 
		We observe that 
		$ x_1 y_1 - c_1 - 1 = P_2 (x_1, y_1) - c_1 $ 
		and the assertion follows in this special case.}

		\bsOUT{Let $ n \geq 2 $.} 
		Consider the presentation \eqref{eq:An_exchange} of $ \cA^\prin_\c(A_n) $.
		The exchange relations are 
		\[ 
		x_1 y_1 - c_1  - x_2 
		\ = \   
		x_k y_k -  c_k x_{k-1} -  x_{k+1} 
		\ = \ 
		x_n y_n - c_n x_{n-1} - 1  
		\ =  \ 0
		\ ,  
		\]
		where $ k \in \{ 2, \ldots, n-1 \} $.
		Since the coefficients $ \c = (c_1, \ldots, c_n) $ are invertible, 
		we may introduce the change of variables 
		\begin{equation}
			\label{eq:new_var} 
			\left\{ \ \ \ 
		\begin{array}{ll} 
			\displaystyle 
			\widetilde x_{2 \ell}  := 
			x_{2\ell} \prod_{\alpha = 1}^\ell c_{2 \alpha -1 }^{-1}
			\ ,   
			
			&
			
			\displaystyle
			\widetilde x_{2\ell - 1} := 
			x_{2\ell - 1} \prod_{\alpha = 1}^{\ell-1} c_{2 \alpha}^{-1} 
			\ ,

			\\[15pt] 
			
			\displaystyle 
			\widetilde y_{2 \ell}  := 
			y_{2\ell} \prod_{\alpha = 1}^\ell c_{2 \alpha -1 }  c_{2 \alpha}^{-1}
			\ ,   
			
			\hspace{15pt}
			
			&
			
			\displaystyle
			\widetilde y_{2\ell - 1} := 
			y_{2\ell - 1} \prod_{\alpha = 1}^{\ell} c_{2 \alpha -1 }^{-1}  \prod_{\alpha = 1}^{\ell-1}  c_{2 \alpha}
			\ ,
			
		\end{array}
		\right.  
		\end{equation} 
		for $ 2\ell, 2\ell-1 \in \{ 1, \ldots, n \} $.
		Note that $ \widetilde x_1 = x_1 $ and $ \widetilde y_1 = y_1 c_1^{-1} $. 
		In Remark~\ref{Rk:new_var}(1), we provide the idea how this change of variables arises. 
		\\
		By substitution, we obtain
		\begin{equation} 
			\label{eq:An_afterchange} 
			\left\{			
			\ 
			\begin{array}{ll} 
				
				x_1 y_1 - c_1 - x_2  =
				c_1 ( \widetilde x_1 \widetilde y_1  - 1  - \widetilde x_2 )
				
				\\[15pt]
				
				x_k y_k - c_k x_{k-1} -  x_{k+1} =
				\lambda_k(\c)^{-1} 
				\big( 
				\widetilde x_{k} \widetilde y_{k}  
				- \widetilde x_{k-1}  
				- \widetilde x_{k+1} 
				\big) 
				\ , 
				
				&
				
				 k \in \{ 2, \ldots, n-1 \}

				\\[15pt]

				x_n y_n - c_n x_{n-1} - 1  
				= 
				\lambda_n(\c)^{-1}
				\bigg( 
					\widetilde x_{n} \widetilde y_{n}  
					- \widetilde x_{n - 1 } 
					- \lambda_n(\c)  
				\bigg) 
								
			\end{array}
		\right.   
		\end{equation} 
		where $ \lambda_n (\c) $ is the term defined in the statement of the proposition,
		\[   
		\lambda_k(\c)^{-1}
		= 
		\begin{cases} 
			\displaystyle 
			\prod_{\alpha = 1}^{\ell} c_{2 \alpha}
			\ ,
			& 
			\mbox{ if } k = 2\ell \ , 
			\\[15pt]
			\displaystyle 
			\prod_{\alpha = 1}^{\ell} c_{2 \alpha-1}
			\ ,
			& 
			\mbox{ if } k = 2\ell - 1 \ . 
			
		\end{cases}  
	\]
	Hence, up to multiplication by an invertible factor, the exchange relations are
	\begin{equation}
		\label{eq:An_ex_new} 
		\widetilde x_1 \widetilde y_1  - 1  - \widetilde x_2
		\ = \ 
		\widetilde x_{k} \widetilde y_{k}  
		- \widetilde x_{k-1}  
		- \widetilde x_{k+1} 
		\ = \
		\widetilde x_{n} \widetilde y_{n}  
		- \widetilde x_{n - 1 }
		- \lambda_n(\c)  
		\ = \ 0 ,
	\end{equation}
	for $ k \in \{ \bs{2}, \ldots, n-1 \} $. 	
	Set
	$ f_n := 
	f_n ( x_{n-1}, x_n, y_n ):=
	 \widetilde x_{n} \widetilde y_{n}  
	 - \widetilde x_{n - 1 }
	 - \lambda_n(\c)  
	$.
	
	The remainder of the proof is almost identical to the proof of \cite[Lemma~4.1]{BFMS}.
	The only difference is that the remaining term is $ 1 $ in loc.~cit., while it is $ \lambda_n(\c) $ here.
	Let us outline the arguments.
	\\
	We abuse notation and drop the tilde, 
	i.e., we write $ x_k $ instead of $ \widetilde x_k $ and $ y_k $ instead of $ \widetilde y_k $. 
	Using Lemma~\ref{Lem:ContFacts}(3), we deduce recursively from \eqref{eq:An_ex_new}
	that the \bs{$ (k-1) $-st} exchange relation can be rewritten as
	\[
	x_k = P_k(x_1, y_1, \ldots, y_{k-1})
	\ \ \ 
	\mbox{ for } k \in \{ 2, \ldots, n \} .
	\]
	We substitute the variables $ x_k $ and thus we get rid of the corresponding exchange relation in \eqref{eq:An_exchange}.
	We are left with $ f_n = 0 $ only.
	If we express $ f_n $ as a polynomial in the variables
	$ (x_1, y_1, \ldots, y_n) $, 
	we get (again using Lemma~\ref{Lem:ContFacts}(3)) that
	\[
	f_n(x_1, y_1, \ldots, y_{n})
	:= 
	P_{n+1} (x_1, y_1, \ldots, y_{n}) - \lambda_n(\c)
	\ .
	\] 
\end{proof}

	\begin{Bem}
		\label{Rk:new_var}
		\begin{enumerate}

			\item 
			If we express \eqref{eq:new_var} 
			in terms of the factors $ \lambda_k(\c) $
			we obtain:
			\begin{equation}
				\label{eq:new_var_lambda}
				\widetilde x_k = x_k \lambda_{k-1}(\c)
				\hspace{10pt}
				\mbox{ and } 
				\hspace{10pt}
				\widetilde y_k = y_k \lambda_k(\c) \lambda_{k-1}(\c)^{-1},
				\hspace{10pt}
				\mbox{for } 
				k \in \{ 1, \ldots, n \} \ ,
			\end{equation} 
			where we remind the reader that
			$ \lambda_{2s}(\c)
			=			
			\displaystyle 
			\prod_{\alpha = 1}^{s} c_{2 \alpha}^{-1} 
			$
			and 
			$ \lambda_{2s+1}(\c)
			=	
			\displaystyle 
			\prod_{\alpha = 1}^{s+1} c_{2 \alpha -1 }^{-1}
			$.
			 
			\smallskip 
			
			\item 
			Via a suitable change of variables using Lemma~\ref{Lem:ContFacts}(1),
			we may deduce from Proposition~\ref{Prop:An_hypersurface} for $ n = 2s $ even another presentation whose relation is independent of $ \c $, 
			namely $ \cA_\c^\prin(A_{2s}) \cong \KK_\c [\widetilde \z]/ \langle P_{2s+1}(\widetilde \z) - 1 \rangle $.
			The reason behind this is that the constant term of the corresponding continuant polynomial $ P_{2s+1}(\z) $ in Proposition~\ref{Prop:An_hypersurface} is zero.
			As this is not relevant for our further investigations, we do not go into details here.

	\end{enumerate}
	\end{Bem}

	Let us reformulate the missing part of Theorem~\ref{Thm:A}. 
	Its proof is a consequence of Propositions~\ref{Prop:ContSing} and~\ref{Prop:An_hypersurface}.
	Recall that we always have $ S := \Spec (\KK_\c) $.
	
	\begin{Thm}
		\label{Thm:A_prin}
		\bs{Let $ n \geq 2 $.}
		Consider the family defined by 
		$ \phi \colon \Spec ( \cA^\prin_{\c}(A_n)) \to S $.
		For a closed point $ \eta \in S $, the fiber 
		$ \Spec ( \cA^\prin_{\c}(A_n))_\eta $ is singular if and only if
		\[
		n = 2m - 1 \ \ \mbox{ and } \ \  
		\lambda_{2m-1} (\eta) = (-1)^{m} \ .
		\]  
		In the singular case, the fiber is isomorphic to an isolated hypersurface singularity of type $ A_1 $.  
	\end{Thm}
	
	\begin{proof}
		By Proposition~\ref{Prop:An_hypersurface}, we have 
		\[ 
			\cA^\prin_\c(A_n) \cong 
			\KK_\c[z_1, \ldots, z_{n+1}]/ \langle 
			P_{n+1} (z_1, \ldots, z_{n+1}) - \lambda_n(\c)
			\rangle   
			\ , 
		\]
		where 
		$ \lambda_{2s}(\c) =			
			\displaystyle 
			\prod_{\alpha = 1}^{s} c_{2 \alpha} 
		$
		and 
		$ \lambda_{2s+1}(\c)
			 =	
			\displaystyle 
			\prod_{\alpha = 1}^{s+1} c_{2 \alpha -1 }
		$.
		\\
		By Proposition~\ref{Prop:ContSing}, the fiber
		$ \Spec (\bs{\kappa(\eta)}[z_1, \ldots, z_{n+1}]/ \langle 
		P_{n+1} (z_1, \ldots, z_{n+1}) - \lambda_n(\eta)
		\rangle) $
		is singular if and only if $ n+1 = 2m $ and 
		$ \lambda_n (\eta) = \lambda_{2m-1} (\eta)  = (-1)^m $.
		Furthermore, Proposition~\ref{Prop:ContSing} implies if the fiber above $\eta$ in $ \Spec(\cA^\prin_\c(A_n)) $ is singular, then it has an isolated singularity of type $ A_1 $ at the origin.

		Note that the coefficients of the polynomial $ P_{n+1} (z_1, \ldots, z_{n+1}) - \lambda_n(\c) $ of the presentation for $ \cA^\prin_\c(A_n)) $ (Proposition~\ref{Prop:An_hypersurface}) 
		are contained in $ \{ -1 , 0 , 1 \} $ in the singular case. 
		Therefore, there is no dependence on the ground field.
		In particular, we do not have to take $ p $-bases into account and the computations in the relative setup \bs{(i.e., the computation of the singular locus relative to $ \Spec (\kappa(\eta)) $)} already provide the singular locus.	
	\end{proof}

\subsection{Type $ \boldsymbol{B_n} $}

Following the alphabetical ordering, we continue with cluster algebras with principal coefficients of finite cluster type $ B_n $, for $ n \geq 2 $.

By Theorem~\ref{Thm:FinTypClass}, the exchange matrix of the initial labeled seed $ \Sigma  = (\x, B) $ can be taken as the $ n \times n $ matrix 
\[ 
	B
	= 
	\begin{pmatrix} 
	0 & 1 & \\
	-1 & \ddots & \ddots   \\
	& \ddots & 0 & 1  \\
	&  & -1 & 0 & 1  \\
	&  &  & -2  & 0 \\
	\end{pmatrix}  
\]
and thus, by Theorem~\ref{Thm:Presentation}, the corresponding cluster algebra with principal coefficients has the presentation 
\begin{equation}
	\label{eq:Bn_exchange} 
	\cA^\prin_\c(B_n) \cong
	\KK_\c [\x, \y]
	\hspace{-0.15cm}\raisebox{-0.2cm}{$\bigg/$}
	\hspace{-0.25cm}\raisebox{-0.8cm}%
	{$\left\langle \,
		\begin{minipage}{4.4cm}
			$ x_1 y_1 - c_1 -  x_2 
			\\[2pt]
			x_k y_k - c_k x_{k-1} -  x_{k+1} 
			\\[2pt]
			x_{n-1} y_{n-1} - c_{n-1} x_{n-2} -  x_{n}^2 
			\\[2pt]
			x_n y_n - c_n x_{n-1} -  1 $
			
		\end{minipage}
		\, \Bigg| \,\, k \in \{ 2, \ldots, n-2 \} 
		\, \right\rangle
		\ ,
		$} 
\end{equation} 
where $(\c,\x,\y) = (c_1, \ldots, c_n; x_1, \ldots, x_n; y_1, \ldots, y_n) $ 
and similar as before, 
we use the notation $ \cA^\prin_\c (B_n) :=  \cA^\prin_\c (\Sigma) $ to indicate that we are in the finite cluster type $ B_n $. 

We deduce the following new presentation, which simplifies the computations for the Jacobian criterion. 
The methods for the proof are closely related to those applied in Proposition~\ref{Prop:An_hypersurface}.

\begin{lemma}
	\label{Lem:Bn_2_eq} 
	There exists an isomorphism \bs{over $ \KK_\c$}
	\[
		\cA^\prin_\c(B_n) \cong 
		\KK_\c[z_1, \ldots, z_{n-1}, u_1, u_2, u_3 ]/ \langle g_n, h_n \rangle 
		\ ,
	\]
	where the generators of the ideal on the right hand side are 
	\[
	\begin{array}{l}
	g_n
	:= 
	\Big( u_1  u_2  - \lambda_{n}(\c) \Big)  u_3
	 - \lambda_{n-1}(\c)^{-1} u_1^2 - P_{n-2}(z_1,\ldots, z_{n-2}) 
	\ ,
	
	\\[15pt]
	
	h_n := u_1 u_2  - \lambda_{n}(\c) - P_{n-1}(z_1, \ldots, z_{n-1}) 
	\ , 
	\end{array}
	\]
	for $ \lambda_k(\c) $ as defined in Proposition~\ref{Prop:An_hypersurface}
	(i.e., 
	$ \lambda_{2s}(\c)
	=			
	\displaystyle 
	\prod_{\alpha = 1}^{s} c_{2 \alpha}^{-1} 
	$ and $ \lambda_{2s+1}(\c)
	=	
	\displaystyle 
	\prod_{\alpha = 1}^{s+1} c_{2 \alpha -1 }^{-1} $)
	and, as before, $ P_{n-2} $ and $ P_{n-3} $ are the continuant polynomials discussed in Section~\ref{Subsec:Cont}.
\end{lemma}

\begin{proof}
	As in the proof of Proposition~\ref{Prop:An_hypersurface}, 
	we perform the change of variables \eqref{eq:new_var_lambda},
	\[
		\widetilde x_k := x_k \lambda_{k-1}(\c)
		\hspace{10pt}
		\mbox{ and } 
		\hspace{10pt}
		\widetilde y_k := y_k \lambda_k(\c) \lambda_{k-1}(\c)^{-1},
		\hspace{10pt}
		\mbox{for } \
		k \in \{ 1, \ldots, n \} \ ,	
	\]
	and it follows that the first $ k - 2 $ exchange relations in \eqref{eq:Bn_exchange} can be rewritten as 
	(up to multiplication by an invertible factor)
	\begin{equation} 
		\label{eq:Bn_afterchange_1} 
			\widetilde x_1 \widetilde y_1  - 1  - \widetilde x_2 
			\ = \ 
			\widetilde x_{k} \widetilde y_{k}  
			- \widetilde x_{k-1}  
			- \widetilde x_{k+1} 
			\ = \ 
			0
			\ , 
			\hspace{10pt}
			\mbox{ for } k \in \{ 2, \ldots, n -2 \}
			\ . 
	\end{equation} 
	On the other hand, applying the above substitution to the polynomials of the remaining relations,
	we see that the missing relations become 
	(up to multiplication by an invertible factor)   
	\begin{equation}
		\label{eq:Bn_afterchange_2}
		\widetilde x_{n-1} \widetilde  y_{n-1} - \widetilde x_{n-2} - \lambda_{n-1}(\c)^{-1} \widetilde  x_{n}^2
		 \ = \ 
		 \widetilde x_n \widetilde y_n - \widetilde  x_{n-1} - \lambda_{n}(\c)
		 \ = \ 
		 0	
		 \ .			
	\end{equation} 
	We proceed analogous to the proof of Proposition~\ref{Prop:An_hypersurface} resp.~as in the proof of \cite[Lemma~5.1]{BFMS}.
	We abuse notation and drop the tilde in $ \widetilde \x $ and $ \widetilde \y $.
	The relations~\eqref{eq:Bn_afterchange_1} provide the substitutions
	\[
		x_k = P_k(x_1, y_1, \ldots, y_{k-1}) 
		\ \ \ 
		\mbox{ for }
		k \in \{ 2 , \ldots, n-1 \} 
		\ .
	\]	
	Consider $ h_n := x_n y_n - x_{n-1} - \lambda_{n}(\c) $
	and 
	\[ 
	\begin{array}{rcl}
		g_n 
		& := 
		& ( x_{n-1} y_{n-1} - x_{n-2} - \lambda_{n-1}(\c)^{-1} x_{n}^2 ) + y_{n-1} h_n =
		
		\\[10pt]
		
		& = 
		& 
		\Big( x_n  y_n  - \lambda_{n}(\c) \Big)  y_{n-1}
		- x_{n-2} - \lambda_{n-1}(\c)^{-1} x_{n}^2 
		\ .
	\end{array} 
	\]
	Applying the substitution for $ x_2, \ldots, x_{n-1} $
	and renaming the variables leads to the statement of the lemma. 
\end{proof}

Taking the proof of~\cite[Proposition~5.2]{BFMS} as guideline, we can show the remaining part of Theorem~\ref{Thm:B}.

\begin{Thm}
	\label{Thm:B_prin}
	The following characterization holds for the fibers of $ \phi \colon \Spec (\cA_\c^\prin (B_n))\to S $ with $ \eta \in S $ being a closed point:
	\begin{enumerate}
		\item 
		If $ n = 2m+1 $ is odd, 
		then  the fiber $\Spec (\cA_\c^\prin (B_n))_\eta $ of $\phi$ above $ \eta $ is singular if and only if 
		$ \lambda_{n}(\eta) = (-1)^{m+1} $.
		\item 
		If $ n = 2m $ is even,
		then the fiber of $\phi$ above $ \eta $ is singular if and only if 
		$ \car(\KK) = 2 $ and $ \lambda_{n-1}(\eta) \in \bs{\kappa(\eta)}^2  $ is a square. 
	\end{enumerate}
	In the singular cases, 
	the singular locus is a closed point and
	the fiber $\Spec (\cA_\c^\prin (B_n))_\eta$ is locally isomorphic to a hypersurface with an isolated singularity of type $ A_1 $.
\end{Thm}

\begin{proof}
	We use the presentation $\cA^\prin_\c(B_n) \cong 
	\KK_\c[z_1, \ldots, z_{n-1}, u_1, u_2, u_3 ]/ \langle g_n, h_n \rangle 
	$
 	deduced in Lemma~\ref{Lem:Bn_2_eq}.
	Since $ \dim ( \Spec (\cA_\c^\prin (B_n))_\eta) = n $, 
	we have to consider the $ 2 $-minors of the Jacobian matrix 
	\[
		\Jac(B_n) := \Jac (g_n, h_n; z_1, \ldots, z_{n-1}, u_1, u_2, u_3 ) 		
	\]
	in order to determine the singular locus. 
	The columns corresponding to the derivatives by $ (u_1, u_2, u_3 ) $ are 
	\[
		\begin{pmatrix}
			u_2 u_3 -  2 \lambda_{n-1}(\eta)^{-1} u_1  \ \
			& u_1 u_3 \ \ 
			& u_1 u_2 - \lambda_{n}(\eta)
			\\
			u_2 
			& u_1 
			& 0  
		\end{pmatrix}
	\]
	and the vanishing of the maximal minors of this sub-matrix provide the following relations for the singular locus
	(taking into account that $ \lambda_{n-1}(\eta) $ is invertible)
	\begin{equation}
		\label{eq:Sing_Bn_1}
		2  u_1^2 
		\ = \ 
		u_2 (u_1 u_2 - \lambda_{n}(\eta))
		\ = \ 
		u_1 (u_1 u_2 - \lambda_{n}(\eta))
		\ = \ 
		0
		\ . 
	\end{equation} 
	First, suppose that $ \car(\KK) \neq 2 $. 
	Then, \eqref{eq:Sing_Bn_1} is equivalent to $ u_1 = u_2 =  0 $.
	This provides that $ g_n = h_n = 0 $ is equivalent to 
	\begin{equation}
		\label{eq:Sing_Bn_2} 
		\lambda_{n}(\eta)  u_3 + P_{n-2}(z_1,\ldots, z_{n-2})  
		\ = \  
		P_{n-1}(z_1, \ldots, z_{n-1}) + \lambda_{n}(\eta)  
		\ = \ 0 	
	\end{equation} 
	and that the transpose of the last column in $ \Jac(B_n) $ becomes $ (- \lambda_{n}(\eta), \, 0 ) $.
	Therefore, the inclusion 
	\[
		\Sing (\Spec (\cA_\c^\prin (B_n))_\eta) \subseteq 
		\Sing ( V (P_{n-1}(z_1, \ldots, z_{n-1}) + \lambda_{n}(\eta)) )
		\ 
	\]
	holds. 
	By Proposition~\ref{Prop:ContSing}, the singular locus on the right hand side is non-empty if and only if $ n-1 = 2m $ is even and $ \lambda_{n}(\eta) = (-1)^{m+1} $.
	In particular, $ \Spec (\cA_\c^\prin (B_n))_\eta $ is regular if $ n $ is even or if $ n = 2m + 1 $ and $ \lambda_{n}(\eta) \neq (-1)^{m+1} $.   
	\\
	Hence, assume $ n = 2m + 1 $ is odd and $ \lambda_{n}(\eta) = (-1)^{m+1} $. 
	Then, the hypersurface determined by the polynomial $ P_{n-1}(z_1, \ldots, z_{n-1}) + \lambda_{n}(\eta)  $ has an isolated singularity at $ V (z_1, \ldots, z_{n-1} ) $.
	Especially, its derivative by $ z_{n-1} $ has to vanish;
	in other words, we have 
	$ P_{n-2} (z_1, \ldots, z_{n-2} ) = 0 $ by Lemma~\ref{Lem:ContFacts}(4). 
	Using this, the first equality of \eqref{eq:Sing_Bn_2} is equivalent to $ u_3 = 0 $.
	Therefore, we deduce that
	\[ 
		\Sing(\Spec (\cA_\c^\prin (B_n))_{\eta}) = V (z_1, \ldots, z_{n-1}, u_1, u_2, u_3 ) 
		\hspace{1.5cm} 
		(\mbox{if } \car(\KK) \neq 2) 
		\ .
	\]
	Let us determine the singularity type. 
	Locally at the origin $ u_1 u_2 - \lambda_n (\eta) $ is a unit and thus $ g_n = 0 $ can be used to eliminate the variable $ u_3 $.
	As the latter does not appear in $ h_n $, the elimination of $ u_3 $ has no effect on $ h_n $. 
	By Proposition~\ref{Prop:ContSing}, 
	the hypersurface defined by $ P_{2m}(z_1, \ldots, z_{n-1}) + (-1)^{m+1} \in \bs{\kappa(\eta)}[z_1, \ldots, z_{n-1}] $ has an isolated singularity of type $ A_1 $ at the origin. 
	It follows from that $ \Spec( \cA_\c^\prin (B_n))_\eta $  is locally isomorphic to a hypersurface singularity with an isolated singularity of type $ A_1 $ at the origin if $ \car(\KK) \neq 2 $.
	
 	It remains to handle the case $ \car(\KK) = 2 $. 
 	Here, \eqref{eq:Sing_Bn_1} only provides
	\begin{equation}
		\label{eq:Sing_Bn_1b}
		u_2 (u_1 u_2 - \lambda_{n}(\eta))
		\ = \ 
		u_1 (u_1 u_2 - \lambda_{n}(\eta))
		\ = \ 
		0
		\ . 
	\end{equation} 
	We make a case distinction for $ u_1 (u_1 u_2 - \lambda_{n}(\eta)) = 0 $.  
	If $ u_1 = 0 $, then the same arguments as in $ \car(\KK) \neq 2 $ lead to an isolated singularity at the origin, 
	which is locally isomorphic to a hypersurface of type $ A_1 $,
	if $ n = 2m +1 $ and $ \lambda_n (\eta) = (-1)^{m+1} $. 
	\\
	Suppose that $ u_1 u_2 - \lambda_{n}(\eta) = 0 $.
	Since $ \lambda_n(\eta) $ is invertible, so are $ u_1 $ and $ u_2 $.
	We have that $ g_n = h_n = 0 $ is equivalent to 
	\begin{equation}
		\label{eq:Sing_Bn_2b}
		\lambda_{n-1}(\eta)^{-1} u_1^2 + P_{n-2}(z_1,\ldots, z_{n-2}) 
		\ = \ 
		P_{n-1}(z_1, \ldots, z_{n-1}) 
		\ = \ 
		0 
		\ . 
	\end{equation}
	In particular, $ P_{n-2}(z_1,\ldots, z_{n-2}) $ is invertible. 
	The vanishing of the minor of $ \Jac(B_n) $ corresponding to $ ( z_{n-1}, u_2) $ is equivalent to $ u_3 = 0 $. 
	\\
	Taking the latter into account, the transpose of the column of $ \Jac(B_n) $ given by the derivatives by $ u_2 $ is $ ( 0, \, u_1 ) $. 
	By considering the minors coming from the columns determined by $ (z_i, u_2) $,
	for $ i \in \{ 1, \ldots, n -2 \} $, 
	 we get 
	 \[
	 \begin{array}{c} 
	 	\Sing (\Spec (\cA_\c^\prin (B_n))_\eta) \setminus V(z_1, \ldots, z_{n-1}, u_1, u_2, u_3)  = 
	 	\\[5pt]
	 	=
	 	V(u_1 u_2 - \lambda_{n}(\eta), \, 
	 	u_3, \, z_{n-1} ) \cap 
	 	\Sing ( V (P_{n-2}(z_1, \ldots, z_{n-2}) + \lambda_{n-1}(\eta)^{-1} u_1^2 )
	 	\ ,	 
 	\end{array} 
	 \]
	 where we use $ P_{n-2} (z_1, \ldots, z_{n-2}) $ invertible, 
	 $ P_{n-3}(z_1, \ldots, z_{n-3}) = \frac{\partial P_{n-2}(z_1, \ldots, z_{n-2}) }{\partial z_{n-2}} = 0 $,
	 and
	 $ P_{n-1} (z_1, \ldots, z_{n-1}) =  P_{n-2} (z_1, \ldots, z_{n-2}) z_{n-1} - P_{n-3}(z_1, \ldots, z_{n-3}) $ 
	 in order to deduce the condition $ z_{n-1} = 0 $. 
	 
	 As before, 
	 by Proposition~\ref{Prop:ContSing},  
	 $ \Sing ( V (P_{n-2}(z_1, \ldots, z_{n-2}) + \lambda_{n-1}({\eta})^{-1} u_1^2 ) $  is non-empty if and only if $ n -2 = 2m $ is even and $ \lambda_{n-1}({\eta})^{-1} u_1^2  = 1 $ 
	 (using $ \car(\KK) = 2 $).
	 Notice that $ u_1^2  = \lambda_{n-1}({\eta})  $ is only possible to hold if $ \lambda_{n-1}({\eta}) \in \bs{\kappa(\eta)}^2  $.
	 If $ \lambda_{n-1}({\eta}) \in \bs{\kappa(\eta)}^2  $ and let $  \rho_{n-1}({\eta}) \in \bs{\kappa(\eta)} $ {be} such that $ \lambda_{n-1}({\eta}) = \rho_{n-1}({\eta})^2 $. 
	 Then we have a singularity at the closed point
	 \[  
	 q := V(u_1 + \rho_{n-1}({\eta}), \, 
	 u_2 + \rho_{n-1}({\eta})^{-1} \lambda_{n}({\eta}), 
	 u_3, z_1, \ldots, z_{n-1} )
	 \ \ \ \mbox{ if } n - 2 = 2m .
	 \] 
	 For the singularity type,
	 we have already seen that $ h_n $ is equal to a unit times $ z_{n-1} $ plus some term independent of $ z_{n-1} $, locally at $ q $. 
	 Hence, we may perform a substitution to eliminate $ z_{n-1} $. 
	 Since $ g_n $ is independent of $ z_{n-1} $, the latter has no effect on it. 
	 Furthermore, 
	 if we introduce the local variables $ v_1 := u_1 + \rho_{n-1}({\eta}) $ 
	 and $ v_2 := u_1 u_2 + \lambda_n ({\eta}) $ at $ q $,
	 we get 
	 (using $ \car(\KK) = 2 $ and $ n = 2m-2 $)
	 \[
	 	\begin{array}{rl} 
	 		g_n
	 		&= 
	 		( u_1  u_2  + \lambda_{n}({\eta}) )  u_3
	 		+ \lambda_{n-1}({\eta})^{-1} u_1^2 + P_{n-2}(z_1,\ldots, z_{n-2}) 
	 		= 
	 		\\[5pt]
	 		& 
	 		= v_2 u_3 +  ( \rho_{n-1}({\eta})^{-1} v_1)^2 + 1 +  P_{2m}(z_1,\ldots, z_{2m})
	 		\ .
	 	\end{array} 
	 \]
	 Since $ P_{2m}(z_1, \ldots, z_{2m}) + 1 = 0 $ (considered  as hypersurface in $ \Spec (\bs{\kappa(\eta)}[z_1, \ldots, z_{2m}] $)
	 has an isolated singularity of type $ A_1 $ at the origin
	 (Proposition~\ref{Prop:ContSing}), 
	 we conclude that, 
	 locally at $ q $, 
	the variety $ \Spec (\cA_\c^\prin (B_n))_{\eta} $ is isomorphic to a hypersurface with an isolated singularity of type $ A_1 $ at $ q $.
	
	Since hypersurface singularities of type $ A_1 $ have integer coefficients,
	we obtain that the relative singular locus (which we determined via the derivatives with respect to the variables $ (z_1, \ldots, z_{n-1}, u_1, u_2, u_3) $ only) already provides the singular locus 
	and it is not necessary to take a $ p $-basis of \bsOUT{the ground field} $\bs{\kappa(\eta)}$ into account. 
	This ends the proof of Theorem~\ref{Thm:B_prin}. 
\end{proof}

\subsection{Type $ \boldsymbol{C_n} $}

Next, we consider cluster algebras with principal coefficients of finite cluster type $ C_n $, where $ n \geq 3 $.  

The exchange matrix for the initial labeled seed of our choice 
is the $ n \times n $ matrix
\[ 
	B = 
	\begin{pmatrix} 
	0 & 1 & \\
	-1 & \ddots & \ddots   \\
	& \ddots & 0 & 1  \\
	&  & -1 & 0 & 2  \\
	&  &  & -1  & 0 \\
	\end{pmatrix}  
\]
(by Theorem~\ref{Thm:FinTypClass}),
which leads to the presentation 
\begin{equation}
	\label{eq:Cn_exchange} 
	\cA^\prin_\c(C_n) \cong
	\KK_\c [\x, \y]
	\hspace{-0.15cm}\raisebox{-0.2cm}{$\bigg/$}
	\hspace{-0.2cm}\raisebox{-0.7cm}%
	{$\left\langle \,
		\begin{minipage}{3.4cm}
			$ x_1 y_1 - c_1 -  x_2 
			\\[2pt]
			x_k y_k - c_k x_{k-1} -  x_{k+1} 
			\\[2pt]
			x_n y_n - c_n x_{n-1}^2 -  1 $
			
		\end{minipage}
		\, \Bigg| \,\, k \in \{ 2, \ldots, n-1 \} 
		\, \right\rangle
		\ ,
		$} 
\end{equation} 
(by Theorem~\ref{Thm:Presentation})
with $(\c,\x,\y) = (c_1, \ldots, c_n; x_1, \ldots, x_n; y_1, \ldots, y_n) $
and using the notation $ \cA^\prin_\c (C_n) $ 
analogous to the previous cases.

\begin{lemma}
	\label{Lem:Cn_hypersurface}
	The cluster algebra $ \cA^\prin_\c(C_n) $
	is isomorphic \bs{over $ \KK_\c$} to 
	\[
		\KK_\c[z_1, \ldots, z_{n+1}]/
		\langle
		\,
		P_n(z_1, \ldots, z_n) z_{n+1}  -  P_{n-1}(z_1, \ldots, z_{n-1})^2  -
		\lambda_{n-2}(\c)  \lambda_n(\c)  
		\, 
		\rangle  
	\] 
	with 
	$ \lambda_{n-2}(\c), \lambda_n(\c) $ as defined in Definition~\ref{Def:ContPoly_lamda_n}(2) and 
	$ P_{n-1} $, $ P_n $ the continuant polynomials defined in Section~\ref{Subsec:Cont}.
\end{lemma} 

\begin{proof}
	The proof is almost the same as the one for Proposition~\ref{Prop:An_hypersurface}
	with the only difference that we have $ x_n^2 $ instead of $ x_n $ in the last exchange relation.
	
	We use the change of variables \eqref{eq:new_var_lambda},
	\[
		\widetilde x_k = x_k \lambda_{k-1}(\c)
		\hspace{10pt}
		\mbox{ and } 
		\hspace{10pt}
		\widetilde y_k = y_k \lambda_k(\c) \lambda_{k-1}(\c)^{-1},
		\hspace{10pt}
		\mbox{for } 
		k \in \{ 1, \ldots, n \} \ ,	
	\]
	so that the first $ k-1 $ exchange relations of \eqref{eq:Cn_exchange} become after multiplication with an invertible factor
	\begin{equation} 
		\label{eq:Cn_afterchange_1} 
		\widetilde x_1 \widetilde y_1  - 1  - \widetilde x_2 
		\ = \ 
		\widetilde x_{k} \widetilde y_{k}  
		- \widetilde x_{k-1}  
		- \widetilde x_{k+1} 
		\ = \ 
		0
		\ , 
		\hspace{10pt}
		\mbox{ for } k \in \{ 2, \ldots, n - 1 \}
		\ . 
	\end{equation} 
	Furthermore, if we define $ \widetilde{ \widetilde y}_n := \lambda_{n-2}(\c) \widetilde y_n $ 
	(which has no effect on \eqref{eq:Cn_afterchange_1}),
	we get 
	\[
		x_n y_n - c_n x_{n-1}^2 -  1
		=
		\lambda_{n-2}(\c)^{-1} \lambda_{n}(\c)^{-1}
		\Big( 
			\widetilde x_n  \widetilde{\widetilde y}_n  -  \widetilde x_{n-1}^2 -  \lambda_{n-2}(\c)  \lambda_n(\c) 
		\Big)
		\ .
	\]
	Here, we use that $ \lambda_{n}(\c)^{-1} =  c_n  \cdot \lambda_{n-2}(\c)^{-1} $, which follows from its definition  (Definition~\ref{Def:ContPoly_lamda_n}(2)).
	
	Again, we abuse notation and do not write the tilde anymore. 
	Using \eqref{eq:Cn_afterchange_1}, we can substitute $ x_k = P_k(x_1, y_1, \ldots, y_{k-1}) $ for $ k \in \{ 2, \ldots, n \} $
	and the claim of the lemma follows.
\end{proof}

Using the proofs of \cite[Propositions~5.4 and~5.5]{BFMS} as blueprint, 
we finish the proof of Theorem~\ref{Thm:C}.
The missing part is

\begin{Thm}
	\label{Thm:C_prin}
	The singularities of the fibers of the family $ \phi \colon \Spec (\cA_\c^\prin (C_n) ) \to S $ are characterized as follows, where $ \eta \bsOUT{ = (\eta_1, \ldots, \eta_n )} \in S $ is a closed point:	
	\begin{enumerate}
		\item 
		Let $ \car(\KK) \neq 2 $.
		The fiber $ \Spec (\cA_\c^\prin(C_n))_{\eta} $ is singular if and only if 
		$ n = 2m+1 $ and 
		$ -{\eta}_n \in \bs{\kappa(\eta)}^2 $ is a square in $ \bs{\kappa(\eta)} $.
		In the singular case, $ \Sing(\Spec (\cA_\c^\prin(C_n))_{\eta}) $ is a closed point and locally at the latter,
		$ \Spec (\cA_\c^\prin(C_n))_{\eta} $ is isomorphic to a hypersurface singularity of type $ A_1 $.
		
		\item 
		If $ \car(\KK) = 2 $, 
		then 
		$ \Spec (\cA_\c^\prin(C_n))_{\eta} $ is singular if and only if 
		$ -\eta_n \in \bs{\kappa(\eta)}^2 $ is a square in $ \bs{\kappa(\eta)} $.
		In the singular case, let  $ \delta_n \in \bs{\kappa(\eta)} $ be such that $ \delta_n^2 = -\eta_n^{-1} $ and 
		set $ \rho_n ( \eta ) := \delta_n \lambda_{n-2} (\eta) $.
		We have:
		\begin{enumerate}
			\item 
			The singular locus is itself singular if and only if $ n-1 = 2m $ and $ \rho_n (\eta) = 1 $.
			
			\item 
			If $ n-1 = 2m $ and $ \rho_n(\eta) = 1 $,
			then $ \Sing ( \Spec (\cA_\c^\prin(C_n))_{\eta} ) $ has an isolated singularity at a closed point and locally at the latter $ \Spec (\cA_\c^\prin(C_n))_{\eta} $ is isomorphic to the hypersurface singularity 
			
			\[
			\Spec (\bs{\kappa(\eta)}[x_1, \ldots, x_{2m} , y, z] / \langle yz + \Big( \sum_{i=1}^{m} x_{2i-1} x_{2i} \Big)^2 \rangle 
			\] 
			
			\medskip 
			
			and  $ \Sing ( \Spec (\cA_\c^\prin(C_n))_{\eta} )  $ identifies along this isomorphism with 
			
			\[
			V(y,z, \sum_{i=1}^{m} x_{2i-1} x_{2i} )
			\]
			
			\medskip 
			
			which is isomorphic to an hypersurface of type $ A_1 $ if $ m > 1 $ and a union of two lines if $ m = 1 $.
			
			\item 
			At any point $ q $ at which $ \Sing ( \Spec (\cA_\c^\prin(C_n))_{\eta} )  $  is regular, $ \Spec (\cA_\c^\prin(C_n)_{\eta}) $ is isomorphic to a $ (n-2) $-dimensional cylinder over the hypersurface singularity 
			\[ 
				\Spec(\bs{\kappa(\eta)}[x,y,z]/\langle xy - z^2 \rangle ) 
			\] 
			of type $ A_1 $.
		\end{enumerate}
	\end{enumerate}
\end{Thm}

\begin{proof}
	Consider the presentation $ \cA_\c^\prin(C_n )  \cong \KK_\c[z_1, \ldots, z_{n+1}]/
	\langle h_n \rangle $, 
	for
	\[
		h_n := P_n(z_1, \ldots, z_n) z_{n+1}  -  P_{n-1}(z_1, \ldots, z_{n-1})^2  +  \mu_n (\c)
	\ ,   
	\] 
	where we abbreviate $ \mu_n (\c) :=
	-\lambda_{n-2}(\c)  \lambda_n(\c) $
	(Lemma~\ref{Lem:Cn_hypersurface}). 
	\\
	As the derivative of $ h_n $ by $ z_{n+1} $ has to vanish for a singularity, we get the condition $ P_n(z_1, \ldots, z_n)  = 0 $. 
	Using the latter, $ h_n = 0 $ is equivalent to (for the fiber) 
\begin{equation}
		\label{eq:Sing_Cn_1}
		P_{n-1}(z_1, \ldots, z_{n-1})^2  -  \mu_n ({\eta}) = 0 
		\ .
	\end{equation} 
	Since $ \mu_n({\eta}) = -{\eta}_n^{-1} \lambda_{n-2}({\eta})^2  $  
	(by the definition of $ \lambda_n({\eta}) $, Definition~\ref{Def:ContPoly_lamda_n}(2)),
 	the equation \eqref{eq:Sing_Cn_1} can only be fulfilled if $ -{\eta}_n \in \bs{\kappa(\eta)}^2 $ is a square.
 	\\
 	Suppose there exists $ {\delta}_n \in \bs{\kappa(\eta)} $ such that $ {\delta}_n^2 = -{\eta}_n^{-1} $.
 	Recall that we defined in the statement $ \rho_n ({\eta}) = {\delta}_n \lambda_{n-2}({\eta}) $ 
 	so that $ \mu_n({\eta}) = \rho_n({\eta})^2 $.
 	Then, \eqref{eq:Sing_Cn_1} can be rewritten as 
 	\begin{equation}
 		\label{eq:Sing_Cn_2}
 		P_{n-1}(z_1, \ldots, z_{n-1})  = \pm \rho_n ({\eta})
 		\ .
 	\end{equation} 
 	Taking the last equality and $ \rho_n ({\eta}) $ being invertible into account provides that the vanishing of the derivative of $ h_n $ by $ z_n $ is equivalent to 
 	$ z_{n+1} = 0 $. 
 	
 	Assume $ \car(\KK) \neq 2 $. 
 	By induction on $ k $,
 	one can show that
 	\[
 		h_n = \frac{\partial h_n}{\partial z_{n+1}} = \cdots = \frac{\partial h_n}{\partial z_{n-2k}} = 0 
 	\]
 	lead to 
 	\begin{enumerate}
 		\item[(a$_k$)]
 		$ z_{n+1} \, =  \, \cdots \, = \, z_{n - 2k + 1 } \, = \, 0 \, $;
  		
  		\smallskip 
 		
 		\item[(b$_k$)]
 		$ P_{n-2k-1} (z_1, \ldots, z_{n-2k-1}) \, = \, (-1)^k (\pm \rho_n ({\eta})) \, $;
 		
 		\smallskip 
 		
 		\item[(c$_k$)] 
 		$ P_{n-2k} (z_1, \ldots, z_{n-2k} ) \, = \, 0 \, $.
 	\end{enumerate}
 
 	This follows the same way as the analogous statements in the proof of \cite[Proposition~5.4]{BFMS}.
 	Thus, we only outline the arguments.
 	The case $ k = 0 $ was shown just before. 
 	For the induction steps we have:
 	
 	\begin{itemize} 
 		\item 
 		(a$_k$) and Lemma~\ref{Lem:ContFacts}(1) provide 
 		\begin{equation}
 			\label{eq:Sing_Cn_3}
 			P_{2k}(z_{n-2k}, \ldots, z_{n-1}) = \pm 1 \ . 
 		\end{equation} 
 		Then (a$_0$), (b$_0$), applying Lemma~\ref{Lem:ContFacts}(4) for $ \frac{\partial h_n}{\partial z_{n-2k-1} } $, and $ \car(\KK) \neq 2 $ imply (c$_{k+1}$).

 		\smallskip

 		\item 
 		Lemma~\ref{Lem:ContFacts}(3) applied to $ P_{n-2k}(z_1, \ldots, z_{n-2k}) $, (b$_k$), and (c$_{k+1}$) yield $ z_{n-2k} = 0 $.
 		Then, Lemma~\ref{Lem:ContFacts}(3) applied to $ P_{n-2k-1}(z_1, \ldots, z_{n-2k-1}) $, (b$_k$), and (c$_{k+1}$) provide (b$_{k+1}$).  
 		
 		\smallskip 
 		
 		\item
 		Lemma~\ref{Lem:ContFacts}(4) for $ \frac{\partial h_n}{\partial z_{n-2k-2} } $, (b$_0$), (b$_{k+1}$), and $ P_{2k+1} (z_{n-2k-1}, 0, \ldots , 0 ) = \pm z_{n-2k-1} $ (via Lemma~\ref{Lem:ContFacts}(1)) finally imply $ z_{n-2k-1} = 0 $ and thus (a$_{k+1} $).  
 	\end{itemize} 

	For the assertion on the singular locus we have to distinguish two cases depending on whether $ n $ is even or odd. 
	First, let us look at $ n = 2m $. 
	To have a singularity, we have to have $ \frac{\partial h_n}{\partial z_1} = 0 $, i.e.,
	\[
		0 
		= P_{n-1}(z_2, \ldots, z_n) z_{n+1} - 2 P_{n-1}(z_1, \ldots, z_{n-1}) P_{n-2}(z_2,\ldots, z_{n-1})   
		=  -2(\pm \rho_n ({\eta})) (\pm 1)
		\ , 
	\]
	where the last equality uses (a$_0$), (b$_0$), and \eqref{eq:Sing_Cn_3} for $ k = m - 1 $. 
	Since the term on the right hand side is invertible, we arrived to a contradiction,
	i.e., 
	$ \Spec(\cA_\c^\prin (C_n))_{\eta} $ is regular if $ n = 2m $ and $ \car(\KK) \neq 2 $. 
	\\
	Suppose that $ n = 2m + 1 $. 
	First, (a$_m$) states that $ z_2 = \cdots = z_{n+1} = 0 $
	and (c$_m$) yields $ z_1 = 0 $. 
	Therefore,  $ \Spec(\cA_\c^\prin (C_n)_{\eta}) $ either has an isolated singularity at the origin or is regular. 
	Moreover, (b$_{m-1}$) and $ z_2 = 0 $ provide $ \pm \rho_n({\eta}) = (-1)^m $
	and hence $ \mu_n({\eta}) = \rho_n({\eta})^2  = 1 $. 
	Therefore, we get
	\[
		h_n = P_{2m+1}(z_1, \ldots, z_{2m+1}) z_{2m+2}  -  (P_{2m}(z_1, \ldots, z_{2m}) + 1)(P_{2m}(z_1, \ldots, z_{2m}) - 1) 
		\ . 
	\]  
	Let us look at the local situation at the origin. 
	By Lemma~\ref{Lem:ContFacts}(1), 
	$ P_{2m}(z_1, \ldots, z_{2m}) $ is a unit and 
	hence, it makes sense to define
	$ w_{2m+1} := P_{2m+1}(z_1, \ldots, z_{2m+1}) $ 
	as local variable replacing $ z_{2m+1} $ 
	using Lemma~\ref{Lem:ContFacts}(3).
	\\
	If $ m $ is odd, then $ P_{2m}(z_1, \ldots, z_{2m}) - 1 $ is a unit (Lemma~\ref{Lem:ContFacts}(1))
	and 
	the exists a local coordinate change $ (w_1, \ldots, w_{2m}) $ 
	in $ (z_1, \ldots, z_{2m}) $ such that 
	$ P_{2m}(z_1, \ldots, z_{2m}) + 1 = \sum_{i=1}^m w_{2i-1} w_{2i} $
	(Proposition~\ref{Prop:ContSing}).
	\\
	On the other hand, if $ m $ is even, 
	$ P_{2m}(z_1, \ldots, z_{2m}) + 1 $ is a unit
	and, for a suitable choice of local variables, we have $ P_{2m}(z_1, \ldots, z_{2m}) - 1 = \sum_{i=1}^m w_{2i-1} w_{2i}  $.
	\\
	In conclusion, we get $ h_n = w_{2m+1} z_{2m+2} - \epsilon \sum_{i=1}^m w_{2i-1} w_{2i} $, for some unit $ \epsilon $,
	which implies that $ \Spec(\cA_\c^\prin (C_n)_{\eta}) $ has an isolated singularity of type $ A_1 $ at the origin
	(under the hypotheses $ \car(\KK) \neq 2 $,  $ -{\eta}_n \in \bs{\kappa(\eta)}^2 $, and $ n = 2m +1 $).
	
	It remains to handle, the case $ \car(\KK) = 2 $. 
	Recall from the beginning of the proof that in order to have a singularity, the conditions
	$ -{\eta}_n \in \bs{\kappa(\eta)}^2 $,
	$ z_{n+1} = 0 $,
	$ P_n (z_1, \ldots, z_n) = 0 $,
	and $ P_{n-1} (z_1, \ldots, z_{n-1}) + \rho_n({\eta}) = 0 $ (cf.~\eqref{eq:Sing_Cn_2} and use $ \car(\KK) = 2 $)
	must hold.  
	Moreover, we have 
	\[
		h_n = P_n(z_1, \ldots, z_n) z_{n+1}  + \Big(  P_{n-1}(z_1, \ldots, z_{n-1}) +  \rho_n ({\eta}) \Big)^2 
	\ .
	\] 
	Notice that the vanishing of $ z_{n+1} $ implies that all partial derivatives $ \frac{\partial h_n}{\partial z_k} $ for $ k \in \{ 1, \ldots, n \} $ automatically are zero since $ \car(\KK) = 2 $.
	Hence, we get
	\begin{eqnarray}\label{singularlocus}
		\Sing (\Spec (\cA_\c^\prin (C_n))_\eta) 
		& = V (\, z_{n+1}, \, 
		P_n (z_1, \ldots, z_n) , \,
		P_{n-1} (z_1, \ldots, z_{n-1}) + \rho_n(\eta) \, ) \nonumber
		\\[8pt]
		& =
		V (\, z_{n+1}, \, 
		w_n , \,
		P_{n-1} (z_1, \ldots, z_{n-1}) + \rho_n(\eta) \, ) 
		=: D \ , 
	\end{eqnarray} 
	where $ w_n := \rho_n(\eta) z_n +  P_{n-2}(z_1, \ldots, z_{n-2}) $ is a variable that can be used to replace $ z_n $
	since $ \rho_n (\eta) $ is invertible.  
	\\
	By Proposition~\ref{Prop:ContSing}, $ D $ is singular if and only if $ n-1 = 2m $ is even and $ \rho_n(\eta) = 1 $ 
	(again take $ \car(\KK) = 2 $ into account).
	Moreover, if $ \Sing(D) \neq \varnothing $, then $ D $ has an isolated singularity at the origin and $ D $ is isomorphic to a hypersurface singularity of type $ A_1 $. 
	In particular, locally at the origin, 
	$ P_{n-1} (z_1, \ldots, z_{n-1} ) $ is a unit and 
	we may choose local variables 
	$ (w_1, \ldots, w_{n-1}) $ coming from $ (z_1, \ldots, z_{n-1}) $ 
	as well as $ x_n := P_n (z_1, \ldots, z_n) = P_{n-1}(z_1, \ldots, z_{n-1}) z_n + P_{n-2} (z_1, \ldots, z_{n-2}) $ replacing $ z_n $
	such that 
	\[
		h_n =
		x_n z_{n+1} + \Big(  \sum_{i=1}^m w_{2i-1} w_{2i} \Big)^2 
		\ .
	\]
	On the other hand, locally at a point of $ D $ which is different from the origin, the variety 
	$ \Spec (\cA_\c^\prin (C_n)_\eta) $ is isomorphic to a cylinder over a hypersurface singularity of type $ A_1 $, $ V(x_n z_{n+1}  + x^2  ) \subseteq \Spec(\bs{\kappa(\eta)}[x_n, z_{n+1}, x] ) $,
	where we choose $ x :=  P_{n-1} (z_1, \ldots, z_{n-1}) + \rho_n(\eta) $ as local variable. 
	\\
	Finally, if $ \Sing(D) = \varnothing $, then at every point of $ D $, the situation is the same as in the previous paragraph.

	As in the proof of Theorem~\ref{Thm:A_prin}, we deduced local hypersurface presentations for which the respective defining polynomial has integer coefficients.
	Hence, we do not have to take a $ p $-basis of $ \bs{\kappa(\eta)} $ into account.  	
	This finishes the proof of Theorem~\ref{Thm:C_prin}.
\end{proof}

\subsection{Type $ \boldsymbol{D_n} $}

We turn our attention to cluster algebras with principal coefficients of finite cluster type $ D_n $
with $ n \geq 4 $. 

We choose the initial labeled seed $ \Sigma = (\x, B ) $ with exchange matrix $ B $ given by 
Theorem~\ref{Thm:FinTypClass}
as the $ n \times n $ matrix
\[ 
	B = 
	\begin{pmatrix} 
	0 & 1 & \\
	-1 & \ddots & \ddots   \\
	& \ddots & 0 & 1  \\
	&  & -1 & 0 & 1 & 1\\
	&  &  & -1 & 0 & 0  \\
	&  &  & -1 & 0  & 0 \\
	\end{pmatrix}  
	\ .
\] 
The corresponding presentation of the cluster algebra $ \cA_\c^\prin (D_n) := \cA_\c^\prin (\Sigma ) $ 
resulting from Theorem~\ref{Thm:Presentation} is
\begin{equation}
	\label{eq:Dn_exchange} 
	\cA^\prin_\c(D_n) \cong
	\KK_\c [\x, \y]
	\hspace{-0.15cm}\raisebox{-0.2cm}{$\bigg/$}
	\hspace{-0.25cm}\raisebox{-0.85cm}%
	{$\left\langle \,
		\begin{minipage}{5.2cm}
			$ x_1 y_1 - c_1 -  x_2 
			\\[2pt]
			x_k y_k - c_k x_{k-1} -  x_{k+1} 
			\\[2pt]
			x_{n-2} y_{n-2} - c_{n-2} x_{n-3} -  x_{n-1} x_{n} 
			\\[2pt]
			x_{\ell} y_{\ell} - c_{\ell} x_{n-2} -  1 $
		\end{minipage}
		\, \Bigg| \,\, 
		\begin{minipage}{3cm}
			$ k \in \{ 2, \ldots, n-3 \} 
			\\ 
			\ell \in \{ n-1, n \} $
		\end{minipage}
		\, \right\rangle
		$} 
\end{equation} 
with $(\c,\x,\y) = (c_1, \ldots, c_n; x_1, \ldots, x_n; y_1, \ldots, y_n) $.

\begin{lemma}
	\label{Lem:Dn_2_eq} 
	There exists an isomorphism \bs{over $ \KK_\c$}
	\[
	\cA^\prin_\c(D_n) \cong 
	\KK_\c[z_1, \ldots, z_{n-2}, u_1, u_2, u_3, u_4 ]/ \langle h_1, h_2 \rangle 
	\ ,
	\]
	where
	\[
	\begin{array}{l}
		h_1 := 
		u_1 u_2 - u_3 u_4  -
		\lambda_{n-1}(\c)^{-1} u_2 u_4 
		\big(
		u_1 u_3 + P_{n-3} (z_1,\ldots, z_{n-3})  \big)
		\ ,
		
		\\[8pt]
		
		h_2 := 
		u_3 u_4  - P_{n-2} (z_1,\ldots, z_{n-2}) - \lambda_{n-1}(\c)
		\ . 
	\end{array} 
	\]
	for $ \lambda_{n-1}(\c) $ as defined in
	Definition~\ref{Def:ContPoly_lamda_n}(2)
	and $ P_{n-2}, P_{n-3} $ the continuant polynomials of Section~\ref{Subsec:Cont}.
\end{lemma}

\begin{proof}
	Analogous to before, we introduce 
	\[
	\widetilde x_k = x_k \lambda_{k-1}(\c)
	\hspace{10pt}
	\mbox{ and } 
	\hspace{10pt}
	\widetilde y_k = y_k \lambda_k(\c) \lambda_{k-1}(\c)^{-1},
	\hspace{10pt}
	\mbox{for } 
	k \in \{ 1, \ldots, n-2 \} \ ,	
	\]
	so that the first $ k - 3 $ exchange relations in \eqref{eq:Dn_exchange} become 
	(after multiplication by an invertible factor)
	\begin{equation} 
		\label{eq:Dn_afterchange_1} 
		\widetilde x_1 \widetilde y_1  - 1  - \widetilde x_2 
		\ = \ 
		\widetilde x_{k} \widetilde y_{k}  
		- \widetilde x_{k-1}  
		- \widetilde x_{k+1} 
		\ = \ 
		0
		\ , 
		\hspace{10pt}
		\mbox{ for } k \in \{ 2, \ldots, n -3 \}
		\ . 
	\end{equation} 
	 Notice that we left $ x_{n-1}, x_n, y_{n-1}, y_n $ unchanged so far. 
	 For the polynomial determining the remaining exchange relations, we get 
	 \[
	 \begin{array}{l}
	 	x_{n-2} y_{n-2} - c_{n-2} x_{n-3} -  x_{n-1} x_{n} 
	 	= 
	 	\lambda_{n-2}(\c)^{-1} 
	 	\big( 
	 		\widetilde x_{n-2} \widetilde y_{n-2} 
	 	- \widetilde x_{n-3} 
	 	- \lambda_{n-2}(\c) x_{n-1} x_n 
	 	\big) 
	 	\ ,
	 	
	 	\\[5pt]
	 	x_{n-1} y_{n-1} - c_{n-1} x_{n-2} -  1 = \lambda_{n-1}(\c)^{-1} 
	 	\big(
	 		\lambda_{n-1}(\c) x_{n-1} y_{n-1} - \widetilde x_{n-2} - \lambda_{n-1}(\c)
	 	\big)
	 	\ ,
	 	
	 	\\[5pt]
	 	x_{n} y_{n} - c_{n} x_{n-2} -  1 
	 	= 
	 	\lambda_{n-1}(\c)^{-1}
	 	\big( 
	 		\lambda_{n-1}(\c) x_{n} y_{n} 
	 		- \widetilde x_{n-2} 
	 		- \lambda_{n-1}(\c) 
	 	\big)
	 	\ . 	
	 \end{array} 
	 \]
	 We introduce 
	 \[ 
	 	\widetilde x_{n-1} :=  \lambda_{n-2}(\c) x_{n-1}
	 	\ ,
	 	\ \ \ \ \
	 	\widetilde y_{n-1} :=  \lambda_{n-2}(\c)^{-1} \lambda_{n-1}(\c) y_{n-1} 
	 	\ ,
	 	\ \ \ \ \
	 	\widetilde y_n := \lambda_{n-1}(\c) y_{n} 
	 \]
	 so that the last three exchange relations can be replaced by 
	 \begin{equation}
	 	\left\{ 
	 	 \ \ 
	 	\begin{array}{l} 
	 	g_3 := 
	 	\widetilde x_{n-2} \widetilde y_{n-2} 
	 	- \widetilde x_{n-3} 
	 	- \widetilde x_{n-1} x_n 
	 	=  0 
	 	\ , 
	 	
	 	\\[8pt]  
	 	
	 	g_2 := 
	 	\widetilde x_{n-1} \widetilde y_{n-1} - \widetilde x_{n-2} - \lambda_{n-1}(\c)
	 	=  0 
	 	\ , 
	 	
	 	\\[8pt]  
	 	
	 	g_1 := 
	 	x_{n} \widetilde y_{n} 
	 	- \widetilde x_{n-2} 
	 	- \lambda_{n-1}(\c)
	 	= 0 
	 	\ . 
	 	\end{array} 
 		\right. 
	 \end{equation} 
	We abuse notation and drop the tilde. 
	Furthermore, we may replace $ g_1 $ and $ g_3 $ by 
	\[
	g_1' := g_1 - g_2 = 
	x_n y_n - x_{n-1} y_{n-1}   
	\]
	\[
		g_3' := g_3 + y_{n-2} g_2 
		= 
		- x_{n-3} 
		- x_{n-1} (x_n  - y_{n-2} y_{n-1} )
		- 
		 \lambda_{n-1}(\c) y_{n-2}
	\]
	respectively. 
	By introducing 
	\[
		(u_1, u_2, u_3, u_4)
		:= 
		( \, x_n  - y_{n-2} y_{n-1},
		\, 
		y_n, 
		\, 
		x_{n-1}, 
		\, 
		y_{n-1} 
		\, ) 
		\ , 
	\]
	we see that $ g_3' = 0 $ is equivalent to 
	\begin{equation}
		\label{eq:Dn_sub_y_n-2}
		y_{n-2} = -\lambda_{n-1}(\c)^{-1} 
		\big(
			x_{n-3} + u_1 u_3 
		\big)
	\end{equation} 
	and that $ g_1' $ becomes
	(using the relation \eqref{eq:Dn_sub_y_n-2}) 
	\[
	\begin{array}{rcl} 
		g_1' 
		& = & 
		x_n y_n - x_{n-1} y_{n-1}   
		\ = \  
		u_1 u_2 - u_3 u_4  + u_2 u_4 y_{n-2}
		\ = \ 
		\\[8pt]
		& = &
		u_1 u_2 - u_3 u_4  -
		\lambda_{n-1}(\c)^{-1} u_2 u_4 (u_1 u_3 + x_{n-3} )
		\ . 
	\end{array} 
	\]
	As before, \eqref{eq:Dn_afterchange_1} leads to the substitution 
	$ x_k = P_k (x_1, y_1, \ldots, y_{k-1}) $, 
	for $ k \in \{ 2, \ldots, n-2 \} $. 
	This provides
	\[
		\begin{array}{l}
			h_1 := g_1' 
			=
			u_1 u_2 - u_3 u_4  -
			\lambda_{n-1}(\c)^{-1} u_2 u_4 (u_1 u_3 + P_{n-3} (x_1,y_1,\ldots, y_{n-4})  )
			\ ,
			
			\\[8pt]
			
			h_2 := g_2 
			=
			u_3 u_4  - P_{n-2} (x_1,y_1,\ldots, y_{n-3}) - \lambda_{n-1}(\c)
			\ . 
		\end{array} 
	\]
	Taking into account \eqref{eq:Dn_sub_y_n-2}, 
	finishes the proof. 
\end{proof}

As in \cite[Lemma~4.5]{BFMS}, we first determine the singular locus. 

\begin{lemma}
	\label{Lem:Dn_Sing}
	The singular locus of the fiber $\Spec (\cA_\c^\prin (D_n) )_\eta$ (for $ \eta \in S $) of the family defined by
	$ \phi \colon \Spec ( \cA^\prin_\c(D_n)) \to S $ is 
	$ Y_0 \cup Y_1 \cup Y_2 \cup Y_3 \cup Y_4 $,
	where the components identify along the presentation of Lemma~\ref{Lem:Dn_2_eq} with
	\[
		\begin{array}{l} 
		Y_0 := V (u_1, u_2, u_3, u_4 , P_{n-2} (z_1, \ldots, z_{n-2}) + \lambda_{n-1}(\eta)) 
		\,
		\\[10pt]
		Y_1 := V(u_2, u_3, u_4) \cap \Sing (V ( P_{n-2} (z_1, \ldots, z_{n-2}) + \lambda_{n-1}(\eta) ) )
		\ ,
		\\[10pt]
		Y_2 :=  V(u_1, u_3, u_4) \cap \Sing (V ( P_{n-2} (z_1, \ldots, z_{n-2}) + \lambda_{n-1}(\eta) ) )
		\ ,
		\\[10pt]
		Y_3 :=  V(u_1, u_2, u_4) \cap \Sing (V ( P_{n-2} (z_1, \ldots, z_{n-2}) + \lambda_{n-1}(\eta) ) )
		\ ,
		\\[10pt]
		Y_4 :=  V(u_1, u_2, u_3) \cap \Sing (V ( P_{n-2} (z_1, \ldots, z_{n-2}) + \lambda_{n-1}(\eta) ) )
		\ . 
	\end{array} 	
	\]
	Observe that $ \Sing (Y_0) = \bigcap_{i=1}^4 Y_i $ 
	and that $ Y_i \neq \varnothing $ for $ i \in \{ 1, \ldots 4 \} $ if and only if $ n - 2 = 2m $ is even and $ \lambda_{n-1}(\eta) = (-1)^{m+1} $, 
	by Proposition~\ref{Prop:ContSing}.	
\end{lemma}

\begin{proof}
	As before, we apply the Jacobian criterion. 
	We use the presentation,
	which we deduced in Lemma~\ref{Lem:Dn_2_eq}, 
	$ 
	\cA^\prin_\c(D_n) \cong 
	\KK_\c[z_1, \ldots, z_{n-2}, u_1, u_2, u_3, u_4 ]/ \langle h_1, h_2 \rangle 
	$.
	We have to consider the $ 2 $-minors of the Jacobian matrix of the fiber above $ \eta $
	\[
		\Jac(D_n) := \Jac( h_1, h_2; z_1, \ldots, z_{n-2}, u_1, \ldots, u_4 ) . 
	\]
	The minors corresponding to the derivatives with respect to $ (u_1, u_4) $ resp.~$ (u_1,u_3 ) $ provide the equations 
	\begin{equation}
		\label{eq:Dn_Sing_1}
		u_2 u_3 ( 1 - \lambda_{n-1}(\eta)^{-1} u_3 u_4 )
		= 
		u_2 u_4 ( 1 - \lambda_{n-1}(\eta)^{-1} u_3 u_4 ) = 0 
	\end{equation}
	for the singular locus.
	\\
	First, assume $ u_2 = 0 $. 
	Then $ h_1 = h_2 = 0 $ is equivalent to 
	\[
		u_3 u_4 = P_{n-2} (z_1, \ldots, z_{n-2}) + \lambda_{n-1}(\eta) = 0 
	\]
	and $ \Jac(D_n) $ becomes
	\[
	\begin{pmatrix}
			- \lambda_{n-1}^{-1} u_2 u_4 \frac{\partial P_{n-3}}{\partial z_k}
		& 0
		& 0
		& u_1 - \lambda_{n-1}^{-1} u_4 (u_1 u_3 + P_{n-3}) 
		& - u_4  
		&  - u_3  
		\\[10pt]
		- \frac{\partial P_{n-2}}{\partial z_k}
		& - P_{n-3}
		& 0
		& 0 
		& u_4 
		& u_3 
	\end{pmatrix}
	\ , 
	\]
	where the first column has to be considered as many columns with
	$ k \in \{ 1, \ldots, n- 3 \} $
	and
	we abbreviate $ P_{n-2} := P_{n-2} (z_1, \ldots, z_{n-2}) , 
	P_{n-3} := P_{n-3} (z_1, \ldots, z_{n-3}) $, and $ \lambda_{n-1} := \lambda_{n-1}(\eta) $.  	
	The minors with respect to $ (u_2, u_3) $ resp.~$ (u_2, u_4) $ provide the additional equations
	\[
		u_3 (u_1 - \lambda_{n-1}^{-1} u_4 (u_1 u_3 + P_{n-3}))
		= 
		u_4 (u_1 - \lambda_{n-1}^{-1} u_4 (u_1 u_3 + P_{n-3})) = 0 
	\]
	for the singular locus.
	\\
	If we consider the case $ u_3 = u_4 = 0 $ (additional to $ u_2 = 0 $), 
	then we find the two components
	\[
		\begin{array}{l} 
		Y_0 = V (u_1, u_2, u_3, u_4 , P_{n-2} (z_1, \ldots, z_{n-2}) + \lambda_{n-1}(\eta)) 
		\\[10pt]
		Y_1 = V(u_2, u_3, u_4) \cap \Sing (V ( P_{n-2} (z_1, \ldots, z_{n-2}) + \lambda_{n-1}(\eta) ) )
		\end{array} 
	\]
	of $ \Sing (\Spec (\cA_\c^\prin (D_n))_\eta) $. 	
	On the other hand, since $ u_3 u_4 = 0 $, at most one of $ u_3, u_4 $ can be invertible. 
	This leads to the next two components of the singular locus, 
	\[
	\begin{array}{l}
		Y_3 =  V(u_1, u_2, u_4) \cap \Sing (V ( P_{n-2} (z_1, \ldots, z_{n-2}) + \lambda_{n-1}(\eta) ) )
		\ ,
		\\[10pt]
		Y_4 =  V(u_1, u_2, u_3) \cap \Sing (V ( P_{n-2} (z_1, \ldots, z_{n-2}) + \lambda_{n-1}(\eta) ) )
		\ . 
	\end{array} 	
	\]
	This ends the  first case ($ u_2 = 0 $) which emerged from \eqref{eq:Dn_Sing_1}.
	
	Next, suppose that $ u_3 = u_4 = 0 $ and that $ u_2 $ is invertible 
	(since we want to avoid seeing components that we already determined). 
	Under these hypotheses, $ h_1 = h_2 = 0 $ is equivalent to 
	\[
		u_1 = P_{n-2} (z_1, \ldots, z_{n-2} ) + \lambda_{n-1} (\eta) = 0 
		\ . 
	\]
	Thus, $ \Jac(D_n) $ becomes
	\[
	\begin{pmatrix}
		0
		& 0
		& u_2
		& 0 
		& 0  
		&  -\lambda_{n-1}(\eta)^{-1} u_2 P_{n-3} (z_1,\ldots, z_{n-3}) 
		\\[10pt]
		- \frac{\partial P_{n-2} (z_1,\ldots, z_{n-2})}{\partial z_k}
		& - P_{n-3} (z_1,\ldots, z_{n-3})
		& 0
		& 0 
		& 0 
		& 0 
	\end{pmatrix}
	\ , 
	\]
	and from this, we get the fifth component of the singular locus,
	\[
		Y_2 = V(u_1, u_3, u_4) \cap  \Sing (V ( P_{n-2} (z_1, \ldots, z_{n-2}) + \lambda_{n-1}(\eta) ) )
		\ . 
	\]
	Finally, we consider the third case arising from \eqref{eq:Dn_Sing_1},
	where $ 1 - \lambda_{n-1} (\eta)^{-1} u_3 u_4 = 0 $. 
	Then $ h_2 = 0 $ is equivalent to $ P_{n-2} (z_1, \ldots, z_{n-2}) = 0 $. 
 	On the other hand, the $ 2 $-minor of $ \Jac(D_n) $ coming from the derivatives by $ (z_{n-2}, u_2 ) $ provides that we have to have $ P_{n-3} (z_1, \ldots, z_{n-3}) = 0 $. 
 	A simple induction on $ n $ (using the recursion in Lemma~\ref{Lem:ContFacts}(3) for $ k = n- 1 $) 
 	shows that it is impossible to have 
 	$ P_{n} (z_1, \ldots, z_{n}) = 0 $
 	and $ P_{n-1} (z_1, \ldots, z_{n-1}) = 0 $ at the same time. 
 	This implies that we do not get any further components for the singular locus.
 	
 	Notice that at the singularities are contained in $ V (u_2, u_4) $
 	and that all entries of the row of $ \Jac(D_n) $ corresponding to $ h_1 $ vanish at a singular point. 
 	Therefore, even if $ \lambda_{n-1} (\eta) $ is part of a $ p $-basis for $ \bs{\kappa(\eta)}  $, 
 	its derivatives have no impact on the computations. 
 	In other words, our considerations in the relative setting determined the singular locus of $ \Spec (\cA_\c^\prin (D_n) )_\eta $.  
\end{proof}

For a better clarity, we unpack Lemma~\ref{Lem:Dn_Sing} into cases, for which we provide a concrete description of the singular locus.
As an immediate consequence of Lemma~\ref{Lem:Dn_Sing}, we have:

\begin{cor}
	\label{Cor:Dn_Sing}
	We have the following cases for $ \Sing ( \Spec(\cA_\c^\prin (D_n))_\eta) $:
	\begin{enumerate}
		\item[(a)] 
		If $ n = 4 $ and $ \lambda_3 (\eta) = 1 $,
		then $ \Sing ( \Spec(\cA_\c^\prin (D_4))_\eta) $ identifies with the six coordinate axes in $ \AA_{\bs{\kappa(\eta)}}^6 $ along the isomorphism of Lemma~\ref{Lem:Dn_2_eq}.
		
		\item[(b)] 
		If $ n-2 = 2m $ is even, $ n > 4 $, and  $ \lambda_{n-1} (\eta) = (-1)^{m+1} $,
		then $ \Sing ( \Spec(\cA_\c^\prin (D_n))_\eta) $ consists of 
		four regular, irreducible components $ Y_1, \ldots, Y_4 $ of dimension one 
		and one singular, irreducible component $ Y_0 $ of dimension $ n - 3 $
		whose singular locus is a closed point 
		coinciding with the intersection $ \bigcap_{i=1}^4 Y_i $.

		\item[(c)] 
		Otherwise
		(i.e., if $ n $ is odd, or if $ n $ is even and $ \lambda_{n-1} (\eta) \neq 1 $),
		then the singular locus of the fiber, $ \Sing ( \Spec(\cA_\c^\prin (D_n))_\eta) \cong Y_0 $, is irreducible, regular, and of dimension $ n-3 $.
	\end{enumerate}
\end{cor}

Let us recall and prove the remaining part of Theorem~\ref{Thm:D}.
	
\begin{Thm}
	\label{Thm:D_prin} 
		We uses the Cases (a)--(c) of Corollary~\ref{Cor:Dn_Sing} to formulate the classification of the singularities. 
		\begin{itemize}
			\item 
			\underline{Case (c):} 
			Along its regular irreducible singular locus,
			$ \Spec(\cA_\c^\prin (D_n))_\eta $ is isomorphic to a cylinder 
			over a $ 3 $-dimensional hypersurface singularity of type $ A_1 $.
			
			\item 
			\underline{Cases (b):} 
			Let $ 0 \in Y_0 $ be the singular point of $ Y_0 $.
			Along $ Y_0 \setminus \{ 0 \} $, the situation is the same as in Case (c).
			Further, along $ Y_i \setminus \{ 0 \} $, for $ i \in \{ 1, \ldots, 4 \} $, 
			$ \Spec(\cA_\c^\prin (D_n))_\eta $ is isomorphic to an $ n $-dimensional hypersurface singularity of type $ A_1 $. 
			Finally, locally at $ 0 $, 
			$ \Spec(\cA_\c^\prin (D_n))_\eta $ is isomorphic to the intersection of two hypersurface singularity of type $ A_1 $ and  
			$ Y_0 $ is isomorphic to a $ ( n- 3) $-dimensional hypersurface singularity of type $ A_1 $.

			\item 
			\underline{Cases (a):}
			The situation is the same as in Case (b) with the exception 
			that $ Y_0 $ is the union of two lines here 
			and locally at $ 0 $, 
			 $ \Spec(\cA_\c^\prin (D_n))_\eta $ is isomorphic to the intersection of a hypersurface singularity of type $ A_1 $ with a divisor of the form $ V (xy) $.
		\end{itemize} 
\end{Thm}

\begin{proof}
	By Lemma~\ref{Lem:Dn_Sing},
	it remains to classify the singularities. 
	We continue to use the presentation deduced in Lemma~\ref{Lem:Dn_2_eq}, as in the proof of Lemma~\ref{Lem:Dn_Sing}. 
	\\
	First, we assume $ \Sing(Y_0) = \varnothing $ (Case (c)).
	Then, $ V ( P_{n-2} (z_1, \ldots, z_{n-2}) + \lambda_{n-1}(\eta) )  $ is regular. 
	This implies that the hypersurface $ H := V (h_2) = V ( u_3 u_4 - P_{n-2} (z_1, \ldots, z_{n-2}) - \lambda_{n-1}(\eta) ) $ is regular as well. 
	On the other hand, we have 
	\[
		h_1 = u_2 \Big( u_1  (1 - \lambda_{n-1}(\eta)^{-1} u_3 u_4 ) 
		- \lambda_{n-1}(\eta)^{-1} u_2 u_4  P_{n-3} (z_1,\ldots, z_{n-3}) \Big) - u_3 u_4 
	\] 
	and since $ Y_0 = V (u_1, u_2, u_3, u_4 , P_{n-2} (z_1, \ldots, z_{n-2}) + \lambda_{n-1}(\eta)) $
	the factor $ 1 - \lambda_{n-1}(\eta)^{-1} u_3 u_4 $ is a unit (locally at a point of $ Y_0 $)
	so that we may introduce the new local variable 
	\[
		w_1 := u_1  (1 - \lambda_{n-1}(\eta)^{-1} u_3 u_4 ) 
		- \lambda_{n-1}(\eta)^{-1} u_2 u_4  P_{n-3} (z_1,\ldots, z_{n-3})
		\ .  
	\]
	Therefore, locally at $ Y_0 $, 
	we have that $ \Spec(\cA_\c^\prin (D_n) )_\eta $ is isomorphic to cylinder over the $ A_1 $-hypersurface singularity $ V (u_1 w_1 - u_3 u_4 ) \subseteq \AA_\KK^4 $.
	
	Suppose we are in Cases (a) or (b), $ \Sing(Y_0) \neq \varnothing $ (i.e., $ n-2 = 2m $ and $ \lambda_{n-1} (\eta) = (-1)^{m+1} $). 
	Along $ Y_0 \setminus \{ 0 \} $, the situation is the same as in the case $ \Sing (Y_0) = \varnothing $.
	\\
	By Proposition~\ref{Prop:ContSing}, 
	$ \Sing (V ( P_{n-2} (z_1, \ldots, z_{n-2}) + \lambda_{n-1}(\eta) ) ) = V (z_1, \ldots, z_{n-2} ) $
	and locally at $ z_1 = \ldots = z_{n-2} = 0 $, there exist local variables $ (t_1, \ldots, t_{n-2} ) $
	such that 	
	\[
		P_{n-2} (z_1, \ldots, z_{n-2}) + \lambda_{n-1}(\eta) = \sum_{i=1}^m t_{2i-1} t_{2i} 
		\ .
	\]
	Further, locally at $ 0 $, 
		we may introduce 
		\[ 
			w_1 := u_1 ( 1 - \lambda_{n-1}(\eta)^{-1} u_3 u_4 ) - \lambda_{n-1}(\eta)^{-1} u_2 u_4 P_{n-3} (z_1,\ldots, z_{n-3}) 
		\] 
		as local variable substituting $ u_1 $
		such that we get $ h_1 = w_1 u_2 - u_3 u_4  $.
	This implies the statement on the classification of the singularities of $ \Spec (\cA_\c^\prin (D_n))_\eta $ and $ Y_0 $ locally at $ 0 $.
	\\
	On the other hand, 
	$ Y_i $ is the $ u_i $-axis, for $ i \in \{ 1, \ldots, 4 \} $. 
	Locally at a point of $ Y_1 \setminus \{ 0 \} $, 
	the factor $ \epsilon_1 := u_1 -
	\lambda_{n-1}(\eta)^{-1} u_2 u_4 
	\big(
	u_1 u_3 + P_{n-3} (z_1,\ldots, z_{n-3})  \big) $ is a unit and 
	we have 
	$
		h_1 = 
		\epsilon_1 u_2 - u_3 u_4  = 0 
	$.
	Thus, 
	we may eliminate $ u_2 $ by substituting $ \epsilon_1 u_2 = u_3 u_4 $
	and (locally at $ Y_1 \setminus \{ 0 \} $) 
	$ \Spec (\cA_\c^\prin (D_n))_\eta $ is isomorphic to the hypersurface $ h_2 = u_3 u_4 -  \sum_{i=1}^m t_{2i-1} t_{2i}  = 0 $, which has a singularity of type $ A_1 $. 
	\\
	The analogous argument applies for $ Y_2 \setminus \{ 0 \} $. 
	\\
	At the situation along $ Y_3 \setminus \{ 0 \} $, we have that $ u_3 $ is invertible and by introducing the local variables $ w_1 := u_1 u_3, w_2 := u_2 u_3^{-1}, w_3 := u_3 u_4 $,
	one may deduce that the variable $ w_4 $ may be eliminated via the relation $ h_1 = 0 $ analogous to the discussed case at $ Y_1 \setminus \{ 0 \} $ and the same arguments imply the assertion along the $ u_3 $-axis excluding the origin.
	\\
	The missing case of the singularities at $ Y_4 \setminus \{ 0 \} $ goes analogous to  $ Y_3 \setminus \{ 0 \} $. 	 
\end{proof}

\subsection{Type $ \boldsymbol{E_6, E_7, E_8} $}
\label{Subsec:En}
Our journey continues with the exceptional cases $ E_6, E_7, E_8 $.
In contrast to \cite[Section~4.3]{BFMS}, we do not treat them one-by-one, but set them in a more general context. 
Let $ n \geq 6 $ and consider the $ n \times n $ exchange matrix 
\[
	B = 
	\begin{pmatrix} 
		0 	& 1  & \\
		-1 	& \ddots  & \ddots   \\
		& \ddots 	& 0  & 1   \\
		&& -1 &  0 & 1  & 1  & 0 \\
		&&    & -1 & 0  & 0  & 0\\
		&&    & -1 & 0  & 0  & 1  \\
		&&    &  0 & 0  & -1 & 0 \\
	\end{pmatrix} 
	\ . 
\]
We use the notation $ \cA_\c^\prin (E_n ) $ for the resulting cluster algebra with principal coefficients. 
Notice that it is not a common notation to write $ E_n $ for $ n \geq 9 $, but we use it here for simplicity. 
The tree for $E_n$ corresponds to the Dynkin diagram of the tree singularity $T_{2,3,n-3}$ according to Gabrielov's classification \cite{Gabrielov74}.  
By Theorem~\ref{Thm:FinTypClass}, the cases of finite cluster type are included when $ n \in \{ 6, 7, 8 \} $.
\\
For $ n \geq 9 $, the cluster algebra is not of finite cluster type, but it is part of a generalization of star cluster algebras (as introduced in~\cite[Section~6]{BFMS})
since the quiver corresponding to $ B $ (Remark~\ref{Rk:Quiver_view}) is a star with three rays of different lengths.  
The full generalizations of star cluster algebras and the investigation of their singularity theory is another interesting topic,
but due to its complexity, it will be the topic of future work. 

Following Theorem~\ref{Thm:Presentation}, there is the following presentation of $ \cA_\c^\prin (E_n) $:
\begin{equation}
	\label{eq:En_pres}
	\cA_\c^\prin (E_n)
	\cong \KK_\c [\x,\y] / I \ , 
\end{equation}
where $ I $ is the ideal generated by 
\begin{equation}
	\label{eq:En_pres_gen}
	\left\{ 
		\begin{array}{rll}
			g_1 := & x_1 y_1 - c_1 - x_2
			\ ,
			\\[3pt]
			g_k := & x_k y_k - c_k x_{k-1} -  x_{k+1}
			\ ,
			& k \in \{ 2, \ldots, n-4\} 
			\ ,	
			\\[3pt]
			g_{n-3} := & x_{n-3} y_{n-3} - c_{n-3} x_{n-4} - x_{n-2} x_{n-1} 
			\ ,
			\\[3pt]
			g_{n-2} := & x_{n-2} y_{n-2} - c_{n-2} x_{n-3} - 1 
			\ ,	
			\\[3pt]
			g_{n-1} := & x_{n-1} y_{n-1} - c_{n-1} x_{n-3} - x_n
			\ , 
			\\[3pt]
			g_{n} := & x_{n} y_{n} - c_n x_{n-1} - 1 
			\ ,			
		\end{array}
	\right. 
\end{equation}
where $(\c,\x,\y) = (c_1, \ldots, c_n; x_1, \ldots, x_n; y_1, \ldots, y_n) $.

\begin{lemma}
	\label{Lem:En_2_eq} 
	There is an isomorphism \bs{over $ \KK_\c$}
	\[
	\cA^\prin_\c(E_n) \cong 
	\KK_\c[z_1, \ldots, z_{n-2}, u_1, \ldots, u_5 ]/ \langle h_1, h_2, h_3 \rangle 
	\ ,
	\]
	where
	\[
	\begin{array}{l}
		h_1 :=
		P_{n-2}(z_1, \ldots, z_{n-2}) - u_3 P_2 (u_1,u_2)
		\ ,
		
		\\[8pt]
		
		h_2 := 
		u_3 u_4  - P_{n-3} (z_1,\ldots, z_{n-3}) - \lambda_{n-2}(\c)
		\ ,
		
		\\[8pt]
		
		h_3 := 
		P_{3} (u_1, u_2, u_5)  - P_{n-3} (z_1,\ldots, z_{n-3})
		\ . 
	\end{array} 
	\]
	for $ \lambda_{n-2}(\c) $ as defined in
	Definition~\ref{Def:ContPoly_lamda_n}(2)
	and $ P_\ell $ the continuant polynomials of Section~\ref{Subsec:Cont}.
\end{lemma}

\begin{proof}
	We define
	\[
		\widetilde x_k := x_k \lambda_{k-1}(\c)
		\hspace{10pt}
		\mbox{ and } 
		\hspace{10pt}
		\widetilde y_k := y_k \lambda_k(\c) \lambda_{k-1}(\c)^{-1},
		\hspace{10pt}
		\mbox{for } 
		k \in \{ 1, \ldots, n-3 \} \ ,	
	\]
	so that the first $ k - 4 $ exchange relations in \eqref{eq:En_pres_gen} become 
	(after multiplication by an invertible factor)
	\begin{equation} 
		\label{eq:En_afterchange_1} 
		\widetilde x_1 \widetilde y_1  - 1  - \widetilde x_2 
		\ = \ 
		\widetilde x_{k} \widetilde y_{k}  
		- \widetilde x_{k-1}  
		- \widetilde x_{k+1} 
		\ = \ 
		0
		\ , 
		\hspace{10pt}
		\mbox{ for } k \in \{ 2, \ldots, n -4 \}
		\ . 
	\end{equation} 
	Furthermore, by setting 
	\[ 
	\begin{array}{ll} 
		\widetilde x_{n-2} := x_{n-2} c_{n}^{-1}\lambda_{n-3}(\c)
		\ , \ \
		&
		\widetilde y_{n-2} := y_{n-2} c_n \lambda_{n-2}(\c) \lambda_{n-3}(\c)^{-1}
		\ ,
		
		\\[8pt]
		\widetilde x_{n-1} := c_n x_{n-1} 
		\ , \ \ 
		& \widetilde y_{n-1} := y_{n-1} c_n^{-1} c_{n-1}^{-1}\lambda_{n-4}(\c)
		\ ,
		
		\\[8pt] 
		
		\widetilde x_n := x_n c_{n-1}^{-1} \lambda_{n-4}(\c)
		\ , \ \ 
		&
		\widetilde y_n := y_n c_{n-1} \lambda_{n-4}(\c)^{-1}
		\ ,
		
	\end{array} 
	\]  
	the remaining relations are equivalent to
	\[
		\begin{array}{l} 
		h_1 := \widetilde x_{n-3} \widetilde y_{n-3} - \widetilde x_{n-4} - \widetilde x_{n-2} \widetilde x_{n-1} = 0
		\ , 
		
		\\[8pt]
		
		h_2 :=  \widetilde x_{n-2} \widetilde y_{n-2} - \widetilde x_{n-3} - \lambda_{n-2} (\c)
		= 0 
		\ ,

		\\[8pt]
		
		h_3 :=  \widetilde x_{n-1} \widetilde y_{n-1} - \widetilde x_{n-3} - \widetilde x_n
		= 0 
		\ ,
		
		\\[8pt]
		
		\widetilde x_n \widetilde y_n - \widetilde x_{n-1} - 1 = 0 
		\ . 
		\end{array}  
	\]
	The last equation leads to the substitution $ \widetilde x_{n-1} = P_2 (\widetilde x_n, \widetilde y_n) $.
	On the other hand, \eqref{eq:En_afterchange_1}  provides the substitutions
	$ \widetilde x_k = P_k (\widetilde x_1, \widetilde  y_1, \ldots, \widetilde  y_{k-1}) $, 
	for $ k \in \{ 2, \ldots, n-3 \} $. 
	Performing these substitutions to $ h_1, h_2, h_3 $ shows the assertion.
\end{proof} 

Using the deduced presentation, we prove the following results on the singularities 
which implies Theorem~\ref{Thm:E} in the special case $ n \in \{ 6,7,8\} $.

\begin{Thm}
	\label{Thm:E_prin}
	The fibers $ \Spec ( \cA_\c^\prin (E_n) )_\eta $ of the family $\phi \colon \Spec ( \cA_\c^\prin (E_n) ) \to S $ is singular if and only if 
	$ n - 3 = 2m $ is even and $ \lambda_{n-2} (\eta)  = (-1)^{m+1} $.
	In the singular case, the singular locus is a regular irreducible surface $ Y $ 
	and locally at $ Y $, the variety $ \Spec ( \cA_\c^\prin (E_n) )_\eta $ is isomorphic to a cylinder over an $ (n-2) $-dimensional hypersurface singularity of type $ A_1 $. 
\end{Thm}

\begin{proof}
	In order to determine the singular locus, 
	consider the presentation deduced in Lemma~\ref{Lem:En_2_eq} 
	$ \cA^\prin_\c(E_n) \cong 
	\KK_\c[z_1, \ldots, z_{n-2}, u_1, \ldots, u_5 ]/ \langle h_1, h_2, h_3 \rangle $.
	\\
	Let $ \Jac(E_n) := \Jac (h_1, h_2, h_3; z_1, \ldots, z_{n-2}, u_1, \ldots, u_5 ) $ be the corresponding Jacobian matrix of the fiber above $ \eta $.
	The singular locus $ \Sing ( \cA_\c^\prin (E_n)_\eta ) $ is determined by the vanishing of the $ 3 $-minors of $ \Jac(E_n) $. 
	The sub-matrix whose columns are given by derivatives with respect to $ (u_1, \ldots, u_5) $ is 
	\[
		\begin{pmatrix}
			- u_2 u_3 & - u_1 u_3 & - P_2 (u_1, u_2) & 0 & 0 
			\\
			0 & 0 & u_4 & u_3 & 0 
			\\
			u_2 u_5 - 1 & u_1 u_5 & 0 & 0 & u_1 u_2 - 1 
		\end{pmatrix}
	\ . 
	\] 
	It is impossible that $ u_2 u_5 - 1 $ and $ u_1 u_5 $ and $ u_1 u_2 - 1  $ vanish at the same time. 
	Because of this, the last row has a non-zero entry.
	Hence, we must have $ u_3 P_2 (u_1, u_2) = 0 $ at a singular point.
	\\
	Suppose $ u_3 = 0 $. 
	Taking this into account in $ h_2 = 0 $, we obtain that the latter is equivalent to $ P_{n-3} (z_1, \ldots, z_{n-3} ) + \lambda_{n-2} (\eta) = 0 $.
	In particular, $ P_{n-3} (z_1, \ldots, z_{n-3} )  = \frac{\partial h_1}{\partial z_{n-2}} $ is invertible and taking $ u_3 = 0 $ into account, we see that $ h_1 = 0 $ is equivalent to $ z_{n-2} = 0 $.
	Furthermore, since $ z_{n-2} $ appears only in $ h_1 $, we get that at a singular point, we conditions 
	\[
		u_4 = 0 
		\ \ \mbox{ and } \ \
		\frac{\partial P_{n-3} (z_1, \ldots, z_{n-3})}{\partial z_k} = 0 
		\ , 
		\ \ \ 
		\mbox{ for } 
		k \in \{ 1, \ldots, n-3 \}
		\ ,
	\]
	have to hold. 
	This implies that the only component possibly appearing in the singular locus of $ \Spec (\cA_\c^\prin (E_n))_\eta $ is
	\[
		Y := V(z_{n-2}, u_3, u_4, P_{3} (u_1, u_2, u_5) + \lambda_{n-2} (\eta)) \cap \Sing ( V (P_{n-3} (z_1, \ldots, z_{n-3} ) + \lambda_{n-2} (\eta) ) ).
	\]
	
	Before getting more into details with $ Y $, 
	let us discuss the case where $ P_2 (u_1, u_2) = 0 $.
	As an immediate consequence, we see that $ h_1 = 0 $ is equivalent to $ P_{n-2} (z_1, \ldots, z_{n-2} ) = 0 $ under this additional hypothesis.
	In order to avoid detecting $ Y $ again, we may assume that $ u_3 $ is invertible. 
	\\
	Recall that there exists at least one $ \ell \in \{ 1, 2, 5 \} $  such that the last entry of the column corresponding to the derivatives by $ u_\ell $ is invertible. 
	By considering the minor coming from the derivatives by $ (u_\ell, u_4, z_{n-2} ) $,
	we obtain the condition
	$ P_{n-3}(z_1, \ldots, z_{n-3}) = 0 $
	has to hold in order to have a singularity (not lying on $ u_3 = 0 $).
	\\
	By the same argument as outlined at the end of the proof of Lemma~\ref{Lem:Dn_Sing},
	it is impossible to have   $ P_{n-2} (z_1, \ldots, z_{n-2} ) = 0 $
	and $ P_{n-3}(z_1, \ldots, z_{n-3}) = 0 $.
	Therefore, there is no other component than $ Y $ in the singular locus. 
	
	Let us take a closer look at $ Y $.
	First, $ \Sing ( V (P_{n-3} (z_1, \ldots, z_{n-3} ) + \lambda_{n-2} (\eta) ) ) $ is non-empty if and only if 
	by $ n - 3 = 2m $ is even and $ \lambda_{n-2} (\eta)  = (-1)^{m+1} $ by Proposition~\ref{Prop:ContSing}.
	In the singular case, $ V (P_{n-3} (z_1, \ldots, z_{n-3} ) + \lambda_{n-2} (\eta) ) $
	has an isolated singularity of type $ A_1 $ at $ V (z_1, \ldots, z_{n-3}) $.
	\\
	On the other hand, $ V ( P_{3} (u_1, u_2, u_5) + \lambda_{n-2} (\eta) ) $ is regular (by Proposition~\ref{Prop:ContSing}).
	Therefore, $ Y $ is non-empty if and only if  $ n - 3 = 2m $ is even and $ \lambda_{n-2} (\eta)  = (-1)^{m+1} $,
	and if $ Y \neq \varnothing $, then $ Y = V (z_1, \ldots, z_{n-2}, u_3, u_4, P_{3} (u_1, u_2, u_5) + \lambda_{n-2} (\eta)) $ is a regular surface. 
	
	In order to classify the singularity locally at $ Y $,
	observe that 
	\[  
		h_1 = z_{n-2} P_{n-3}(z_1, \ldots, z_{n-3}) - P_{n-4}(z_1, \ldots, z_{n-4}) - u_3 P_2 (u_1,u_2) 
	\] 
	and since $ P_{n-3}(z_1, \ldots, z_{n-3}) $ is invertible locally at $ Y $, 
	we may eliminate the variable $ z_{n-2} $ using $ h_1 = 0 $ without affecting the equations $ h_2 = h_3 = 0 $. 
	A similar argument applies for $ h_3 $ as the term $ P_3 (u_1, u_2, u_5 ) $ leads to a new local variable replacing one of $ \{ u_1, u_2, u_5 \} $
	and since none of the three appears in $ h_2 $, we may drop $ h_3 = 0 $ without any effects on the classification problem. 
	We are left with
	$ h_2 = 
	u_3 u_4  - P_{n-3} (z_1,\ldots, z_{n-3}) - \lambda_{n-2}(\eta) = 0 $
	and using Proposition~\ref{Prop:ContSing}, we see that this is a hypersurface singularity of type $ A_1 $.
	The assertion follows. 
	
	Notice that all coefficients in $ h_1, h_2, h_3 $ are contained in $ \{-1,0,1\} $ in the singular case. 
	Therefore, the considerations in the relative setting already provide the singular locus. 
\end{proof}

\subsection{Type $ \boldsymbol{F_4, G_2} $}

Finally, we come to the two missing finite cluster type cases $ F_4 $ and $ G_2 $. 

For type $ F_4 $, 
we choose the initial labeled seed $ \Sigma = (\x, B) $ 
with exchange matrix 
\[
	B = 
	\begin{pmatrix} 
	0 & 1 \\ 
	-1 & 0 &1 \\ 
	& -2 & 0 & 1\\  
	& & -1 & 0 
	\end{pmatrix} 
	\ ,
\]
as in 
Theorem~\ref{Thm:FinTypClass}.
This leads to the following presentation for $ \cA^\prin_\c(F_4) := \cA^\prin_\c(\Sigma) $ (using~Theorem~\ref{Thm:Presentation})
	\begin{equation}
\label{eq:F4_exchange} 
\cA^\prin_\c(F_4) \cong
\KK_\c [\x,  \y]
\hspace{-0.15cm}\raisebox{-0.2cm}{$\bigg/$}
\hspace{-0.2cm}\raisebox{-0.5cm}%
{$\bigg\langle
	\begin{minipage}{6.5cm}
	$ 
	\begin{array}{ll}
	x_1 y_1 - c_1 -  x_2
	\, 
	, 
	&
	x_2 y_2 - c_2 x_1 -  x_3^2 
	\,
	, 
	\\[2pt]
	x_3 y_3 - c_3 x_2 - x_4 
	\,
	, 
	&
	x_4 y_4 - c_4 x_3 -  1
	\end{array} 
	$
	\end{minipage}
	\, \bigg\rangle
	\ ,
	$} 
\end{equation} 	
where $ \c = (c_1, \ldots, c_4) , \x = (x_1, \ldots, x_4), \y = (y_1, \ldots, y_4 ) $. 

\begin{lemma}
	\label{Lem:F_4} 
	The cluster algebra 
	$ \cA^\prin_\c(F_4) $ is isomorphic to a trivial family over $ \KK_\c $,
	where each fiber is isomorphic to the corresponding cluster algebra
	$ \cA(F_4) $
	with trivial coefficients. 
\end{lemma}

\begin{proof}
	We introduce 
	\[
	\begin{array}{llll}
		\widetilde x_1 := c_2 c_4^2 x_1
		\, ,
		&
		\widetilde x_2 := c_1^{-1} x_2 
		\, ,
		&
		\widetilde x_3 := c_4 x_3 
		\, ,
		&
		\widetilde x_4 := c_1^{-1} c_3^{-1} x_4
		\, , 
		\\[2pt]
		
		\widetilde y_1 := c_1^{-1} c_2^{-1} c_4^{-2} y_1
		\, ,
		\ \ \ 
		&
		\widetilde y_2 := c_1 c_4^2 y_2 
		\, ,
		\ \ \ 
		&
		\widetilde y_3 := c_1^{-1} c_3^{-1} c_4^{-1} y_3 
		\, ,
		\ \ \ 
		&
		\widetilde y_4 := c_1 c_3 y_4
		\, .
	\end{array} 
	\]
	By substitution, 
	we get 
	\[
		\begin{array}{l}
		x_1 y_1 - c_1 -  x_2
		\, 
		= \,
		c_1 ( \widetilde x_1 \widetilde  y_1 - 1 -  \widetilde  x_2 )
		\, , 
		\\[5pt]
		x_2 y_2 - c_2 x_1 -  x_3^2 
		\,
		=
		\,
		c_4^{-2}
		( \widetilde x_2 \widetilde y_2 - \widetilde x_1 -  \widetilde x_3^2 )
		\, ,  
		\\[5pt]
		x_3 y_3 - c_3 x_2 - x_4 
		\,
		=
		\,
		c_1 c_3
		( \widetilde x_3 \widetilde y_3 - \widetilde x_2 - \widetilde x_4)
		\, ,
		\\[5pt]
		x_4 y_4 - c_4 x_3 -  1
		\,
		=
		\,
		 \widetilde x_4 \widetilde y_4 - \widetilde x_3 -  1
		\, .
		\end{array} 
	\]
	Hence, up to multiplication by invertible elements,
	the relations are the same as for 
	$ \cA(F_4) $.
	The assertion follows. 
\end{proof}

Proposition~\ref{Prop:F} is now an immediate consequence of Lemma~\ref{Lem:F_4} 
and \cite[Theorem~A(7), resp.~Lemma~5.7]{BFMS}.

\medskip 

Let us turn to 
the finite cluster type $ G_2 $.
By Theorem~\ref{Thm:FinTypClass}, we take as the corresponding exchange matrix 
\[
	B = 
	\begin{pmatrix} 
	0 & 1 \\ 
	-3 & 0 
	\end{pmatrix}
\]
and Theorem~\ref{Thm:Presentation} provides the presentation
\[
	 \cA^\prin_\c(G_2)
	 \cong
	 \KK_{(c_1,c_2)}[x_1, x_2, y_1, y_2]/
	 \langle 
	 \, 
	 x_1 y_1 - c_1 - x_2^3
	 \, 
	 , 
	 \
	 x_2 y_2 - c_2 x_1 - 1
	 \,
	 \rangle  
	 \ . 
\]

\begin{prop}[Proposition \ref{Prop:G}]
	\label{Prop:G_prin}
	There is an isomorphism 
	\[	
	\cA^\prin_\c(G_2)
	\cong
	\KK_{(c_1,c_2)}[x, y, z]/
	\langle 
	\, 
	xyz -
	y -  c_1 - x^3
	\,
	\rangle  
	\ .
	\] 
		A fiber above $ \eta \bsOUT{= (\eta_1, \eta_2)} \in S $ of the family defined by 
		$ \phi \colon \Spec(\cA_\c^\prin (G_2))\to S $ 
	is singular if and only if $ \car(\KK) = 3 $ and $\eta_1 \in \bs{\kappa(\eta)}^3 $ is a cubic element. 
		In the singular case, $ \Spec(\cA_\c^\prin (G_2))_\eta $  is isomorphic to a hypersurface with an isolated singularity of type $ A_2 $ at a closed point.
\end{prop}

\begin{proof}
	Defining $ \widetilde x_1 := c_2 x_1 $, 
	$ \widetilde y_1 := c_2^{-1} y_1 $, and substituting $ \widetilde x_1 = x_2 y_2 - 1 $ in the first relation
	provides the isomorphism
	\[
	\cA^\prin_\c(G_2)
	\cong
	\KK_{(c_1,c_2)}[x_2,  \widetilde y_1, y_2]/
	\langle 
	\, 
	x_2 \widetilde y_1 y_2 -
	 \widetilde y_1 -  c_1 - x_2^3
	\,
	\rangle  
	\ .
	\] 
	Consider the fiber of $ \phi $ above $ \eta \bsOUT{= (\eta_1, \eta_2)} \in S $.
	Applying the Jacobian criterion to the hypersurface $ X := V(x_2 \widetilde y_1 y_2 -
	\widetilde y_1 -  \eta_1 - x_2^3) $
	leads to 
	\[
		\Sing(X) = V (\, \widetilde y_1, \, x_2 y_2 - 1 , \, 3x_2 \, ) \cap X 
		= 
		V (\, \widetilde y_1, \, x_2 y_2 - 1 , \, 3x_2, \, x_2^3 + \eta_1  \, )
		\, .
	\]
	If $ \car(\KK) \neq 3 $, 
	then $ \Sing(X) = \varnothing $.	
	
	On the other hand, if $ \car(\KK) = 3 $,
	then we get 
	$ 
	\Sing(X) = 
	V (\, \widetilde y_1, \, x_2 y_2 - 1 , \, x_2^3 + \eta_1  \, ) $.
	Note that the latter is empty if $ \eta_1 \notin \bs{\kappa(\eta)}^3 $. 
	\\
	Suppose there is $ \delta_1 \in \bs{\kappa(\eta)} $ such that $\delta_1^3 = \eta_1  $.
	We get
	$ 
	\Sing(X) = 
	V (\, \widetilde y_1, \,  y_2 + \delta_1^{-1} , \, x_2 + \delta_1  \, ) $	
	which is a closed point, and in particular, regular. 
	Locally at the singular point, we may introduce the local variables 
	$ u_2 := x_2 y_2 - 1 $, and $ w_2 := x_2 + \delta_1 $ leads to 
	\[
		x_2 \widetilde y_1 y_2 -
		\widetilde y_1 -  \eta_1 - x_2^3
		= \widetilde y_1 u_2 - w_2^3
		\ . 
	\]
	In other words, 
	if $ \car(\KK) = 3 $ and $ \eta_1 \in \bs{\kappa(\eta)}^3 $,
	then $ \Spec(\cA_\c^\prin (G_2))_\eta $ has an isolated singularity of type $ A_2 $ at a closed point.
	\\
	Observe that the coefficients do not play a role in the computations and thus it is not necessary to consider a $ p $-basis of $ \bs{\kappa(\eta)} $.
	The result follows.  
\end{proof}

\section{Cluster algebras of rank two}
\label{Sec:Rk2}

We end by providing a glimpse into the case of cluster algebras which are not necessarily of finite cluster type.
We focus on cluster algebras of rank two over algebraically closed fields.

Recall that we impose the additional hypothesis that the base field $ \KK  $ is algebraically closed  in Theorem~\ref{Thm:rank2} (resp.~in Theorem~\ref{Thm:rk2_nonzero} below)
in order to avoid heavily technical statements. 
Nonetheless, we briefly address the differences that need to be taken into account for an arbitrary field in Remark~\ref{Rk:rk2_arbi_field}.

\subsection{Setup}
\label{Subsec:Rk2} 
The object of this section are cluster algebras of rank two with principal coefficients. 
Therefore, we consider an initial labeled seed $ \Sigma_2 = (\x, B) $ with an exchange matrix of the form 
\[
	B = \begin{pmatrix}
		0 & a \\ b & 0 
	\end{pmatrix}
\ ,
\]
where $ a, b \in \ZZ $ are integers such that either $ ab < 0 $ or $ a = b = 0 $  (cf.~\cite[Section~3.2]{FWZ2016}). 

By Theorem~\ref{Thm:Presentation}, the corresponding cluster algebra with principal coefficients $ \cA_\c^\prin ( \Sigma_2 ) $ has the presentation
\[ 
\cA_\c^\prin (\Sigma_2)
	\cong \KK_{(c_1, c_2)} [x_1, x_2,y_1, y_2] / \langle \, x_1 y_1 - c_1  - x_2^b  , \,  x_2 y_2 - c_2 x_1^a  - 1 \,  \rangle  \ , 
\]
resp.
\[ 
\cA_\c^\prin (\Sigma_2)
\cong \KK_{(c_1, c_2)} [x_1, x_2,y_1, y_2] / \langle \, x_1 y_1 - c_1 x_2^b - 1  , \,  x_2 y_2 - c_2 - x_1^a \,  \rangle  \ , 
\]
depending on the sign of $ a $ and $ b $.
Of course, $ \KK $ can be any field at this stage.  
\\
Since we always assume that $ c_1, c_2 $ are invertible, we can deduce the following presentation 
(clearly, abusing notation)
\begin{equation}
	\label{eq:rk2}	
	\cA_\c^\prin (\Sigma_2)
	\cong \KK_{(c_1, c_2)} [x_1, x_2,y_1, y_2] / \langle \, x_1 y_1 - c_1 - x_2^b   , \,  x_2 y_2 - c_2 - x_1^a \,  \rangle  \ .
\end{equation}

\begin{Bem}
	\label{Rk:non-polynomial}
	In the previous section, we implicitly classified the singularities of rank two cluster algebras of finite cluster type. 
	Hence, the new part of this section is on the non-finite cluster type. 
	\\
	Principal coefficients and universal coefficients are less similar in the non-finite cluster type cases. 
	For the universal coefficients (Remark~\ref{Rk:Univ}), 
	we have to determine all $ g $-vectors of the cluster algebra. 
	Since we have infinitely many cluster variables, 
	there are infinitely many $ g $-vectors. 
	Therefore, the cluster algebra with universal coefficients is determined by infinitely many polynomial equations in infinitely  many variables.
	In contrast to this, we have for principal and generic coefficients always polynomial exchange relations.  
	We also refer to \cite{Read} for a description of the cluster algebra with universal coefficients. 
\end{Bem}

\subsection{Theorem~\ref{Thm:rank2} and its proof}
We keep working with the description \eqref{eq:rk2}. 
Since we have only two exchange relations, we do not perform simplification steps as before,
even though it would be possible if $ a = 1 $ or $ b = 1 $.  

The geometry of $ \Spec ( \cA_\c^\prin (\Sigma_2 ))_\eta $ is rather simple from the perspective of singularities if $ a = b = 0 $. 
\bs{Depending on the images of $ c_1+1 $ and $ c_2+1$ in $ \kappa(\eta) $ the fiber is either regular, singular of type $ A_1 $, or singular of type $ A_1 \times A_1 $.}

\begin{lemma}
	\label{Lem:rk2_zero}
	Let $ \KK $ be an arbitrary field (not necessarily algebraically closed).
		Consider the fiber above $ \eta = (\eta_1, \eta_2) \in S $ of the family determined by $ \phi \colon \Spec (\cA_\c^\prin (\Sigma_2)) \to S $. 
	Using the notation of Section~\ref{Subsec:Rk2}, assume $ a = b = 0 $. 
	We have 
	\begin{enumerate}
		\item 
		If $ \eta_1 = \eta_2 = - 1 $, then 
		$ \Spec ( \cA_\c^\prin (\Sigma_2 ))_\eta $ is the union of four coordinate planes. 
		We obtain a desingularization by first blowing up the intersection of all four planes (which is the origin) followed by blowing up the strict transforms of the pairwise intersections of the planes. 
		
		\item 
		If $ \eta_1 = -1 $ and $ \eta_2 \neq - 1 $, or if $ \eta_1 \neq -1 $ and $ \eta_2 = - 1 $,
		then 
		$ \Spec ( \cA_\c^\prin (\Sigma_2 ))_\eta $ is the union of two regular surfaces intersecting in a regular curve $ C $. 
		Blowing up $ C $ separates the two surfaces and hence resolves the singularities. 
		
		\item 
		If $ \eta_1 \neq -1 $ and $ \eta_2 \neq -1 $,
		then $ \Spec ( \cA_\c^\prin (\Sigma_2 ))_\eta $ is a regular irreducible surface. 
	\end{enumerate}
\end{lemma}

\begin{proof}
	Since 
	$ \cA_\c^\prin (\Sigma_2)
	\cong \KK_{(c_1, c_2)} [x_1, x_2,y_1, y_2] / \langle \, x_1 y_1 - c_1 - 1  , \,  x_2 y_2 - c_2 - 1 \,  \rangle  $,
	by \eqref{eq:rk2} for $ a = b = 0 $,  
	the results follows directly. 
\end{proof}

\begin{Bem} 
Let us have a brief look into the deformation theoretic background of Lemma~\ref{Lem:rk2_zero}.
We consider the intersection of the two-dimensional cylinders over the surface family $ \Spec(\KK_{\c} [x_i,y_i]/ \langle x_i y_i - t_i \rangle ) $,
where $ i = 1 $, resp. $ i = 2 $, 
and we abbreviate $ t_i := c_i + 1 $.
\\ 
Looking at the one-dimensional family $ \mathcal{X}_t :=  \Spec ( \KK_t [x,y] / \langle xy - t \rangle ) $, 
we observe that $ \mathcal X_t $ is a regular irreducible surface for every value of $ t $, expect when $ t =  0 $. 
For $ t = 0 $, we get that $ \mathcal{X}_0 $ is the union of the two coordinate axis.
\\
Combining the observations of the last two paragraphs leads to a different perspective on Lemma~\ref{Lem:rk2_zero}. 
\end{Bem}

 \begin{Thm}
 	\label{Thm:rk2_nonzero}
 	Let $ X_\eta := \Spec (\cA_\c^\prin (\Sigma_2))_\eta $ be the fiber above $ \eta = (\eta_1, \eta_2) \in S $ of the family determined by $ \phi \colon \Spec (\cA_\c^\prin (\Sigma_2)) \to S $. 
 	\bs{	Using the notation of Section~\ref{Subsec:Rk2}, assume $ a, b \neq 0 $.}
 	Assume that $ \KK $ is algebraically closed.
 	\begin{enumerate}
 		\item 
 		The singular locus of $ X_\eta  $ consists of two disjoint components (which are not necessarily irreducible and which are possibly empty),
 		$ \Sing (X_\eta) = Y_a \, \sqcup \,  Y_ b $,
 		where	
 		\[  
 		\begin{array}{l} 
 			Y_a :=  V (a, x_1^a + \eta_2 , x_2,  x_1 y_1 - \eta_1, y_2 )
 			\ ,
 			\\[3pt]
 			Y_b :=
 			V (b, x_1, x_2^b + \eta_1, y_1,  x_2 y_2 - \eta_2 ) 
 			\ .
 		\end{array}  
 		\] 
 		\bs{Note that, up to isomorphism, $Y_a$ and $ Y_b$ are independent of $ \eta $.}  		
 		Notice that $ Y_a $ or $ Y_b $ may be empty depending on the characteristic $ p := \car(\KK) $ of $ \KK $. 
 		For instance, we have
 		\[
 		a \not\equiv 0 \mod p  \Longrightarrow Y_a = \varnothing
 		\ \
 		\mbox{ and } 
 		\ \
 		b \not\equiv 0 \mod p  \Longrightarrow 
 		Y_b = \varnothing
 		\ .  
 		\]
 		
 		\item[(2)] 
 		$ \Spec (\cA_\c^\prin (\Sigma_2))  $ is isomorphic to a trivial family over $ S $,
 		where each fiber is isomorphic to the cluster algebra corresponding to $ \Sigma_2 $ with trivial coefficients.
 		Hence, fixing $ \eta = (\eta_1, \eta_2) \in S $, we have  
 		\[
 		\begin{array}{rcl}
 			X_\eta
 			& \cong 
 			& \Spec (\KK [x_1', x_2',y_1', y_2'] / \langle \, x_1' y_1' - 1 - x_2'^b  , \,  x_2' y_2' - 1 - x_1'^a \,  \rangle ) 
 			\\[3pt]
 			& \cong 
 			& \Spec (\KK [w_1, w_2, z_1, z_2] / \langle \, w_1 z_1 - 1 + w_2^b  , \,  w_2 z_2 - 1 + w_1^a \,  \rangle )		
 		\end{array} 
 		\]
 		\item[(3)]
 		Let $ \alpha, \beta \in \ZZ_+ $ be the largest positive integers such that $ \alpha \, | \, a $ and $ \alpha \not\equiv 0 \mod p $,
 		resp.~$ \beta \, | \, b $ and $ \beta \not\equiv 0 \mod p $.
 		We have: 
 		\[  
 		\begin{array}{l}
 			\displaystyle 
 			Y_a \cong \bigcup_{\zeta \in \mu_\alpha (\KK)}  
 			V (a, w_1 - \zeta , w_2, z_1 - \zeta^{-1}, z_2 ) 
 			\ ,
 			\\[5pt]
 			\displaystyle   
 			Y_b \cong
 			\bigcup_{\xi \in \mu_\beta (\KK)}  
 			V (b, w_1, w_2 - \xi , z_1,  z_2 - \xi^{-1} )
 			\ .
 		\end{array}
 		\] 
 		where the disjoint unions range over the $ \alpha $-th (resp.~$ \beta $-th) roots of unity $ \mu_\alpha (\KK) $ (resp.~$ \mu_\beta (\KK)) $ in $ \KK $. 
 		
 		\item[(4)] 
 		If $ \Sing ( X_\eta ) $ is non-empty, then the singularities classify as follows:
 		Let $ Y_{i,j} $ be a connected component of $ Y_i $, with $ i \in \{ a, b \} $.
 		Locally, along $ Y_{i,j} $,
 		the fiber $ X_\eta $ is isomorphic to a two-dimensional hypersurface singularity of type $ A_{p^m-1} $,
 		where $ m := m(i) $ is the positive integer such that $ |a| = \alpha p^m  $ (if $ i = a $), resp.~$ |b| = \beta p^m $ (if $ i = b $).
 	\end{enumerate}
 \end{Thm}

\begin{proof}
	Consider the presentation of \eqref{eq:rk2},
	$ \cA_\c^\prin (\Sigma_2)
	\cong
	\KK_\c [\x,\y]/ \langle g_1, g_2 \rangle $, 
	where we define
	$ g_1 := x_1 y_1 - c_1 - x_2^b  $
	and $ g_2 := x_2 y_2 - c_2 - x_1^a $.
	The Jacobian matrix of the fiber above $ \eta = (\eta_1, \eta_2 ) \in S $ is
	\[
		\Jac_2 := 
		\Jac (g_1, g_2; x_1, x_2, y_1, y_2) = 
		\begin{pmatrix}
			y_1 &  - b x_2^{b-1} & x_1 & 0 
			\\
			- a x_1^{a-1} & y_2 & 0 & x_2  
		\end{pmatrix}
		\ . 
	\]
	Since $ \KK$ is algebraically closed, the vanishing of the $ 2 $-minors of $ \Jac_2 $ leads to
	\[
	\Sing (\cA_\c^\prin (\Sigma_2)_\eta)
	= 
	V(\, x_1 x_2, \, ax_1^a , \, b x_2^b, \, x_1 y_2 , \, x_2 y_1 , \, y_1 y_2 \,  )
	\, \cap \,  \Spec(\cA_\c^\prin (\Sigma_2))_\eta \ .
	\]
	Taking into account that $ \eta_1, \eta_2 $ are invertible, 
	this implies the description of the singular locus in (1). 
	\\
	Due to the additional assumption on $ \KK $,
	there exist $ \gamma_a , \gamma_b \in \KK $ such that $ \gamma_a^a = \eta_2 $ and $ \gamma_b^b = \eta_1 $. 
	By  introducing 
	\[ 
		x_1' := \gamma_a^{-1} x_1
		\ , 
		\ \ 
		x_2' := \gamma_b^{-1} x_2
		\ , 
		\ \ 
		y_1' := \eta_1^{-1} \gamma_a  y_1
		\ ,
		\ \ 
		y_2' := \eta_2^{-1} \gamma_b y_2
		\ , 
	\]
	we obtain that $ g_1 = \eta_1 ( x_1' y_1' - 1 - x_2'^b ) $
	and $ g_2 = \eta_2 (x_2' y_2' - 1 - x_1'^b) $,
	which provides the first isomorphism of (2). 
	The second isomorphism emerges from the coordinate change 
	\[
		w_1 := \rho_a^{-1} x_1'
		\ , 
		\ \ 
		w_2 := \rho_b^{-1} x_2'
		\ , 
		\ \ 
		z_1 := \rho_a  y_1'
		\ ,
		\ \ 
		z_2 := \rho_b y_2'
		\ , 
	\]
	where $ \rho_a, \rho_b \in \KK $ are elements fulfilling $ \rho_a^a = - 1 $ and $ \rho_b^b = - 1 $. 
	\\
	Part (3) is an immediate consequence of (1) and (2), using that $ \KK $ is algebraically closed. 
	For example, $ w_1^a - 1 = (w_1^\alpha - 1 )^{p^m} = ( \prod_{\zeta \in \mu_\alpha (\KK)} (w_1 - \zeta) )^{p^m} $ for $ m \in \ZZ $ such that $ a = \alpha p^m $, where we assume without loss of generality $ a > 0 $. 
	\\
	Finally, let us come to (4). 
	As the situation is analogous for all irreducible components, 
	we consider $ Y := V(a, w_1 - \zeta, w_2 , z_1 - \zeta^{-1}, z_2) $ for a fixed root of unity $ \zeta \in \mu_\alpha (\KK) $.
	Locally at $ Y $, we have that $ w_1 $ is invertible 
	and thus we may take the first equation, 
	$  w_1 z_1 - 1 + w_2^b  = 0 $,
	to eliminate $ z_1 $. 
	Therefore, we are left with the hypersurface 
	$ w_2 z_2 - 1 + w_1^a = 0  $.
	Since we assume that $  Y $ is non-empty, we must have $ a = \alpha p^m $ for some $ m > 0 $ (still assuming $ a > 0 $ for simplicity -- if $ a < 0 $ multiply the equation by the invertible $ w_1^{-a} $ and modify $ z_2 $ appropriately). 
	This implies that
	\[
		w_2 z_2 - 1 + w_1^a  = 
		w_2 z_2 + \epsilon (w_1 - \zeta)^{p^{m}}
		= \epsilon (v_2 z_2 -  v_1^{p^m})
		\ ,  
	\]
	where $ \epsilon := \prod_{\zeta' \in \mu_\alpha(\KK) : \zeta' \neq \zeta } (w_1- \zeta')^{p^{m}} $
	is a unit locally at $ Y $,
	and for the second equality, 
	we introduce the local variables $ v_2 := \epsilon^{-1} w_2 $ and $ v_1 := w_1 - \zeta $.
	In conclusion, (4) follows. 
\end{proof}

 \begin{Bem}
 	\label{Rk:rk2_arbi_field}
 		Similar results as in Theorem~\ref{Thm:rk2_nonzero}~(3) and (4) are true over arbitrary perfect fields and can be seen using the same arguments. 
 		The only difference is that the formulation is getting more technical as one has to determine the factorization of
 		$ x_1^a + c_2 $ and $ x_2^b + c_1 $ over \bsOUT{$ \KK $} \bs{the residue field $ \kappa(\eta) $}. 
 		Since this does not provide significant new insights, we skip this technicality here. 
 		\\
 		For non-perfect fields, the situation is even more complex and further technical distinctions not promising interesting findings need to be made. 
 		Hence, we do not investigate this direction further.
  \end{Bem}
  
\noindent \textbf{Acknowledgments}:
All authors contributed to the conceptualization, methodology, project administration, development and writing of this paper.

\noindent \textbf{Data Access Statement}:
No data are associated with this paper.

\noindent \textbf{Rights Retention Statement}:
For the purpose of open access, the author has applied a Creative Commons Attribution (CC BY) license to any Author Accepted Manuscript version arising from this submission.

\noindent \textbf{Conflicts of interest}: The authors declare none.

\def\cprime{$'$}


\end{document}